\documentclass[a4paper,12pt]{article}

\usepackage[titletoc,toc,title]{appendix}

\usepackage{geometry}%\geometry{top=5cm,bottom=2cm,left=3cm,right=3cm}
\usepackage{amsmath, amsthm, amsfonts, amssymb}
\usepackage{graphicx}
\usepackage{txfonts}

\usepackage{colordvi}
\usepackage{hyperref}
\usepackage[latin1]{inputenc}
\usepackage[english]{babel}
\usepackage{xcolor}
\usepackage{verbatim}
\usepackage[all]{xy}

\usepackage{wrapfig}
\usepackage{lipsum}
\usepackage{boxedminipage}

\usepackage{wasysym}

\usepackage{stackrel}

\theoremstyle{plain}
  \begingroup
        \newtheorem{theorem}[equation]{Theorem}
        \newtheorem{conjecture}[equation]{Conjecture}
        \newtheorem{lemma}[equation]{Lemma}
        \newtheorem{proposition}[equation]{Proposition}
        \newtheorem{corollary}[equation]{Corollary}

        \newtheorem{assumption}[equation]{Assumption}
        
	    \newtheorem{definition}[equation]{Definition}
        \newtheorem{notation}[equation]{Notation}
        
        \newtheorem{sinnadaitalica}[equation]{}
\endgroup

\theoremstyle{definition}
  \begingroup
        \newtheorem{remark}[equation]{Remark}

        \newtheorem{sinnadastandard}[equation]{}

\endgroup

\numberwithin{equation}{section}	
	
\newcommand{\cc}{\mathcal}

\newcommand{\mr}[1]{\buildrel {#1} \over \longrightarrow}

\newcommand{\ml}[1]{\buildrel {#1} \over \longleftarrow}
\newcommand{\Mr}[1]{\buildrel {#1} \over \Longrightarrow}

\newcommand{\dd}{\ar@2{-}[d]}

\newcommand{\adjuntos}[2][]{ \ar@/^1ex/[r]^{#1} \ar@{}[r]|{\bot} \ar@{<-}@/_1ex/[r]_{#2} }

\newcommand{\adjuntosd}[2]{ \ar@{<-}@/_1ex/[d]_{#1} \ar@{}[d]|{\dashv} \ar@/^1ex/[d]^{#2} }

\newcommand{\dosflechasr}[2]{ \ar@<.5ex>[r]^{#1} \ar@<-.5ex>[r]_{#2} } %\ar@{}[r]|{\bot} 

\newcommand{\igu}[2][]{ \llbracket #1 \! = \! #2 \rrbracket }

\newcommand{\df}[3]{ \ar@/^1ex/[#1]^{#2} \ar@{<-}@/_1ex/[#1]_{#3} } %\ar@{}[r]|{\bot} 
\newcommand{\dfbis}[3]{ \ar@/_1ex/[#1]_{#2} \ar@{<-}@/^1ex/[#1]^{#3} } 

\newcommand{\op}[1]
          {
           \ar@{-}[ld] 
           \ar@{-}[rd] 
           \ar@{}[d]|{#1}  
          }

\newcommand{\cl}[1]
          { 
           \ar@{-}[ur] 
           \ar@{}[u]|{#1} 
           \ar@{-}[ul] 
          }

\newcommand{\ig}[1]
          {
           \ \ \ \ar@{}[d]|{\stackrel{#1}{=}}
          }

\newcommand{\X}{\ell X}

\newcommand{\dcell}[1]
          {
           \ar@<4pt>@{-}'+<0pt,-6pt>[d] 
           \ar@<-4pt>@{-}'+<0pt,-6pt>[d]
           \ar@{}[d]|{#1}
          }

\newcommand{\dcellb}[1]
          {
           \ar@<5pt>@{-}'+<0pt,-6pt>[d] 
           \ar@<-5pt>@{-}'+<0pt,-6pt>[d]
           \ar@{}[d]|{#1}
          }

\newcommand{\dcellbb}[1]
          {
           \ar@<6pt>@{-}'+<0pt,-6pt>[d] 
           \ar@<-6pt>@{-}'+<0pt,-6pt>[d]
           \ar@{}[d]|{#1}
          }

\newcommand{\did}       % identidad (down)%
         {
          \ar@{=}[d]
         }

\newcommand{\pmr}[2]
{
\xymatrix@C=5ex@R=2.4ex
         {
          {} \ar@<1.6ex>[r]^{#1} 
	     \ar@<-1.1ex>[r]^{#2} & {}
         }
}

\newcommand{\pml}[2]
{
\xymatrix@C=5ex@R=2.4ex
         {
            {} 
          & {} \ar@<1.0ex>[l]_{#1} 
	       \ar@<-1.7ex>[l]_{#2}
         }
}

\newcommand{\cellr}[3]
{
\xymatrix@C=7ex@R=2.4ex
         {
          {} \ar@<1.6ex>[r]^{#1} 
          \ar@{}@<-1.3ex>[r]^{\!\! #2 \, \!\Downarrow}
                                         \ar@<-1.1ex>[r]_{#3} & {}
         }
}

\newcommand{\celll}[3]
{
\xymatrix@C=7ex@R=2.4ex
         {
            {} 
          & {} \ar@<1.0ex>[l]^{#1} 
          \ar@{}@<-1.7ex>[l]^{\!\! #2 \, \!\Downarrow}
	                                 \ar@<-1.7ex>[l]_{#3}
         }
}

\newcommand{\dl}    % una raya para poner a la izquierda (inclinada) % 
          {                        
           \ar@<-2pt>@{-}[d]+<4pt,8pt>
          }

\newcommand{\dr}    % una raya para poner a la derecha (inclinada) % 
          {                        
           \ar@<2pt>@{-}[d]+<-4pt,8pt> 
          }
\newcommand{\dc}[1]    % para poner el label en el centro % 
          {                        
           \ar@{}[d]|{#1}  
          }
\newcommand{\dcr}[1]    % para poner el label en el centro a la derecha para tamaÃ±os impares% 
          {                        
           \ar@{}[dr]|{#1}  
          }

\newcommand{\du}[1]  % celda para unidades %
          {
					 \ar@<-2pt>@{-}[d]+<-4pt,8pt> 
           \ar@<3pt>@{-}[d]+<4pt,8pt>
           \ar@{}[d]|{#1}
          }
          
\newcommand{\drho}[1]  % celda para comodulos %
          {
			      \ar@{-}[d] 
           \ar@{-}[lld] 
           \ar@{}[ld]|{#1}  
          }          

\newcommand{\colim}[2] {\displaystyle \lim_{\overrightarrow{#1}} {#2}}

\newcommand{\mono}{\hookrightarrow}
\newcommand{\mmr}[1]{\buildrel {#1} \over \hookrightarrow}

\newcommand{\cqd}{\hfill$\Box$}

\newcommand{\BE}{\begin{equation}}
\newcommand{\EE}{\end{equation}}

	\def \O{\mathcal{O}}
	\def \Sat{\mathcal{S}}
	\def \Cat{\mathcal{C}}
	
	\def \Eat{\mathcal{E}}
	
	\def \Xat{\mathcal{X}}
	
	\def \Vat{\mathcal{V}}
	\def \Ens{\mathcal{E}ns}

	\def \Z{\mathbb{Z}}
	
	\def \be{\begin{enumerate}}
	\def \en{\end{enumerate}}
	\def \eps{\varepsilon}
  \def \-{\textdblhyphenchar}
	\def \Brr {\ar@{}[rr]|*+<.6ex>[o][F]{\scriptscriptstyle{B}}}
	\def \Brrr {\ar@{}[rrr]|*+<.6ex>[o][F]{\scriptscriptstyle{B}}}
	\def \Hrr {\ar@{}[rr]|*+<.6ex>[o][F]{\scriptscriptstyle{H}}}
	\def \Hrrr {\ar@{}[rrr]|*+<.6ex>[o][F]{\scriptscriptstyle{H}}}

\def \de{definition}

\def \prop{proposition}

\begin{document}

\begin{center}
 
\textbf{\large Tannaka Theory over Sup-Lattices}

\vspace{2ex}

PhD thesis\footnote{Under the direction of Eduardo J. Dubuc (University of Buenos Aires, Argentina).} - March 2015.

\vspace{2ex}

Mart\'in Szyld, University of Buenos Aires, Argentina.

\end{center}

\vspace{5ex}

The main result of this thesis is the construction of a \textit{tannakian context} over the category $s \ell$ of sup-lattices, associated with an arbitrary Grothendieck topos, and the attainment of new results in tannakian representation theory from it. 

Although many results were obtained and published historically linking Galois and Tannaka theory (see introduction), these are different and less general since they assume the existence of Galois closures and work on Galois topos rather than on arbitrary topos. Instead we, when talking about Galois theory, mean the extension to arbitrary topos of the article \cite{JT}, critical to get the results of this thesis. 

The tannakian context associated with a Grothendieck topos is obtained through the process of taking relations of its localic cover. Then, through an investigation and exhaustive  comparison of the constructions of the Galois and Tannaka theories, we prove the equivalence of their fundamental recognition theorems (see section \ref{theorems}). 

Since the (bi)categories of relations of a Grothendieck topos were characterized in \cite{CW}, a new recognition-type tannakian theorem (theorem \ref{recognitionforsl}) is obtained, essentially different from those known so far (see introduction).

\pagebreak

\section*{Introduction}

{\bf On Galois Theories. }
In SGA 1, Expos\'e V section 4, ``Conditions axiomatiques d'une theorie de Galois'', (\cite{G2}, see also \cite{DSV}) Grothendieck reinterprets Galois Theory as a theory for functors $\Cat \mr{F} \Ens_{<\infty}$, and by doing so he lays the foundation for many generalizations of this theory. 
In SGA 4 \cite{G1}, Grothendieck himself makes a first generalization of \cite{G2}, by considering functors $\Cat \mr{F} \Ens$ such that the elements $(x,X)$ of the diagram of $F$ with $X$ galoisian is cofinal. This axiom implies that the system consisting of the groups $Aut(X)$ is pro-discrete. 
As we will see in detail, to obtain representation theorems for topoi as instances of Galois Theory, we must drop the pro-discreteness completely.
The article \cite{D2} contains an organized and detailed survey of these generalizations of Galois Theory, including the representation theorems for topoi of \cite{D1},  \cite{JT} and many others. We refer the interested reader to \cite{D2}, and focus now in the theories we will work with, both of which are theories of representation of topoi: Localic Galois Theory as in \cite{D1} and Joyal-Tierney's extension of Galois Theory as in \cite{JT}.

\vspace{1ex}

{\bf Neutral case.}
In \cite{D1}, beginning with a topos $\Eat$ with a point $\Ens \mr{p} \Eat$, $\Eat \mr{p^* = F} \Ens$; a localic group $G = \ell Aut(F)$ and a \emph{lifting} of $F$ $\Eat \mr{\tilde{F}} \beta^G$ into the \emph{classifying topos} of $G$ are constructed, and the following is proved: 

\begin{center}
 \emph{recognition theorem (theorem B in \cite{D1})}: 
$\Eat$ is connected atomic if and only if the lifting is an equivalence, i.e. $\Eat \cong \beta^G$.
\end{center}

We call $\Eat \mr{F} \Ens$ a \emph{(neutral) Galois context}.

\vspace{1ex}

{\bf Non-neutral case.}
In \cite{JT}, VIII.3, beginning with an abritrary Grothendieck topos $\Eat$ (over a base topos $\Sat$), its \emph{spatial cover} $shH \mr{} \Eat$ (with inverse image $\Eat \mr{F} shH$) is constructed and considered (though not explicitly) as a Galois context: a localic groupoid $G$ is built, all the information required for a lifting $\Eat \mr{\tilde{F}} \beta^G$ is present 
and the following is proved: 

\begin{center}
\emph{recognition theorem (theorem VIII.3.2 in \cite{JT})}: the lifting is an equivalence, i.e. $\Eat \cong \beta^G$.
\end{center}

We call $\Eat \mr{F} shH$ a \emph{(non-neutral) Galois context}.

\vspace{2ex}

\noindent {\bf On Tannaka Theories. }
The interpretation of the results of Tannaka \cite{T} as a theory of representations of (affine) $K$-schemas was developed by Saavedra Rivano \cite{Sa}, Deligne \cite{De} and Milne \cite{DM}. 

\vspace{1ex}

{\bf Neutral case.} We begin by considering Joyal-Street's description of (the neutral case of) this theory in \cite{JS} as a theory for functors $\Xat \mr{T} K$-$Vec_{<\infty}$ into the category of finite dimensional $K$-vector spaces. A $K$-coalgebra $L:= End^{\lor}(T)$ and a lifting \hbox{$\Xat \mr{\tilde{T}} Cmd_{<\infty}(L)$} into the category of finite dimensional $L$-comodules are constructed, and the following is proved (\cite{JS}, \S 7 Theorem 3): 

\begin{center}
\emph{recognition theorem}: if $\cc{X}$ is abelian and $F$ is faithful and exact, the lifting is an equivalence. 
\end{center}

We call $\Xat \mr{T} K$-$Vec_{<\infty}$ a \emph{(neutral) tannakian context}.

\vspace{1ex}

{\bf Non-neutral case.} 
Part of the results of \cite{De} (corresponding to the affine non-neutral case, see \cite{De} 6.1, 6.2, 6.8) can also be presented in the following way: given a $K$-algebra $B$ and a functor $\Xat \mr{T} B$-$Mod_{ptf}$ into the category of projective $B$-modules of finite type, a \emph{cog\`ebro\"ide} $L:=L(T)$ \emph{sur} $B$ and a lifting $\Xat \mr{\tilde{T}} Cmd_{ptf}(L)$ into the category of $L$-comodules (called representations of $L$ in \cite{De}) whose subjacent $B$-module is in $B$-$Mod_{ptf}$ are constructed, and the following is proved: 

\begin{center}
\emph{recognition theorem}: if $\cc{X}$ is \emph{tensorielle sur} $K$ (\cite{De} 1.2, 2.1) and $F$ is faithful and exact, the lifting is an equivalence.
\end{center}

We call $\Xat \mr{T} B$-$Mod_{ptf}$ a \emph{(non-neutral) tannakian context}.

\vspace{1ex}

Since then many generalizations of this theory have been made, mainly in two different directions: either relaxing some hyphotesis for $K$ (instead of a field we can consider rings \cite{H}, valuation rings \cite{We}), or considering an arbitrary base monoidal category $\Vat$ instead of $K$-vector spaces (\cite{SP} \cite{McC}, \cite{SC}). Though the constructions of Tannaka theory and some of its results regarding for example the \emph{reconstruction theorem} (see \cite{Day}, \cite{McC}) have been obtained, it should be noted that no proof has been made so far of a satisfactory recognition theorem for an arbitrary base category.

There are also propositions that, under additional properties for the tannakian context, give additional structure to $L$ (see \cite{JS}, \S 8 and \S 9, \cite{De} 6.4 and 6.8). This results are independent of the recognition theorem and can be generalized to an arbitrary base category (see \ref{neutralbialg} and \ref{neutralhopf}, \ref{bialg} and \ref{hopf}).

\vspace{2ex}

{\bf On the relations between both theories. }
Strong similarities are evident to the naked eye, and have been long observed, between different ``versions'' of Galois and Tannaka representation theories.
Various approaches to relate Tannaka and Galois Theory are developed for example in \cite{R} and \cite{JAS}, where the existence of Galois closures (disguised in one form or another) is essential, and which cover Galois topoi but not the Joyal-Tierney extension to atomic or arbitrary topoi.

In this thesis, to relate Tannaka and Galois Theory we proceed as follows: from a Galois context as in \cite{D1}, \cite{JT}, we construct an associated Tannaka context over $s \ell$, and by comparing the constructions of both theories in each context we obtain new Tannaka recognition theorems from the Galois recognition theorems.

 As a first, simpler example, consider the neutral version that we develop in section \ref{sec:Neutral}. Here by Galois theory we refer to Localic Galois Theory as developed by Dubuc in \cite{D1}, and by Tannaka theory we refer to the generalization to an arbitrary base monoidal category $\Vat$ of the definitions and constructions of \cite{JS} that we do in appendix \ref{appendix}.
These are the ``strong similarities'': both in Galois and Tannaka theories, from a \emph{context} we construct an object ($G$ or $L$) and a \emph{lifting} into a category of representations of the object. The \emph{recognition theorems} are as follows: the lifting is an equivalence of categories if and only if some conditions on the context are satisfied (like ``if $\cc{X}$ is abelian and $F$ is faithful and exact'' for neutral Tannaka theory over vector spaces, or ``$\Eat$ is connected atomic'' for Localic Galois Theory). Note that $G$ is a group object in a \emph{geometric} category and $L$ a cogroup object in an \emph{algebraic} category, the cogroup structure for $L$ yields a group structure for its formal dual $\overline{L}$.

However, these similarities are just of the ``form'' of both theories, and don't allow us a priori to translate any result from one theory to another, in particular Localic Galois Theory and neutral Tannaka theory over vector spaces remain independent. But what we can do is find the tannakian context corresponding to a Galois context, and we do this by taking relations: from $\Eat \mr{F} \Ens$ we construct $\Xat \mr{T} \Vat := Rel(\Eat) \mr{Rel(F)} Rel \mmr{} s \ell$ and we prove the following compatibilities.
\begin{enumerate}
 \item The objects constructed from both contexts are isomorphic as localic groups \linebreak (\hbox{$G = \overline{L}$}, i.e. $\mathcal{O}(G) = L$ where $\mathcal{O}(G)$ is the locale corresponding to the space $G$).
 \item For any localic group $G$, the categories of representations $Cmd_0(\mathcal{O}(G))$ and $Rel(\beta^G)$ are equivalent.
\end{enumerate}

With these compatibilities, we can complete the following diagram that relates Galois theory to a neutral tannakian theory over the base category $s \ell$ of sup-lattices:

$$\xymatrix
        {
          \beta^G            \ar[r] \ar[rdd]
        & \cc{R}el(\beta^G)  \ar[r]^{=}
        & Cmd_0(\mathcal{O}(G))          \ar[r]^{=}
        & Cmd_0(L)          \ar[ldd]
        \\
        & \cc{E}             \ar[d]^F \ar[r] \ar[lu]_{\widetilde{F}}
        & \cc{R}el(\cc{E})   \ar[ul]_{\cc{R}el(\widetilde{F})}  
                             \ar[d]^T \ar[ur]^{\widetilde{T}}
        \\
        & \cc{E}ns            \ar[r]
        & \cc{R}el 
        &  \hspace{-8ex} = s \ell_0 \subset s \ell.
       }$$

We obtain immediately that the Tannaka lifting functor $\widetilde{T}$ is an equivalence if and only $\widetilde{F}$ is so \mbox{(Theorem \ref{neutralAA}).} Then, from the fundamental theorem of localic Galois theory (Theorem \ref{fundamentalLGT}, Theorem B of \cite{D1}), we obtain the following Tannaka recognition theorem for the (neutral) Tannaka context $\Xat \mr{T} \Vat = Rel(\Eat) \mr{Rel(F)} Rel \mmr{} s \ell$ associated to a pointed topos: $\widetilde{T}$ is an equivalence if and only if the topos is connected atomic (Theorem \ref{BB}). These topoi are then a new concrete example where the recognition theorem holds which is completely different than the other cases in which the Tannaka recognition theorem is known to hold, where the unit of the tensor product is an object of finite presentation. Simultaneously, the non atomic pointed topoi furnish examples where the lifting is not an equivalence, i.e. the categories of relations of non atomic pointed topoi are not neutral tannakian categories (we will 
show 
later 
that they are non-neutral tannakian categories). 

It should be noted that the properties of the Tannaka context equivalent to the lifting being an equivalence are expressed in terms of the topos, i.e. of the Galois context. This is not exactly what one would expect from a Tannaka recognition theorem, but we were able to solve this by developing the more general non-neutral case (see theorem \ref{recognitionforsl}). The results described so far, and developed with detail in section \ref{sec:Neutral}, were obtained in the first years of our doctoral career and published as \cite{DSz}.

\vspace{2ex}

{\bf The non-neutral case and the new Tannaka recognition theorem for $s \ell$.}
As we have mentioned before, both Localic Galois Theory and neutral Tannaka Theory admit generalizations that we will refer to as Galois Theory (as developed by Joyal-Tierney in \cite{JT}) and non-neutral Tannaka Theory (as developed by Deligne in \cite{De}, whose constructions we make in an arbitrary base monoidal category in appendix \ref{sec:Tannaka}).

The jump in generality from pointed atomic to arbitrary topoi is the jump from groups to (maybe pointless) groupoids, and corresponds exactly to the jump from neutral to non-neutral Tannaka theory. Following the constructions of \cite{De}, from a monoid $B$ in a monoidal base category $\Vat$ and a monoidal functor $\Xat \mr{T} B$-$Mod$ (satisfying some duality conditions, see \ref{sec:Tannaka}), we construct a cog\`ebro\"ide in $\Vat$, $L:= End^{\lor}(T)$, and a \emph{lifting} $\Xat \mr{\tilde{T}} Cmd_0(L)$ into a full subcategory of the category of $L$-comodules. We call $\Xat \mr{T} \Vat$ a \emph{(non-neutral) tannakian context}. In \cite{De} it is shown that in the case of schemas, which is deduced from the affine case that corresponds to vector spaces, if $T$ is faithful and exact (i.e. a \emph{fiber functor}, see \cite{De}, 1.9) the lifting is an equivalence (recognition theorem, \cite{De}, 1.12), but, as for the neutral case, no recognition theorem has been proved so far for arbitrary base 
categories.

The same ``similarities'' of the neutral case appear here, and we exploit them by constructing the non-neutral tannakian context over $s \ell$ associated to the Galois context given by the spatial cover. This generalization, corresponding to the jump in generality from the pointed over $\cc{E}ns$ case of \cite{D1} to the unpointed general case of \cite{JT} is by no means direct and conforms sections \ref{relintopos} to \ref{theorems}. 
Once again we do this by taking relations, so in a way it is a similar process to the one of section \ref{sec:Neutral}, but it is instructive to examine the differences between the neutral (from \cite{D1}) and non-neutral (from \cite{JT}) cases, while we describe the contents of these sections:

\begin{itemize}
 \item Since Joyal-Tierney in \cite{JT} work over an arbitrary base topos, we have to consider $\ell$-relations in an arbitrary topos. These are arrows $X \times Y \mr{} G$, where $G$ is a \linebreak sup-lattice. We do this in section \ref{relintopos}. We begin by proving results for relations (i.e. where $G = \Omega$), some of which are already known, but with different definitions and proofs that are easier to extend to $\ell$-relations in section \ref{sec:ellrel}. In particular we consider the axioms that make a relation a function (univalued, everywhere defined) and show that functions correspond to (the graphs of) arrows of the topos. We will show the corresponding result for $\ell$-relations in section \ref{sub:EshP}.
 
 \item An analysis of $\rhd$- and $\lozenge$-diagrams and cones (see \ref{diagramadiamante12}, \ref{defdeconos}) is needed to show the equivalence between the universal properties defining $G$ and $L$. In section \ref{diagrams} and \ref{sec:cones} we establish the results needed for this. The following phenomena is worth mentioning: since the locales are commutative algebras, a Galois context yields a non-neutral commutative tannakian context. But a non-neutral commutative tannakian context is in a sense ``neutralized'' over the base category of $B$-bimodules. Instead of a fiber functor, we now have two fiber functors corresponding to the two inclusions $B \mr{} B \otimes B$ (see \ref{tannakacontext}). We develop in section \ref{sec:cones} a theory of cones for two different functors that we will use in section \ref{sec:Contexts} to exploit this fact.
 
 We show in section \ref{sec:cones} that $\rhd$-cones of functions correspond to natural transformations, and analyse their behavior through topoi morphisms. This will allow us to express the property defining the localic groupoid considered in \cite{JT}, VIII.3 Theorem 2 p.68 as a universal property of $\rhd$-cones (theorem \ref{JTGeneral}).
 
 We show that cones defined over a site of a topos can be extended uniquely to the topos (preserving its properties). Since we will consider a tannakian coend that is a universal $\lozenge$-cone over $Rel(\cc{E})$, this will allow us to solve size problems when constructing the coend by considering a small site of the topos $\cc{E}$.
 
 \item Section \ref{sub:EshP} is the most technical section of this thesis and it is devoted essentially to giving external characterizations, for a locale $P$ in a topos $\cc{S}$, of the developments of section \ref{relintopos} when considered internally in the topos $shP$. 
  
 Recall that Joyal and Tierney develop in \cite{JT}, VI a \emph{change of base} for sup-lattices and locales. In particular for a locale $P$ in a topos $\cc{S}$ they characterize sup-lattices and locales in the topos $\cc{E} = shP$ by showing that $s \ell(shP) \mr{\gamma_*} P$-$Mod$ is an equivalence that restricts as an equivalence $Loc(shP) \mr{\gamma_*} P$-$Loc$. (see \ref{enumeratedeJT}). Also, they characterize etale spaces as those spaces whose corresponding locale is of the form $\gamma_*(\Omega_P^X) = \gamma_*(\cc{O}(X_{dis}))$, with $X \in shP$, where $\Omega_P$ is the subobject classifier of $shP$. We denote $X_d := \gamma_*(\Omega_P^X) = \Omega_P^X(1)$.
 
 We develop in section \ref{sub:EshP} what may be called a \emph{change of base} for relations, given a locale $P \in \Sat$ we examine the correspondence between relations $\gamma^* X \times \gamma^* Y \mr{} \Omega_P$ in the topos $shP$ and arrows (that we call $\ell$-relations) $X \times Y \mr{} P$ in the base topos. 
 
 Then we consider $\ell$-relations $X \times Y \mr{\lambda} G$ in the topos $shP$, we show that they correspond to $P$-module morphisms $X_d \otimes_P Y_d \mr{\mu} G(1)$ and give external (i.e. in terms of $\mu$, in the base topos $\Sat$) formulae equivalent to the axioms of section \ref{sec:ellrel}. We also ``externalize'' the formulae of the duality of $\Omega_P^X$ in $s \ell(shP)$. All this is neccessary to treat the general unpointed case of Galois theory in the following sections.
 
  \item In section \ref{sec:Cmd=Rel} we establish the equivalence between discrete actions of a localic groupoid and discrete comodules of its subjacent cog\`ebro\"ide, and between comodule morphisms and relations in the category of discrete actions, generalizing the results from section \ref{Cmd_0}. The definition of action of \cite{JT} is a priori different to the one of \cite{D1}, so we had to show that in the discrete case they coincide in order to generalize our previous results (see remarks \ref{coinciden1}, \ref{coinciden2}).
 
 \item In section \ref{sec:Contexts} we show explicitly how the localic groupoid $G$ constructed in \cite{JT}, VIII.3 from the spatial cover $shH \mr{} \Eat$ is a universal $\rhd$-cone of $\ell$-bijections in the topos $sh(H \otimes H)$ for two different functors, as mentioned before. This was not neccessary when working with the neutral Galois context of \cite{D1}, since $G$ is constructed there precisely as a universal $\rhd$-cone of $\ell$-bijections, but Joyal and Tierney deduce the existence of $G$ with a different technique, so this result of theorem \ref{JTGeneral} is crucial for us in order to prove $G=L$. This theorem is also interesting by itself, since it shows a different way in which we can interpret (and construct) the fundamental groupoid $G$. After showing this, the previous properties of $\rhd$ and $\lozenge$ diagrams and cones let us show the isomorphism $G=L$.

 \item Finally, in section \ref{theorems}, we combine all the previous results to obtain a new non-neutral Tannaka recognition type theorem over for a type of $s \ell$-enriched categories (over a base topos) that are called distributive categories of relations (DCRs), that generalize the categories of relations $Rel(\cc{E})$ of topoi. We begin by proving an analogous extension property to the one of section \ref{sec:cones} for DCRs, that will allow us to construct the tannaka coend $L$ for bounded DCRs. 

 We use previous results from \cite{Pitts}, \cite{CW} that relate DCRs with Grothendieck topoi to construct a tannakian fiber functor $\cc{A} \mr{} B$-$Mod$. 
 In this way, we obtain the following new Tannaka recognition theorem (for the base category $s \ell$ of sup-lattices),

\begin{center}
 \emph{recognition theorem, theorem \ref{recognitionforsl}}: A bounded DCR $\cc{A}$ is complete if and only if the lifting $\cc{A} \mr{} Cmd_0(L)$ of its fiber functor is an equivalence.
\end{center}

\end{itemize}

\pagebreak

\tableofcontents

\pagebreak

\section{The neutral case over $\Ens$} \label{sec:Neutral}

In this section we construct an explicit (neutral) Tannakian context for the Galois theory of atomic topoi, and prove the equivalence between its fundamental theorems. Since the theorem is known for the Galois context, this yields, in particular, a proof of the fundamental (recognition) theorem for a new Tannakian context. This example is different from the additive cases \cite{JS}, \cite{H}, \cite{C}, or their generalization \cite{SC}, where the theorem is known to hold, and where the unit of the tensor product is always an object of finite presentation (that is, filtered colimits in the tensor category are constructed as in the category of sets), which is not the case in our context.

\vspace{1ex}

In this section by Galois theory we mean Grothendieck's Galois theory of progroups (or prodiscrete localic groups) and Galois topoi \cite{G1}, \cite{G2}, as extended by 
Joyal-Tierney in \cite{JT}. More precisely, the particular case of arbitrary localic groups and pointed atomic topoi. 

For the Galois theory of atomic topoi we follow Dubuc \cite{D1}, where he develops \emph{localic Galois theory} and makes an explicit construction of the localic group of automorphisms $Aut(F)$ of a set-valued functor  $\cc{E} \mr{F} \cc{E}ns$, and of a lifting $\cc{E} \mr{\widetilde{F}} \beta^{Aut(F)}$ into the topos of sets furnished with an action of the localic group (see \ref{Aut(F)}). He proves in an elementary way\footnote{meaning, without recourse to change of base and other sophisticated tools of topos theory over an arbitrary base topos.} 
that when $F$ is the inverse image of a point of an atomic topos, this lifting is an equivalence \mbox{\cite[Theorem 8.3]{D1},} which is Joyal-Tierney's theorem \cite[Theorem 1]{JT}. 

For Tannaka theory we follow Joyal-Street \cite{JS} (for the original sources see the references therein). The construction of the Hopf algebra $End^\vee(T)$ of endomorphisms of a finite dimensional vector space valued functor $T$ can be developed for a $\cc{V}_0$-valued functor, \mbox{$\cc{X} \mr{T} \cc{V}_0 \subset \cc{V}$,} where $\cc{V}$ is any cocomplete monoidal closed category, and $\cc{V}_0$ a (small) full subcategory of objects with duals, see for example \cite{P}, \cite{SC}, \cite{SP}. There is a lifting 
\mbox{$\cc{X} \mr{\widetilde{T}} Cmd_0(End^\vee(T))$} into the category of \mbox{$End^\vee(T)$-comodules} with underlying object in $\cc{V}_0$. For a handy reference and terminology see \mbox{section \ref{appendix}.} In \cite{JS}, \cite{SP} it is shown that in the case of vector spaces, if 
$\cc{X}$ is abelian and $F$ is faithful and exact, the lifting is an equivalence (recognition theorem). 

\vspace{1ex}

Recall that given a regular category $\Cat$ we can consider the category $\cc{R}el(\Cat)$ of relations in $\Cat$. There is a faithful functor (the identity on objects) \mbox{$\Cat \rightarrow \cc{R}el(\Cat)$,} and any regular functor $\cc{C} \mr{F} \cc{D}$ has an extension 
\mbox{$\cc{R}el(\cc{C}) \mr{\cc{R}el(F)} \cc{R}el(\cc{D})$.}

The category $\cc{R}el = \cc{R}el(\cc{E}ns)$ is a full subcategory of the category $s \ell$ of sup-lattices, set $\cc{R}el = s \ell_0$. This determines the base $\cc{V},\cc{V}_0$ of a Tannaka context. Furthermore, a localic group is the same thing as an Hopf algebra in the category $s \ell$ which is also a locale (see section \ref{background}).

Given any pointed topos with inverse image $\cc{E} \mr{F} \cc{E}ns$ of a Galois context, we associate a (neutral) Tannakian context as follows:
$$
\xymatrix
       {
	        \beta^G \ar[rd]_{} 
	      & \Eat \ar[l]_{\widetilde{F}} \ar[d]^{F} \ar[r] 
	      & Rel(\Eat) \ar[d]_{T} \ar[r]^{\widetilde{T}} 
	      & Cmd_0(H) \ar[dl]^{} 
	     \\
		  & \Ens \ar[r]
		  & \cc{R}el = s \ell_0,
		 }
$$
where $G = Aut(F)$, $H = End^\vee(T)$, and $T = \cc{R}el(F)$.

\vspace{1ex}

We prove that \emph{$\widetilde{F}$ is an equivalence if and only if 
$\widetilde{T}$ is so} \mbox{(Theorem \ref{neutralAA}).}  The result is based in two theorems. First, we prove that for any localic group $G$, there is an isomorphism of categories $\cc{R}el(\beta^G) \cong Cmd_0(G)$ 
(Theorem \ref{neutralComd=Rel}). Second, we prove that the Hopf algebra $End^\vee(T)$ is a locale, and that there is an isomorphism \linebreak $Aut(F) \cong End^\vee(T)$ (Theorem \ref{neutralG=H}).

In particular, from Theorem \ref{neutralAA} and the fundamental theorem of localic Galois theory (Theorem \ref{fundamentalLGT}), we obtain that the following \emph{Tannaka recognition theorem} holds in the (neutral) Tannaka context associated to a pointed topos: \emph{$\widetilde{T}$ is an equivalence if and only if the topos is connected atomic (Theorem \ref{BB}).} 

\pagebreak

\subsection{Background, terminology and notation} \label{background}

We begin by recalling some facts on sup-lattices, locales and
monoidal categories, and by doing so we fix notation and terminology.

We will consider the monoidal category $s \ell$ of \emph{sup-lattices},
whose objects are posets $S$ with arbitrary suprema $\bigvee$ (hence
finite infima $\wedge$, $0$ and $1$) and 
whose arrows are the suprema-preserving-maps. We call these arrows
\emph{linear maps}. We will write $S$ also
for the underlying set of the lattice.  
The \emph{tensor product} of two sup-lattices $S$ and $T$ is the codomain
 of the universal bilinear map $S \times T \mr{} S \otimes T$. Given 
 $(s,\,t) \in S \times T$, we denote the corresponding element in 
 $S \otimes T$ 
by $s \otimes t$. The unit for $\otimes$ is the sup-lattice $2=\{0 \leq 1\}$. The linear map $S \otimes T \stackrel{\psi}{\rightarrow} T \otimes S$ sending 
$s \otimes t \mapsto t \otimes s$ is a symmetry. 
Recall that, as in any monoidal category, a \emph{duality} between two sup-lattices $T$ and $S$ is a pair of arrows $2 \stackrel{\eta}{\rightarrow} T \otimes S$, 
$S \otimes T \stackrel{\varepsilon}{\rightarrow} 2$ satisfying the usual triangular equations (see \ref{neutraltriangular}). We say that $T$ is \emph{right dual} to $S$ and that $S$ is \emph{left dual} to $T$, and denote $T = S^\wedge$, $S = T^\vee$. Note that since $sl$ is symmetric as a monoidal category, $S$ has a right dual if and only if it has a left dual, and $S^\wedge = S^\vee$.

There is a free sup-lattice functor $\Ens \mr{\ell} s \ell$. Given  $X \in \cc{E}ns$, $\ell X$ is the power set of $X$, and for    
\mbox{$X \stackrel{f}{\rightarrow} Y$}, $\ell f = f$ is the direct image. This functor extends to a functor $\cc{R}el \mr{\ell} s \ell$, defined on the category $\cc{R}el$ of sets with relations as morphisms. A linear map $\ell X \to \ell Y$ is the ``same thing'' as a relation $R \subset X \times Y$. In this way $\cc{R}el$ can be identified with a full subcategory $\cc{R}el \mmr{\ell} s \ell$. We define $s \ell_0$ as the full subcategory of $s \ell$ of objects of
the form $\ell X$. Thus, abusing notation, $\cc{R}el = s \ell_0 \subset s \ell$ (``$=$'' here is actually an isomorphism of categories).
Recall that $\cc{R}el$ is a monoidal category with tensor product given by the cartesian product of sets (which is not a cartesian product in $\cc{R}el$). It is immediate to check that $\ell X \otimes \ell Y = \ell(X \times Y)$ in a natural way.
\begin{sinnadaitalica} \label{tensoriso}
The functor $\cc{R}el \mmr{\ell} s \ell$ is a tensor functor, and the identification $\cc{R}el = s \ell_0$ is an isomorphism of monoidal categories.
\end{sinnadaitalica}
The arrows
$2 \stackrel{\eta}{\rightarrow} \ell X \otimes \ell X$, \mbox{$\ell X
\otimes \ell X \stackrel{\varepsilon}{\rightarrow} 2$,} defined on the generators as $\eta(1) = \bigvee_x x \otimes x$ and 
\mbox{$\varepsilon (x\otimes y) = \delta_{x=y}$} determine a duality, and in this way the objects
of the form  $\ell X$ have both duals and furthermore they are self-dual, 
\mbox{$(\ell X)^\wedge = (\ell X)^\vee = \ell X$.} Under the isomorphism 
$\cc{R}el = s \ell_0$, $\varepsilon$ and $\eta$ both correspond to the diagonal relation $\Delta \subset X \times X$. Duals are contravariant functors, if $R \subset X \times Y$ is the relation corresponding to a linear map 
\mbox{$\ell X \to \ell Y$,} then the opposite relation $R^{op} \subset Y \times X$ corresponds to the dual map 
$(\ell Y)^\wedge  \to (\ell X)^\wedge$.
\begin{sinnadaitalica} \label{ellX=X}
We will abuse notation (see for example \eqref{neutraltriangleequation}) by omitting to write the functor $\cc{E}ns \mr{\ell} s\ell_0 = \cc{R}el$, i.e. by denoting by $X \mr{f} Y$ the direct image of $f$ which is the relation given by its graph $R_f \subseteq X \times Y$. $R_f^{op}$ is the relation corresponding to the inverse image of $f$, which we will denote by $Y \mr{f^{op}} X$. 
\end{sinnadaitalica}
As in any monoidal category, an \emph{algebra} (or \emph{monoid})
in $s \ell$ is an object $G$ with an associative multiplication $G \otimes G \mr{*} G$ which 
has a unit $u \in G$. If $*$ preserves the
symmetry $\psi$, the algebra is \emph{commutative}. An algebra morphism is a
linear map which preserves $*$ and $u$.

A \emph{locale} is a sup-lattice $G$ where finite infima $\wedge$
distributes over arbitrary suprema $\bigvee$, that is, it is bilinear, and
so induces a multiplication \mbox{$G \otimes G \mr{\wedge} G$.}
A locale morphism is a linear map which preserves \mbox{$\wedge$ and
$1$.} In
this way locales are commutative algebras, and there is a full inclusion
of categories $Loc \subset Alg_{s \ell}$ into the category of commutative algebras in $s \ell$.
\begin{sinnadaitalica} \label{charlocales}
In \cite{JT}, III.1, p.21, Proposition 1, locales are characterized as those commutative algebras such
that $x * x = x$ and $u = 1$.
\end{sinnadaitalica}

A (commutative) \emph{Hopf algebra} in $s \ell$ is a group object in $(Alg_{s \ell})^{op}$. 
A \emph{localic group} (resp. \emph{monoid}) $G$ is a group (resp. monoid) object in
the category $S \! p$ of \emph{localic 
spaces}, which is defined to be the formal dual of the category of locales, $S \! p =
Loc^{op}$. Therefore $G$ can be also considered as a Hopf algebra in
$s \ell$. The unit and the multiplication of $G$ in $S \! p$ are the counit $G \mr{e} 2$ and comultiplication 
$G \mr{w} G \otimes G$ of a coalgebra structure
for $G$ in $Alg_{s \ell}$. The inverse is an antipode 
$G \mr{\iota} G$. Morphisms correspond but change direction, and we actually have a contravariant equality of
categories \mbox{$(Hopf_{Loc})^{op} = Loc$-$Group$}, where $Hopf_{Loc}$ consists of those Hopf algebras in $s\ell$ which happen to be a locale, i.e. which satisfy the conditions of \ref{charlocales}.

 \begin{remark} \label{numeroarriba}
  Throughout this thesis, a number or symbol above an ``$\leq$'' or an ``$=$'' indicates the previous result that justifies the assertion. 
 \end{remark}

\pagebreak

\subsection{Preliminaries on bijections with values in a locale} \label{sec:general}

As usual we view a relation $\lambda$ between two sets $X$ and $Y$ as
a map (i.e. function of sets) $X \times Y \mr{\lambda}2$. We consider maps $X \times Y \mr{\lambda} G$ with values in an arbitrary sup-lattice $G$, that we will call \emph{$\ell$-relations}. Since $\ell(X \times Y) = \ell X \otimes \ell Y$, it follows that $\ell$-relations are the same thing that linear maps $\ell X \otimes \ell Y \mr{\lambda} G$. 
The results of this section are established in order to be used in the 
next sections, and they are needed only in the case $X=Y$.

\begin{sinnadastandard} \label{neutraldiagramadiamante12}
Consider two $\ell$-relations $X \times Y \mr{\lambda} G$, \mbox{$X' \times
Y'\mr{\lambda'}G$,} and two maps $X \mr{f} X'$, $Y \mr{g} Y'$, or, more generally, two spans (which induce relations that we also denote with the same letters), 
$$
\xymatrix@C=3ex@R=1ex
         {
         {} & R \ar[dl]_{p} \ar[dr]^{p'}
         \\
         X & {} & X',
         }
\hspace{3ex}                  
\xymatrix@C=3ex@R=1ex
         {
         {} & S \ar[dl]_{q} \ar[dr]^{q'}
         \\
         Y & {} & Y'\;
         }
\hspace{3ex}
\xymatrix@C=2ex@R=1ex
         {
          {}
          \\
          R = p' \circ p^{op}, \;\; S = q' \circ q^{op}\,,
         }
$$
and a third $\ell$-relation $R \times S \mr{\theta} G$. 

These data give rise to the following diagrams:
\begin{equation} \label{neutraltriangleequation}
\xymatrix@C=1.8ex@R=3ex
        {
         \hspace{-1ex} \lozenge_1 = \lozenge_1(f,g)  
         &&&& \lozenge_2 = \lozenge_2(f,g) 
         &&&& \lozenge = \lozenge(R,S)
        }
\end{equation}
\mbox{
$ \quad
       {
        \xymatrix@C=1.4ex@R=3ex
                 {
                  & X \times Y  \ar[rd]^{\lambda} 
                  \\
			        X \times Y' \ar[rd]_{f \times Y'} 
			                    \ar[ru]^{X \times g^{op}} & \equiv & G\,,
			      \\
			       & X' \times Y' \ar[ru]_{\lambda '} 
			      } 
	    }
$
$ \quad
       {
        \xymatrix@C=1.4ex@R=3ex
                 {
                  & X \times Y  \ar[rd]^{\lambda} 
                  \\
			        X' \times Y \ar[rd]_{X' \times g} 
			                    \ar[ru]^{f^{op} \times Y} & \equiv & G\,,
			      \\
			       & X' \times Y' \ar[ru]_{\lambda '} 
			      } 
	    }
$
$ \quad
       {
        \xymatrix@C=1.4ex@R=3ex
                 {
                  & X \times Y  \ar[rd]^{\lambda} 
                  \\
			        X \times Y' \ar[rd]_{R \times Y'} 
			                    \ar[ru]^{X \times S^{op}} & \equiv & G\,,
			      \\
			       & X' \times Y' \ar[ru]_{\lambda '} 
			      } 
	    }
$
}

\vspace{1ex}

expressing the equations: 

\vspace{1ex}

\begin{center} 
$\hspace{2ex} \lozenge_1:\,  
\lambda'\langle f(a),b' \rangle  \,=\, \displaystyle \bigvee_{g(y)=b'} \lambda\langle a,y \rangle \:,$   
$\hspace{2ex} \lozenge_2:\, 
\lambda'\langle a',g(b) \rangle  \,=\, \displaystyle \bigvee_{f(x)=a'} \lambda\langle x,b \rangle$,
\end{center}

and $\hspace{1ex} \lozenge$:
$\hspace{1ex}  
           \displaystyle \bigvee_{(y,\, b')\in S} \lambda\langle a,y \rangle \;\;      
         = \displaystyle \bigvee_{(a,\, x')\in R} \lambda'\langle x',b' \rangle$.
\end{sinnadastandard}

\begin{remark} \label{diamondesmasfuerteneutral}
It is clear that diagrams $\lozenge_1$ and  $\lozenge_2$ are particular cases of \mbox{diagram $\lozenge$.} Take $R = f,\; S = g$, then $\lozenge_1(f,g) = \lozenge(f, g)$, and  $R = f^{op},\; S = g^{op}$, then  \mbox{$\lozenge_2(f,g) = \lozenge(f^{op}, g^{op})$.} 
\end{remark}

The general $\lozenge$ diagram follows from these two particular cases. 

\begin{proposition} \label{neutralpre2rhdimpliesdiamante}
Let $R$, $S$ be any two spans connected  by an $\ell$-relation $\theta$ as above. If  $\lozenge_1(p',q')$ and $\lozenge_2(p,q)$ hold, then so does 
$\lozenge(R,S)$.
\end{proposition}
\begin{proof}
We use the elevators calculus, see appendix \ref{ascensores} (and recall our remark \ref{numeroarriba} on notation):
$$
\xymatrix@C=-1.5ex
         {
          \\
             X \did & & Y' \dcellb{\!\!S^{op}}
          \\              
             X & & Y
          \\
            & G \cl{\lambda} 
         }
\xymatrix@R=10ex{ \\ \;\;\stackrel{\textcolor{white}{\lozenge_2}}{=}\;\; \\}
\xymatrix@C=-1.5ex
         {
             X \did & & Y' \dcellb{\!\!q'^{op}}
          \\
             X \did & & S  \dcell{q}
          \\
             X & & Y
          \\
           & G \cl{\lambda}
         }
\xymatrix@R=10ex{ \\ \;\;\stackrel{\lozenge_2}{=}\;\; \\}
\xymatrix@C=-1.5ex
         {
             X \did & & Y' \dcellb{\!\!q'^{op}}
          \\
             X \dcellb{\!\!p^{op}} & & S  \did
          \\
             R & & S
          \\
            & G \cl{\theta}
         }
\xymatrix@R=10ex{ \\ \;\;\stackrel{\textcolor{white}{\lozenge_2}}{=}\;\; \\}
\xymatrix@C=-1.5ex
         {
             X \dcellb{\!p^{op}} & & Y' \did
          \\
             R \did  & & Y' \dcellb{\!\!q'^{op}}
          \\
	      R & & S
          \\
            & G \cl{\theta}
         }
\xymatrix@R=10ex{ \\ \;\;\stackrel{\lozenge_1}{=}\;\; \\}
\xymatrix@C=-1.5ex
         {
             X \dcellb{\!p^{op}} & & Y' \did
          \\
             R \dcell{\!p'} & & Y' \did
          \\
             X' & & Y'
          \\
           & G \cl{\lambda'}
         }
\xymatrix@R=10ex{ \\ \;\;\stackrel{\textcolor{white}{\lozenge_2}}{=}\;\; \\}
\xymatrix@C=-1.5ex
         {
          \\
             X \dcell{R} & & Y' \did
          \\              
             X' & & Y'
          \\
           & G \cl{\lambda'}
         }
$$ 
\end{proof}

\begin{sinnadastandard} \label{neutraldiagramatriangulo}
Two maps $X \mr{f} X'$, $Y \mr{g} Y'$ also give rise to the following diagram:
$$
\rhd = \rhd(f,g): \hspace{5ex}
\vcenter
        {
         \xymatrix@C=3ex@R=1ex
                   {
                    X \times Y  \ar[rrd]^{\lambda} 
                                \ar[dd]_{f \times g}
                    \\
                    {} & \hspace{-4ex} {}^{\!\!\geq} & G \,.
                    \\
                    X' \times Y'  \ar[rru]_{\lambda '}
                   }
        }
$$
\end{sinnadastandard}

\begin{proposition} \label{neutraldiamondimpliesrhd}
If either $\lozenge_1(f,\,g)$ or $\lozenge_2(f,\,g)$ holds, then so does $\;\rhd(f,\,g)$.
\end{proposition}
\begin{proof}
$\lambda\langle a,b \rangle  \leq \displaystyle \bigvee_{g(y)=g(b)} \lambda\langle a,y \rangle  =
  \lambda'\langle f(a),g(b) \rangle $ using $\lozenge_1$. Clearly a symmetric arguing holds
  using $\lozenge_2$. 
\end{proof}
\noindent {\bf For the rest of this section $G$ is assumed to be a locale.}

Consider the following axioms:
\begin{sinnadaitalica} {\bf Axioms on an $\ell$-relation} \label{bijection}

 \vspace{1ex}

$ed) \; \bigvee_{y\in Y} \;\lambda\langle a,y \rangle  = 1$,  
\hspace{6ex} for each a \hfill (everywhere defined).

\vspace{1ex}

$uv) \; \lambda\langle x,b_1 \rangle  \wedge \lambda\langle x,b_2 \rangle  = 0$, \hspace{2ex} for each $x, b_1 \neq b_2$  \hfill (univalued).

\vspace{1ex}

$su) \; \bigvee_{x \in X} \;\lambda\langle x,b \rangle  = 1$, 
\hspace{6ex} for each b \hfill (surjective). 

\vspace{1ex}

$in) \; \lambda\langle a_1,y \rangle  \wedge \lambda\langle a_2,y \rangle  = 0$, \hspace{2ex} for each $y, a_1 \neq a_2$  \hfill (injective).
\end{sinnadaitalica}

Clearly any morphism of locales $G \to H$ preserves these four axioms.

\vspace{1ex}

An $\ell$-relation $\lambda$ is a \emph{$\ell$-function} if and only if satisfies
axioms $ed)$ and $uv)$. We say that an $\ell$-relation is a \emph{$\ell$-opfunction}
when it satisfies   
axioms $su)$ and $in)$. Then an $\ell$-relation is a \emph{$\ell$-bijection} if and only
if it is an $\ell$-function and an $\ell$-opfunction.  

\begin{sinnadastandard} \label{neutralproductrelationdef}
Given two $\ell$-relations, $X \times Y \mr{\lambda} G$, $X' \times
Y'\mr{\lambda'}G$, the product $\ell$-relation $\lambda \boxtimes \lambda'$ is defined by the composition

\begin{center}
$X \times X' \times Y \times Y' 
 \mr{X \times \psi \times Y'} X \times Y \times X' \times Y' 
 \mr{\lambda \times \lambda'} G \times G \mr{\wedge} G$ 
\end{center}

\begin{center}
$
(\lambda \boxtimes \lambda')  \langle (a,a'),(b,b') \rangle  = \lambda  \langle a,b \rangle  \wedge \lambda'  \langle a',b' \rangle$. 
\end{center}
\end{sinnadastandard}
The following is immediate and straightforward:
\begin{proposition} \label{neutralproductrelation}
Each axiom in \ref{bijection} for $\lambda$ and $\lambda'$ implies the respective axiom for the product $\lambda \boxtimes \lambda'$. \cqd 
\end{proposition}

\begin{proposition} \label{infessup}
We refer to \ref{neutraldiagramadiamante12}: If equations $\lozenge_1(p,q)$ and $\lozenge_1(p',q')$ hold, and $\theta$ satisfies $uv)$, then equation 1) below holds. \mbox{Symmetrically,} if  
$\lozenge_2(p,q)$ and $\lozenge_2(p',q')$ hold, and $\theta$ satisfies 
$in)$, then equation 2) below holds.
\begin{center}
$
1) \hspace{5ex}
    \lambda \langle p(r), b \rangle \wedge 
    \lambda' \langle p'(r), b' \rangle  \,=\, 
    \displaystyle\bigvee_{\substack{q(v) = b \\ q'(v) = b'}} \theta 
    \langle  r, v \rangle.
$
\vspace{1ex}
$
2) \hspace{5ex}
    \lambda \langle a, q(s) \rangle \wedge 
    \lambda' \langle  a', q'(s) \rangle  \,=\, 
    \displaystyle\bigvee_{\substack{p(u) = a \\ p'(u) = a'}} \theta 
    \langle  u, s \rangle.
$
\end{center}
\end{proposition}
\begin{proof}
We only prove the first statement, since the second one clearly has a symmetric proof.
\vspace{1ex}

$
\lambda \langle p(r), b \rangle \wedge 
    \lambda' \langle p'(r), b' \rangle   
    \; \stackrel{\lozenge_1}{=} \; \displaystyle \bigvee_{q(v)=b} \theta \langle r,v \rangle 
\,\wedge  
\displaystyle \bigvee_{q'(w)=b'} \theta \langle r,w \rangle  \; = \; 
$

\hfill
$
= \; \displaystyle\bigvee_{\substack {q(v) = b \\ q'(w) = b'}} \theta 
    \langle  r, v \rangle  \wedge \theta \langle  r, w \rangle \; 
    \stackrel{uv)}{=} \;
 \displaystyle\bigvee_{\substack {q(v) = b \\ q'(v) = b'}} \theta 
    \langle  r, v \rangle.
$
\end{proof}   

We study now the validity of the reverse implication in proposition \ref{neutraldiamondimpliesrhd}.

\begin{proposition} \label{neutralrhdimpliesdiamond} 
We refer to \ref{neutraldiagramadiamante12}:
${}$

1) If $\lambda$ is ed) and $\lambda'$ is uv) (in particular, if they are $\ell$-functions), then $\rhd(f,g)$ implies $\lozenge_1(f,g)$. 
 
\vspace{1ex}

2) If $\lambda$ is su) and $\lambda'$ is in) (in particular, if they are $\ell$-opfunctions), then $\rhd(f,g)$ implies $\lozenge_2(f,g)$. 
\end{proposition}
\begin{proof}
We prove $1)$, a symmetric proof yields $2)$.

$\lambda'\langle f(a),b' \rangle  \stackrel{ed)_{\lambda}}{\;=\;} \lambda'\langle f(a),b' \rangle  \wedge \bigvee_{y} \lambda\langle a,y \rangle  \;=\; \bigvee_{y}
\lambda'\langle f(a),b' \rangle  \wedge \lambda\langle a,y \rangle  \stackrel{(*)}{\;=\;}$ 

\vspace{1ex}

{\hfill $\bigvee_{g(y)=b'} \lambda'\langle f(a),b' \rangle  \wedge \lambda\langle a,y \rangle  \stackrel{\rhd}{\;=\;} \bigvee_{g(y)=b'} \lambda\langle a,y \rangle ,$}

\vspace{1ex}

\noindent where for the equality marked with $(*)$ we used that if $g(y) \neq b'$ then $$\lambda'\langle f(a),b' \rangle  \wedge \lambda\langle a,y \rangle \stackrel{\rhd}{\leq} \lambda'\langle f(a),b' \rangle  \wedge \lambda'\langle f(a),g(y) \rangle  \stackrel{uv)_{\lambda'}}{=} 0.$$
\end{proof}

\begin{sinnadastandard} \label{neutralmoregenerally}
More generally, consider two spans as in \ref{neutraldiagramadiamante12}. We have the following \mbox{$\rhd$ diagrams:} 
\begin{equation} \label{neutral2rhd}
         \xymatrix@C=3ex@R=1ex
                   {
                    R \times S  \ar[rrd]^{\theta} 
                                \ar[dd]_{p \times q}
                    \\
                    {} & \hspace{-4ex} {}^{\!\!\geq} & G\,,
                    \\
                    X \times Y  \ar[rru]_{\lambda}
                   } 
\hspace{5ex}       
         \xymatrix@C=3ex@R=1ex
                   {
                    R \times S  \ar[rrd]^{\theta} 
                                \ar[dd]_{p' \times q'}
                    \\
                    {} & \hspace{-4ex} {}^{\!\!\geq} & G \,.
                    \\
                    X' \times Y'  \ar[rru]_{\lambda '}
                   }
\end{equation}
\end{sinnadastandard}

\begin{proposition} \label{neutral2rhdimpliesdiamante}

We refer to \ref{neutraldiagramadiamante12}: Assume that  
$\lambda$ is in), 
$\lambda '$ is uv), and that the 
$\rhd(p,\,q)$, $\rhd(p',\,q')$ diagrams hold. Then  if 
$\theta$ is ed) and su), diagram $\lozenge(R,\,S)$ holds.
\end{proposition}
\begin{proof}
Use proposition \ref{neutralrhdimpliesdiamond} twice: First with $f =  p'$, $g = q'$, $\lambda = \theta$, 
$\lambda' = \lambda'$ to have $\lozenge_1(p',\,q')$. Second with $f = p$, $g = q$, $\lambda = \theta$, 
$\lambda' = \lambda$ to have 
$\lozenge_2(p,\,q)$. Then use proposition \ref{neutralpre2rhdimpliesdiamante}.
\end{proof}

\begin{remark} \label{neutralproductistheta}
Note that the diagrams $\rhd$ in \ref{neutral2rhd} mean that \mbox{$\theta \leq \lambda \boxtimes \lambda' \circ (p, p') \times (q, q')$} (see \ref{neutralproductrelationdef}). In particular, when $G$ is a locale, there is always an $\ell$-relation $\theta$ in \ref{neutraldiagramadiamante12}, which may be taken to be the composition 
\mbox{$R \times S \mr{(p, p') \times (q, q')} X \times X' \times Y \times Y' 
\mr{\lambda \boxtimes \lambda'} G$}. However, it is important to consider an arbitrary $\ell$-relation $\theta$ (see propositions \ref{neutraltrianguloesdiamante} and \ref{neutralsupisloc}).
\end{remark}
\begin{proposition} \label{neutraldiamanteimplies2rhd}
We refer to \ref{neutraldiagramadiamante12}: Assume that $R$ and $S$ are relations, that $\lambda$, $\lambda '$ are $\ell$-bijections, and that  
$\rhd(p,\,q)$, $\rhd(p',\,q' )$ in \eqref{neutral2rhd} hold. Take 
\mbox{$\theta=\lambda \boxtimes \lambda' \circ (p, p') \times (q, q')$}. Then,
if $\lozenge(R,\,S)$ holds, $\theta$ is an $\ell$-bijection.
\end{proposition}
\begin{proof}
We can safely assume  $R \subset X \times X'$ and 
$S \subset Y \times Y'$, and \mbox{$\lambda \boxtimes \lambda' \circ (p, p') \times (q, q')$} to be the restriction of $\lambda \boxtimes \lambda'$ to $R \times S$.
From the $\rhd$ diagrams \eqref{neutral2rhd}  we easily see that axioms $uv)$ and $in)$ for $\theta$ follow from the corresponding axioms for 
$\lambda$ and $\lambda'$. We prove now axiom $ed)$, axiom $su)$ follows in a symmetrical way. 
Let $(a,a') \in R$, we compute:

\vspace{1ex}

$
\displaystyle\bigvee_{(y,y') \in S} \theta \langle (a,a'), (y,y') \rangle 
\,=\, 
\bigvee_{y'} \bigvee_{(y,y') \in S} \lambda \langle a,y \rangle \wedge \lambda' \langle a',y'\rangle
\stackrel{\lozenge}{=}$ 

\hfill 
$
\stackrel{\lozenge}{=} 
\displaystyle\bigvee_{y'} \bigvee_{(a,x') \in R} \lambda' \langle x',y' \rangle \wedge \lambda' \langle a',y'\rangle 
\geq 
\bigvee_{y'} \lambda' \langle a',y'\rangle 
\stackrel{ed)}{=} 1
$
\end{proof}

We found convenient to combine \ref{neutral2rhdimpliesdiamante} and \ref{neutraldiamanteimplies2rhd} into:
\begin{proposition} \label{neutralcombinacion}
Let $R \subset X \times X'$, $S \subset Y \times Y'$ be any two relations, and $X \times Y \mr{\lambda} G$, 
$X' \times Y' \mr{\lambda'} G$ be $\ell$-bijections.  Let $R \times S \mr{\theta} G$ be the restriction of $\lambda \boxtimes \lambda'$ to $R \times S$. Then, 
$\lozenge(R, S)$ holds if and only if $\theta$ is an $\ell$-bijection. \cqd
\end{proposition}

\pagebreak

\subsection{On $\rhd$ and $\lozenge$ cones} \label{conesneutral}

We consider a pointed topos $\cc{E}ns \mr{f} \cc{E}$, with inverse image $f^* = F$. 
\begin{sinnadastandard} \label{neutralrel}
 Let $\cc{R}el(\cc{E})$ be the category of relations in $\cc{E}$. $\cc{R}el(\cc{E})$ is a symmetric monoidal category with tensor product given by the cartesian product in 
$\cc{E}$ (which is not cartesian in $\cc{R}el(\cc{E})$). Every object $X$ has a dual, and it is self dual, the unit and the counit of the duality are both given by the diagonal relation $\Delta \subset X \times X$ (see \ref{tensoriso}). There is a faithful functor \mbox{$\cc{E} \rightarrow \cc{R}el(\cc{E})$}, the identity on objects and the graph on arrows, we will often abuse notation and identify an arrow with its graph. The functor $\cc{E} \mr{F} \cc{E}ns$ has an extension 
\mbox{$\cc{R}el(\cc{E}) \mr{\cc{R}el(F)} \cc{R}el$,} if $R \subset X \times Y$ is a relation, then $FR \subset FX \times FY$, and 
$\cc{R}el(F)$ is in this way a tensor functor.
We have a commutative diagram:
$$
\xymatrix
        {
           \cc{E} \ar[r] \ar[d]_F
         & \cc{R}el(\cc{E}) \ar[d]^{T}
         \\
           \cc{E}ns \ar[r] 
         & \cc{R}el \; \ar@{^(->}[r]^{\ell} 
         & s \ell 
         & {(\text{where} \;  T = \cc{R}el(F))} 
        }
$$
\end{sinnadastandard}

\begin{sinnadaitalica} \label{neutralFequivalenceiffRel(F)so}
It can be seen that          
$F$ is an equivalence if and only if $T$ is so.  
\cqd
\end{sinnadaitalica} 

Note that on objects $TX = FX$ and on arrows in $\cc{E}$, $T(f) = F(f)$. Since $F$ is the inverse image of a point, the diagram of $F$ is a cofiltered category, $T(X \times Y) = TX \times TY$, if $C_i \to X$ is an epimorphic family in 
$\cc{E}$, then $TC_i \to TX$ is a surjective family of sets. If $R$ is an arrow in $\cc{R}el(\cc{E})$, \mbox{$T(R^{op}) = (TR)^{op}$.}

\begin{\de} \label{neutraldefdeconos}
Let $H$ be a sup-lattice furnished with an $\ell$-relation $TX \times TX \mr{\lambda_{X}} H$ for each $X \in \cc{E}$. Each arrow $X \mr{f} Y$ in 
$\cc{E}$ and each arrow 
$X \mr{R} Y$ in $\cc{R}el(\cc{E})$ (i.e relation 
$R \mmr{} X \times Y$), determine the following diagrams: 
$$
\xymatrix@C=4ex@R=3ex
        {
         FX \times FX \ar[rd]^{\lambda_{X}}  
                      \ar[dd]_{F(f) \times F(f)}  
        \\
         {} \ar@{}[r]^(.3){\geq}
         &  H\,,    
        \\
         FY \times FY \ar[ru]^{\lambda_{Y}} 
         } 
\hspace{3ex}
\xymatrix@C=4ex@R=3ex
        {
         & TX \times TX \ar[rd]^{\lambda_{X}}  
        \\
           TX \times TY 
               \ar[rd]_{TR \times TY \hspace{2.5ex}} 
			   \ar[ru]^{TX \times TR^{op}\hspace{2.5ex}} 
	     & \equiv 
         & H\,. 
        \\
         & TY \times TY \ar[ru]^{\lambda_{Y}} 
         }
$$ 
We say that $TX \times TX \mr{\lambda_{X}} H$ is a \emph{$\rhd$-cone} if the $\, \rhd(F(f),\, F(f)) \,$ diagrams  hold, and that it is a 
\emph{$\lozenge$-cone} if the $\,\lozenge(TR,TR)\,$ diagrams hold. Similarly we talk of \emph{$\lozenge_1$-cones} and \emph{$\lozenge_2$-cones} if the 
$\lozenge_1(F(f),\, F(f))$ and $\lozenge_2(F(f),\, F(f))$ \mbox{diagrams} hold. We will abbreviate $\lozenge(R) = \lozenge(TR,TR)$, and similarly $\rhd(f)$, $\lozenge_1(f)$ and $\lozenge_2(f)$. 
If $H$ is a locale and the 
$\lambda_X$ are $\ell$-bijections, we say that we have a $\lozenge$-cone or a $\rhd$-cone \textit{of $\ell$-bijections}.
\end{\de}

\begin{proposition} \label{neutraldim1ydim2esdim}
A family $TX \times TX \mr{\lambda_{X}} H$ of $\ell$-relations is a $\lozenge$-cone if and only if it is both a $\lozenge_1$ and a $\lozenge_2$-cone.
\end{proposition}
\begin{proof}
Use proposition \ref{neutralpre2rhdimpliesdiamante} with $R = TR$, $S = TR$, $p = p' = \pi_1$, $q = q' = \pi_2$, $\lambda = \lambda_X$,  
$\lambda' = \lambda_Y$, and $\theta = \lambda_R$. Then, $\lozenge_1(\pi_2)$ and $\lozenge_2(\pi_1)$ imply $\lozenge(R)$
\end{proof}
\begin{proposition} \label{neutraltrianguloesdiamante}
Any $\rhd$-cone $TX \times TX \mr{\lambda_{X}} H$ of $\ell$-bijections with values in a locale $H$ is a $\lozenge$-cone (of $\ell$-bijections).
\end{proposition}
\begin{proof}
Given any relation $R \mmr{} X \times Y$, consider proposition \ref{neutral2rhdimpliesdiamante} with \mbox{$\lambda = \lambda_X$,} $\lambda' = \lambda_Y$, and $\theta = \lambda_R$.
\end{proof}

\begin{definition} \label{neutralcomp}
Let $TX \times TX \mr{\lambda_{X}} H$ be a $\lozenge$-cone with values in a commutative algebra $H$ in $s \ell$, with multiplication $*$ and unit $u$. 
We say that it is \emph{compatible} if the following equations hold:
$$
\lambda_X \langle a,\,a' \rangle *  \lambda_Y \langle b,\,b' \rangle = 
\lambda_{X \times Y}\langle (a,\,b),\,(a',\,b') \rangle \,,
\hspace{3ex} \lambda_1(*,\,*) = u. 
$$
\end{definition}

Any compatible $\lozenge$-cone wich covers $H$ forces $H$ to be a locale, and such a cone is necessarily a cone of $\ell$-bijections (and vice versa). We examine this now:

Given a compatible cone, consider the diagonal $X \mr{\Delta} X \times X$, the arrow $X \mr{\pi} 1$, and the following $\lozenge_1$ diagrams:
$$
\xymatrix@R=4ex@C=2ex
        { 
         & TX \!\times\! TX \ar[rd]^{\lambda_X} 
         & & & TX \!\times\! TX \ar[rd]^{\lambda_X} 
        \\
	     TX \!\times\! TX \!\times\! TX 
	                      \ar[ru]^{TX \!\times\! \Delta^{op}} 
	                      \ar[rd]_{\Delta \!\times\! TX \!\times\! TX \ \ } 
	     & \equiv 
	     & \hspace{0ex} H, 
	     & TX \!\times\! 1 \ar[ru]^{TX \!\times\! \pi^{op}} 
	                  \ar[rd]_{\pi \!\times\! 1} 
	     & \!\! \equiv 
	     & \hspace{0ex} H.
	    \\
	     & TX \!\times\! TX \!\times\! TX \!\times\! TX 
	                          \ar[ru]_{\lambda_{X \!\times\! X}} 
	     & & & 1 \!\times\! 1 \ar[ru]_{\lambda_1} 
	    }
$$
Let $a,\, b_1,\, b_2 \in TX$, and let $b$ stand for either $b_1$ or $b_2$. Chasing $(a,\,b_1,\,b_2)$ in the first diagram and $(a,*)$ in the second it follows: 

\vspace{1ex}

$
(1) \;\;
\lambda_X\langle a,\,b_1\rangle * \lambda_X\langle a,\,b_2\rangle \,=
\lambda_{X \times X}\langle (a,\,a),\,(b_1,\,b_2) \rangle \, = \;
\delta_{b_1 = b_2} \, \lambda_X\langle a,\,b \rangle. 
$

\vspace{1ex} 

$
 (2) \;\; 
 \lambda_X(a,\,b) \;\leq\; 
 \bigvee_x \lambda_X\langle a,\,x\rangle \;=\; 
 \lambda_1(*,*) = u.
$
\begin{proposition} \label{neutralcompislocale}
Let $H$ be a commutative algebra, and $TX \times TX \mr{\lambda_{X}} H$ be a compatible $\lozenge$-cone such that the elements of the form $\lambda_X(a,\,a')$, $a,\, a' \in TX$ are sup-lattice generators. Then   
$H$ is a locale and $*=\wedge$. 
\end{proposition}
\begin{proof}
We have to prove that for all $w \in H$, (L1) $w * w = w$ and (L2) $w \leq u$ (see \ref{charlocales}). 

It immediately follows from equations (1) and (2) above that (L1) and (L2) hold for $w = \lambda_C( a,\,b )$. Then clearly (L2) holds for any supremum of elements of this form. To show (L1) we do as follows:

$w * w \leq w * 1 = w$ always holds, and to show $\geq$, if $w = \displaystyle \bigvee_{i \in I} w_i$ satisfying $w_i * w_i = w_i$ we compute:

$$ \displaystyle \bigvee_{i \in I} w_i * \bigvee_{i \in I} w_i \geq \bigvee_{i \in I} w_i * w_i \stackrel{(L1)}{=} \bigvee_{i \in I} w_i.$$

\end{proof}

\begin{proposition}\label{neutralDiamondisbijection}  
A $\lozenge$-cone $TX \times TX \mr{\lambda_{X}} H$ with values in a locale $H$ is compatible if and only if it is a $\lozenge$-cone of $\ell$-bijections.
\end{proposition}
\begin{proof}
($\Rightarrow$): Clearly equations (1) and (2) above are the axioms uv) and ed) for $\lambda_X$. Axioms in) and su) follow by the symmetric argument using the corresponding $\lozenge_2$ diagrams.

($\Leftarrow$) $u = 1$ in $H$, so the second equation in definition \ref{neutralcomp} is just axiom $ed)$ (or $su)$) for $\lambda_1$. To prove the first equation we do as follows: 

Consider the projections $X \times Y \mr{\pi_1} X$, $X \times Y \mr{\pi_2} Y$. The $\lozenge_1(\pi_1)$  and $\lozenge_1(\pi_2)$ diagrams express the equations:

\vspace{1ex}

$
 \lambda_X\langle a,a'\rangle = \bigvee_y \lambda_{X \times Y} \langle (a,b),(a',y) \rangle, \;\;  \lambda_Y \langle b,b'\rangle = \bigvee_x \lambda_{X \times Y} \langle (a,b),(x,b') \rangle.
$

\vspace{1ex}

Taking the infimum of these two equations:

\vspace{1ex}

\noindent
$
\lambda_X \langle a,a' \rangle \wedge \lambda_Y \langle b,b' \rangle  \;=\;\bigvee_{x,y} \lambda_{X \times Y} \langle (a,b),(a',y) \rangle  \wedge \lambda_{X \times Y} \langle (a,b),(x,b') \rangle \stackrel{(*)}{\;=\;} 
$

\vspace{1ex}

\noindent
$
\stackrel{(*)}{\;=\;} \lambda_{X \times Y}\langle(a,b),(a',b')\rangle$, as desired 
($\stackrel{(*)}{\;=\;}$ justified by $uv)$ for $\lambda_{X \times Y}$). 
\end{proof}

\begin{proposition} \label{neutralsupisloc}
Let $TX \times TX \mr{\lambda_{X}} H$ be a $\lozenge$-cone of $\ell$-bijections such that the elements of the form $\lambda_X(a,\,a')$, $a,\, a' \in TX$ are locale generators. Then, any linear map $H \mr{\sigma} G$ into another 
$\lozenge$-cone of $\ell$-bijections, 
\mbox{$TX \times TX \mr{\lambda_{X}} G$,} satisfying $ \sigma \lambda_X = \lambda_X$, preserves infimum and $1$, thus it is a locale morphism.
\end{proposition}
\begin{proof}
By axiom $ed)$ for $\lambda_1$, in both locales 
$\lambda_1(*,\,*) = 1$. Since $\sigma \lambda_1 = \lambda_1$, this shows that $\sigma$ preserves $1$. 

To show that infima are preserved it is enough to prove that infima of the form $\lambda_X \langle a, a' \rangle  \wedge 
\lambda_Y \langle b, b' \rangle$, $a,\,a' \in TX,\; b,\, b' \in TY$ are preserved. 
Take  
$
\vcenter
       {
        \xymatrix@R=0ex@C=3ex
                {
                 & (X,\, a) 
                \\ (Z,\, c) \ar[ru]^{f} \ar[rd]_{g} 
                \\ 
                 & (Y,\, b)
                 } 
       }
$ in the diagram of $F$. 
Then, by \mbox{proposition} \ref{infessup} with \mbox{$\lambda = \lambda_X$,} $\lambda' = \lambda_Y$, and $\theta = \lambda_Z$, it follows that the equation  
\mbox{$\lambda_X \langle a, a' \rangle  \wedge 
\lambda_Y \langle b, b' \rangle  = 
\displaystyle\bigvee_{T(f)(z) = a' \,,\, T(g)(z) = b'} \lambda_Z \langle  c, z \rangle$} holds in both locales. The proof finishes using that  $\sigma$ preserves suprema and $\sigma\lambda_Z=\lambda_Z$.
\end{proof}

Consider now a (small) site of definition $\cc{C} \subset \cc{E}$ of the topos $\cc{E}$. Suitable cones defined over $\cc{C}$ can be extended to $\cc{E}$. More precisely:
\begin{proposition} $ $ \label{neutralextension}

1) Let $TC \times TC \mr{\lambda_C} H$ be a $\lozenge_1$-cone (resp. 
a $\lozenge_2$-cone) over $\cc{C}$. Then, $H$ can be (uniquely) furnished with $\ell$-relations $\lambda_X$ for all objects $X \in \cc{E}$ in such a way to determine a $\lozenge_1$-cone (resp. a $\lozenge_2$-cone) over $\cc{E}$. 
 
2) If $H$ is a locale and all the $\lambda_C$ are $\ell$-bijections, so are all the $\lambda_X$.
\end{proposition}
\begin{proof} 

1)  
Let $X \in \cc{E}$ and $(a,\,b) \in TX \times TX$.  Take $C \mr{f} X$ and $c \in TC$ such that $a = T(f)(c)$. If $\lambda_X$ were defined so that the 
$\lozenge_1(f)$ diagram commutes, the equation 
$$
(1) \hspace{2ex} \lambda_X\langle a,\,b\rangle  = \bigvee_{T(f)(y)=b} \lambda_C\langle c,\,y\rangle
$$ 
should hold (see \ref{neutraltriangleequation}). We define $\lambda_X$ by this equation. This definition is independent of the choice of $c, \, C,$ and $f$. In fact, let $D \mr{g} X$ and $d \in TD$ such that 
$a = T(g)(d)$, and take 
$(e,\, E)$ in the diagram of $F$, $E \mr{h} C$, $E \mr{\ell} D$ such that 
$T(h)(e) = c$, $T(\ell)(e) = d$ and $T(fh)=T(g\ell)$. Then we compute
$$
\bigvee_{T(f)(y)=b} \lambda_C\langle c,\,y\rangle \stackrel{\lozenge_1(h)}{=}
\bigvee_{T(f)(y)=b} \;\;\bigvee_{T(h)(w)=y} \lambda_E\langle e,\,w\rangle \;= 
\bigvee_{T(fh)(w)=b} \lambda_E\langle e,\,w\rangle.
$$
From this and the corresponding computation with $d, \, D,$ and $\ell$ it follows:
$$
\bigvee_{T(f)(y)=b} \lambda_C\langle c,\,y\rangle = \bigvee_{T(g)(y)=b} \lambda_{D}\langle d,\,y\rangle.
$$
Given $X \mr{g} Y$ in $\cc{E}$, we check that the  
$\lozenge_1(g)$ diagram commutes: Let $(a,\,b) \in TX \times TY$, take $C \mr{f} X$ such that $a = T(f)(c)$, and let \mbox{$d = T(g)(a) = T(gf)(c)$.} Then:
$$
\lambda_Y \langle d,b \rangle = 
\hspace{-1ex} \bigvee_{T(gf)(z)=b} \lambda_C \langle c,z \rangle = 
\hspace{-1ex} \bigvee_{T(g)(x)=b} \bigvee_{T(f)(z)=x} \lambda_C \langle c,z \rangle = 
\hspace{-1ex} \bigvee_{T(g)(x)=b} \lambda_X \langle a,b \rangle.
$$
Clearly a symmetric argument can be used if we assume at the start that the $\lozenge_2$ diagram commutes. In this case, $\lambda_X$ would be defined by:                
$$
(2) \hspace{2ex}  \lambda_X\langle a,\,b\rangle  = \bigvee_{T(f)(y)=a} \lambda_C\langle y,\,c\rangle
$$ 
with $T(f)(c) = b$. 

If the $TC \times TC \mr{\lambda_C} H$ form a $\lozenge_1$ and a 
$\lozenge_2$ cone, definitions (1) and (2) coincide. In fact, since they are both independent of the chosen $c$, it follows they are equal to 
$$
\bigvee_{T(f)(y)=b, \;T(f)(c)=a} \lambda_C\langle c,\,y\rangle 
\hspace{2ex} = \hspace{2ex} 
\bigvee_{T(f)(y)=a, \;T(f)(c)=b} \lambda_C\langle y,\,c\rangle 
$$ 

\vspace{1ex}

2) It is straightforward and we leave it to the reader.
\end{proof}

It is worthwhile to consider the case of a locally connected topos. In this case it clearly follows from the above (abusing notation) that given $a, \; b \in TX$, if $a, \, b$ are in the same connected component 
$C \subset X$,  $a, \; b \in TC$, then 
$\lambda_X(a,\,b) = \lambda_C(a,\,b)$, and if they are in different connected components, then $\lambda_X(a,\,b) = 0$. When the topos is atomic and $H = Aut(F)$ (see \ref{Aut(F)}), the reverse implication holds, namely, if $\lambda_X(a,\,b) = 0$, then $a, \; b$ must be in different connected components (Theorem \ref{keyLGT}, 1)).

 \pagebreak

\subsection{The isomorphism $Cmd_0(G) = Rel(\beta^G$)} \label{Cmd_0}

The purpose of this section is to establish an isomorphism of categories between \linebreak $Cmd_0(G)$ and $Rel(\beta^G)$, where $G$ is a 
fixed localic group, or, what amounts to the same thing,
an idempotent Hopf algebra in the monoidal category $s \ell$ of sup-lattices, as we explained in section \ref{background}. 

\begin{sinnadastandard} {\bf The category $Cmd_0(G)$.} 
\end{sinnadastandard}
As for any coalgebra, a \emph{comodule} structure over $G$ is a 
sup-lattice  $S \in s \ell$ together with a map $S \stackrel{\rho}{\rightarrow} G \otimes S$ satisfying the coaction axioms:
\begin{equation} \label{coaction}
(G \otimes \rho) \circ \rho =  (w \otimes S) \circ \rho, \;\; and \;\; 
 (e \otimes S) \circ \rho = \; \cong_S.
\end{equation}
where $w$, $e$ are the comultiplication and the counit of $G$, and $\cong_S$ is the isomorphism $2 \otimes S \cong S$.

A \emph{comodule morphism} between two comodules is a map which makes the
usual diagrams commute (see \cite{JS}). We define the category $Cmd_0(G)$ as the full subcategory 
with objects the comodules of the form $S = \ell X$, for any set $X$. If we forget the comodule structure we have a faithful functor 
$$Cmd_0(G) \mr{T} s \ell_0 = \cc{R}el.$$

\begin{sinnadastandard} \label{thecategorybetaG} {\bf The category $\beta^G$.} \end{sinnadastandard}
The construction of the category $\beta^G$ of sets furnished with an action of $G$ (namely,
the \emph{classifying topos} of $G$) requires
some considerations (for details see \cite{D1}).
Given a set $X$, we define the locale $Aut(X)$ to be the universal \mbox{$\ell$-bijection} in the category of locales, $X \times X \mr{\lambda} Aut(X)$. It is constructed in two steps: first consider the free locale on $X \times X$, $X \times X \mr{\jmath} Rel(X)$. Clearly it is the universal $\ell$-relation in the category of locales. Second, $Rel(X) \mr{} Aut(X)$ is determined by the topology generated by the covers that force the four axioms in \ref{bijection} (see \cite{GW}, \cite{D1}).
 Notice that it follows by definition that the points of the locales $Rel(X)$ and $Aut(X)$ are the relations and the bijections of the set $X$.
Given $(x,\,y) \in X \times X$, 
we will denote $\langle x\,|\,y \rangle  = \jmath \langle x,y \rangle = \lambda \langle x,y \rangle$ indistinctly in both cases. We abuse notation and omit to indicate the associated sheaf morphism  $Rel(X) \mr{}  Aut(X)$. The elements of the form $\langle x\,|\,y \rangle$ generate both locales by taking arbitrary suprema of finite infima. 

It is straightforward to check that the following maps are $\ell$-bijections.
\begin{multline} \label{algebrastructure}
w: \; X \times X  \mr{}  Aut(X) \otimes Aut(X),\;\;
w \langle x \:|\:y \rangle  \;=\;
 \bigvee\nolimits_{z} \; \langle  x \:|\:z  \rangle \otimes \langle  z \:|\:y  \rangle \,, 
\\
\hspace{-36ex} e: \; X \times X \mr{} 2, \;\;  e \langle  x\,|\,y  \rangle  \;=\;  \delta_{x=y}.
\\
\hspace{-28ex} \iota: \; X \times X \mr{} Aut(X), \;\;
\iota \langle  x \:|\:y  \rangle  \;=\; \langle  y \:|\:x  \rangle \,.
\\
\end{multline}

 It follows (from the universal property) that they determine locale morphisms with domain $Aut(X)$. They define a coalgebra structure on the locale $Aut(X)$, which furthermore results a Hopf algebra (or localic group).

\vspace{1ex}

An \emph{action} of a localic group $G$ in a set $X$ is defined as a localic
group morphism $G \mr{\widehat \mu} Aut (X)$. This
corresponds to a Hopf algebra morphism $Aut(X) \mr{\mu} G$, which 
is completely determined by its value on the generators, that is,
an $\ell$-bijection $X \times X \mr{\mu} G$, that in addition satisfies
\begin{equation} \label{morfismo}
w \mu  \;=\; (\mu  \otimes \mu )w \,,\;\;\;\;\;\;
e \mu  \;=\; e \,,\;\;\;\;\;\; 
\mu \iota  \;=\; \iota \mu. 
\end{equation}
(the structures in both Hopf algebras are indicated with the same letters).

As we shall see in Proposition \ref{monoidaction}, the equation $\mu \iota  \;=\; \iota \mu$ follows from the other two. That is, any action of $G$ viewed as a monoid is automatically a group action.

Given two objects $X, X' \in \beta^G$, a \emph{morphism} between them is a function between the sets $X \mr{f} X'$ 
satisfying $\mu\langle a|b \rangle  \leq \mu'\langle f(a)|f(b) \rangle $. Notice that this is a $\;\rhd\;$ diagram as in section \ref{sec:general}.

\vspace{1ex}

If we forget the action we have a faithful functor $\beta^G \mr{F} \cc{E}ns$  (which is the inverse image of a point of the topos, see \cite{D1} Proposition 8.2). Thus, we have a commutative square (see \ref{neutralrel}):
$$
\xymatrix
        {
           \beta^G \ar[r] \ar[d]_F
         & \cc{R}el(\beta^G) \ar[d]^{\cc{R}el(F)}
         \\
           \cc{E}ns \ar[r] 
         & \cc{R}el.
        }
$$
We have the following theorem, that we will prove in the rest of this section.
\begin{theorem} \label{neutralComd=Rel}
There is an isomorphism of categories making the triangle commutative:
$$
\xymatrix@C=0ex
        {
          Cmd_0(G) \ar[rr]^{=} \ar[rd]_T 
      & & \cc{R}el(\beta^G) \ar[ld]^(.4){\cc{R}el(F)} 
        \\
         & s \ell_0 = \cc{R}el.
        }       
$$ 
The identification between relations $R \subset X \times X'$ and linear maps $\ell X \to \ell X'$ lifts to the upper part of the triangle.

\vspace{-2.5ex}

\hfill $\Box$
\end{theorem} 
 
\vspace{1ex}       

Recall that since the functor $F$ is the inverse image of a point, it follows that mono- morphisms of $G$-sets are injective maps.
\begin{proposition} \label{mono1}
 Let $f: X \rightarrow X'$ a morphism of $G$-sets. Then for each 
 \mbox{$a, b \in X$}, $$ \mu'  \langle f(a)|f(b) \rangle  = \bigvee_{f(x)=f(b)} \mu  \langle a|x \rangle .$$
In particular, if $f$ is a monomorphism, we have $\mu'  \langle f(a)|f(b) \rangle  = \mu  \langle a|b \rangle $.
\end{proposition}
\begin{proof}
Since the actions are $\ell$-bijections, in particular $\ell$-functions, by  proposition \ref{neutralrhdimpliesdiamond} the $\rhd$ diagram implies the $\lozenge_1$ diagram. The statement follows by taking $(a,f(b)) \in X \times X'$.
\end{proof}
Proposition \ref{mono1} says that the subobjects $Z \mono X$ of an object $X$ in                               
$\beta^G$ are the subsets $Z \subset X$ such that the restriction of the action \mbox{$Z \times Z \subset X \times X \mr{\mu} G$} is an action on $Z$. We have:
\begin{proposition} \label{mono2}
Let $X$ be a $G$-set and $Z \subset X$ any subset. If the restriction of the action to $Z$ is an $\ell$-bijection, then it is already an action.
\end{proposition}
\begin{proof}
We have to check the equations in \ref{morfismo}. The only one that requires some care is the first. Here it is convenient to distinguish notationally as $w_Z$, $w_X$ and $w$ the comultiplications of $Aut(Z)$, $Aut(X)$ and $G$ respectively. By hypothesis we have (1)
$w\mu \langle a|b \rangle = (\mu \otimes \mu) w_X \langle a|b \rangle = \displaystyle \bigvee_{x \in X} \mu \langle 
a|x \rangle \otimes \mu \langle x|b \rangle$. We claim that when 
$a,\, b \in Z$, this equation still holds by restricting the supremum to the $x \in Z$, which is the equation $w\mu \langle a|b \rangle = (\mu \otimes \mu) w_Z$. 
In fact, from axioms $ed)$ and $su)$ for $\mu$ on $Z$ it follows  
(2) $1 = \displaystyle \bigvee_{y,\, z \,\in Z} \mu \langle a|y \rangle \otimes \mu \langle z|b \rangle$. Then, the claim follows by taking the infimum in both sides of equations (1) and (2), and then using the axioms $uv)$ and $in)$ for $\mu$ on $X$.
\end{proof}

\begin{proposition}  \label{monoidaction}
 Given a localic group $G$ and a localic monoid morphism $G \stackrel{\widehat{\mu}}{\rightarrow} Rel(X)$, there exists a unique action of $G$ in $X$ such that
$$ 
\xymatrix@C=4ex@R=4ex
         {
          Rel(X) 
          & & G, \ar[ll]_{\widehat{\mu}} \ar@{-->}[dl]^{\widehat\mu}	
          & \hbox{i.e.} 
          & Rel(X) \ar[rr]^{\mu} \ar[rd] 
          & & G. 
          \\ 	
		  & Aut(X) \ar[ul] 
		  & & & & Aut(X) \ar@{-->}[ru]_{\mu} 
         } 
$$
\end{proposition}
\begin{proof}
 $\mu$ is determined by an $\ell$-relation $X \times X \mr{\mu} G$ preserving $w$ and $e$ (see equations \ref{morfismo}). It factorizes through $Aut(X)$ provided it is an $\ell$-bijection, and the factorization defines an action if it also preserves $\iota$. 

Consider the following commutative diagram
$$
\xymatrix
         {
          & X \times X \ar[r]^(.35){w} \ar[d]^\mu \ar[dl]_e
          & Rel(X) \otimes Rel(X) \ar[d]^{\mu \otimes \mu}
          \\
          2 \ar[dr]_u
          & G \ar[r]^w \ar[l]_e
          & G \otimes G \ar@<-1.5ex>[d]_{\iota \otimes G}
                        \ar@<1.5ex>[d]^{G \otimes \iota}
          \\
          & G
          & G \otimes G. \ar[l]_\wedge
         } 
$$
Chasing an element $(b,\, b) \in X \times X$  all the way down to $G$ using the arrow $G \otimes \iota$ it follows
$\displaystyle \bigvee_{y} \mu  \langle b|y \rangle  \wedge \iota \mu  \langle y|b \rangle  = 1$. Thus, in particular, we have  
\mbox{(1) $\displaystyle\bigvee_{y} \mu  \langle b|y \rangle  = 1$.}
Chasing in the same way an element $(a,\, b)$ with $a \neq b$, but this time using the arrow $\iota \otimes G$, it follows 
\mbox{$\displaystyle \bigvee_{x} \iota  \mu  \langle a|x \rangle  \wedge \mu  \langle x|b \rangle  = 0$. Thus}
\mbox{(2) $\iota  \mu  \langle a|x \rangle  \wedge \mu  \langle x|b \rangle  = 0$ for all $x$.}

\vspace{1ex}

We will see now that $\iota \mu \leq \mu \iota$ (since $\iota^2 = id$, it follows that also $\mu \iota \leq \iota \mu$).

\vspace{1ex}
 
$ \iota  \mu  \langle a|b \rangle  \stackrel{(1)}{=} \iota  \mu  \langle a|b \rangle  \wedge \displaystyle\bigvee_{y} \mu  \langle b|y \rangle  = $
$ \displaystyle\bigvee_{y} \iota  \mu  \langle a|b \rangle  \wedge \mu  \langle b|y \rangle  \stackrel{(2)}{=} \iota  \mu  \langle a|b \rangle  \wedge \mu  \langle b|a \rangle$, 
since all the other terms in the supremum are $0$. 
Then $\iota  \mu  \langle a|b \rangle  \leq \mu  \langle b|a \rangle = \mu \iota \langle a|b \rangle $. 

\vspace{1ex}

Thus we have $\iota  \mu  \langle a|b \rangle  = \mu \iota \langle a|b \rangle \; ( = \mu  \langle b|a \rangle)$. With this, it is clear from the equations (1) and (2) above that the four axioms \ref{bijection} of an $\ell$-bijection hold. 
\end{proof}

\begin{proposition} 
There is a bijection between the objects of the categories $Cmd_0(G)$ and $Rel(\beta^G)$.
\end{proposition}
\begin{proof}
 Since $(\ell X)^\wedge = \ell X$, we have a bijection of linear maps
$$ 
\xymatrix@R=0.1pc{& \ell X \ar[r]^{\rho} & G \otimes \ell X \\
	    \ar@{-}[rrr] & & & \\
	     & \ell X \otimes \ell X \ar[r]^{\mu } & G.}
$$
As with every duality $(\varepsilon$, $\eta$), $\mu $ is defined as the composition 
$$
 \xymatrix { \mu : \ell X \otimes \ell X \ar[r]^{\rho \otimes \ell X} & G \otimes \ell X \otimes \ell X \ar[r]^(.65){G \otimes \varepsilon} & G.}
$$
And conversely, we construct $\rho$ as the composition
$$
 \xymatrix { \rho : \ell X \ar[r]^(.35){\ell X \otimes \eta} & \ell X \otimes \ell X \otimes \ell X \ar[r]^(.6){\mu  \otimes \ell X} & G \otimes \ell X.}
$$
It is easy to check (for example, using the elevators calculus) that 
that $\rho$ satisfies equations \ref{coaction} if and only if $\mu$ satisfies the first two equations \ref{morfismo} (by proposition \ref{monoidaction}, such a $\mu$ satisfies also the third equation).  
\end{proof}

The product of two $G$-sets $X$ and $X'$ is equipped with the action given by the product $\ell$-relation $\mu \boxtimes \mu'$ (\ref{neutralproductrelationdef}),  which is an action by proposition \ref{neutralproductrelation}.
 
\vspace{1ex}

An arrow of the category $Rel(\beta^G)$ is a monomorphism 
$R \hookrightarrow X \times X'$, in particular, a relation of sets
$R \subset X \times X'$. 
It follows from propositions \ref{mono1} and \ref{mono2}, that a relation $R \hookrightarrow X \times X'$ in the category $\beta^G$ is the same thing that a relation of sets $R \subset X \times X'$ such that the restriction of the product action to $R$ is still an $\ell$-bijection (on $R$). The following proposition finishes the proof of theorem  \ref{neutralComd=Rel}.
\begin{proposition} \label{conclaim}
Let $X$, $X'$ be any two $G$-sets, and $R \subset X \times X'$ a relation on the underlying sets. Then, $R$ underlines a monomorphism of \mbox{$G$-sets} $R \hookrightarrow X \times X'$ if and only if the corresponding linear map 
$R: \ell X \rightarrow \ell X'$ is a comodule morphism.
\end{proposition}
\begin{proof}
Let $\theta$ be the restriction of the product action $\mu \times \mu'$ to $R$. We claim that the diagram expressing that $R: \ell X \rightarrow \ell X'$ is a comodule morphism is equivalent to the diagram $\lozenge(R,R)$ in \ref{neutraldiagramadiamante12}. The proof follows then by 
\mbox{proposition \ref{neutralcombinacion}.}

 \vspace{1ex}
 
{\it proof of the claim}: 
It can be done by chasing elements in the diagrams, or more generally by using the elevators calculus explained in appendix \ref{ascensores}:

The comodule morphism diagram is the equality
\begin{equation} \label{neutralcommorph}
\xymatrix@C=-0.3pc@R=1.5pc
          {  
                   \ell X  \ar@2{-}[d] 
           & & &           \ar@{-}[dl] 
                           \ar@{}[d]|{\eta} 
                           \ar@{-}[dr] 
           & & &   \ell X  \ar@<4pt>@{-}'+<0pt,-6pt>[d]  
                           \ar@<-4pt>@{-}'+<0pt,-6pt>[d]^{R}	
           & & &           \ar@{-}[dl] 
                           \ar@{}[d]|{\eta} 
                           \ar@{-}[dr] 
           \\
				  \ell X   \ar@{-}[dr] 
		   &               \ar@{}[d]|{\mu_X } 
		   &      \ell X   \ar@{-}[dl] 
		   & &    \ell X   \ar@<4pt>@{-}'+<0pt,-6pt>[d] 
		                   \ar@<-4pt>@{-}'+<0pt,-6pt>[d]^{R}	
		   &      \hspace{3ex} = \hspace{3ex}
		   &	  \ell X'  \ar@{-}[dr] 
		   &               \ar@{}[d]|{\mu_{X'} } 
		   &      \ell X'  \ar@{-}[dl] 
		   & &    \ell X'  \ar@2{-}[d]
		   \\
		   &      G 
		   & & &  \ell X' 															   & & &  G 
		   & & &  \ell X', 
		  }
\end{equation}
while the diagram $\lozenge$ is
\begin{equation} \label{neutralgrafdiam}
\xymatrix@C=-0.3pc@R=1.5pc
         {  
                 \ell X  \ar@2{-}[d] 
          & & &          \ar@{-}[dl] 
                         \ar@{}[d]|{\eta} 
                         \ar@{-}[dr] 
          & & &  \ell X' \ar@2{-}[d]							
          & & 	 \ell X  \ar@<4pt>@{-}'+<0pt,-6pt>[d] 
                         \ar@<-4pt>@{-}'+<0pt,-6pt>[d]^{R}	
          & &    \ell X' \ar@2{-}[d]
          \\
				 \ell X  \ar@2{-}[d] 
	      & &    \ell X  \ar@2{-}[d] 
	      & &    \ell X  \ar@<4pt>@{-}'+<0pt,-6pt>[d] 
	                     \ar@<-4pt>@{-}'+<0pt,-6pt>[d]^{R}
	      & &    \ell X' \ar@2{-}[d] 	
	      &       \hspace{3ex} = \hspace{3ex}  
	      & 	 \ell X' \ar@{-}[dr] 
	      &              \ar@{}[d]|{\mu_{X'} } 
	      &      \ell X' \ar@{-}[dl]
	      \\
				 \ell X  \ar@{-}[dr] 
		  &              \ar@{}[d]|{\mu_X } 
		  &      \ell X  \ar@{-}[dl] 
		  & &    \ell X' \ar@{-}[dr] 
		  &              \ar@{}[d]|{\varepsilon} 
		  &      \ell X' \ar@{-}[dl] 		
		  & & &   G.
		  \\
		  &       G 
		  & & & & &	
		 }
\end{equation}
Recall that the triangular equations of a duality pairing are:
\begin{equation} \label{neutraltriangular}
\xymatrix@C=-0.3pc@R=0.1pc
         { 			
          &  \ar@{-}[ldd] \ar@{-}[rdd] \ar@{}[dd]|{\eta} 
          & & &  X \ar@2{-}[dd] 					
          & & & & & & & & &  Y \ar@2{-}[dd] 
          & & & \ar@{-}[ldd] \ar@{-}[rdd] \ar@{}[dd]|{\eta} 
          & & & &  
          \\
		  & & & & & & & X \ar@2{-}[dd] 
		  & & & & & & & & & & & & &  Y \ar@2{-}[dd] 
		  \\
		    X \ar@2{-}[dd] 
		  & & Y \ar@{-}[rdd] 
		  & \ar@{}[dd]|{\varepsilon} 
		  & X \ar@{-}[ldd] 
		  & \textcolor{white}{X} = \textcolor{white}{X} 
		  & & & & & \textcolor{white}{XX} \hbox{and} \textcolor{white}{XX} 
		  & & & Y \ar@{-}[ddr] 
		  & \ar@{}[dd]|{\varepsilon} 
		  & X \ar@{-}[ldd] 
		  & & Y \ar@2{-}[dd] 
		  & \textcolor{white}{X} = \textcolor{white}{X} 
		  & &  
		  \\
		  & & & & & & & X & & & & & & & & & & & & &  Y. 
		  \\
			X 
		  & & & & & & & & & & & & & & & & & Y	
		 }
\end{equation}

{\it Proof of \eqref{neutralcommorph} $\implies$ \eqref{neutralgrafdiam}}:
$$
\xymatrix@C=-0.3pc@R=1.5pc
          {
                   \ell X  \ar@2{-}[d] 
           & & &           \ar@{-}[ld] 
                           \ar@{-}[rd] 
                           \ar@{}[d]|{\eta} 
           & & &   \ell X' \ar@2{-}[d] 						
           & & 	   \ell X  \ar@2{-}[d] 
           & & &           \ar@{-}[ld] 
                           \ar@{-}[rd] 
                           \ar@{}[d]|{\eta} 
           & & &   \ell X' \ar@2{-}[d] 		
           & & 	   \ell X  \dr{R} 
           & & &           \ar@{-}[ld] 
                           \ar@{-}[rd] 
                           \ar@{}[d]|{\eta} 
           & & &   \ell X' \ar@2{-}[d] 					
           & & 
           \\
                   \ell X  \ar@2{-}[d] 
           & &     \ell X  \ar@2{-}[d] 
           & &     \ell X  \dr{R} 
           & &     \ell X' \ar@2{-}[d] 							
           & 
           \ \ \ 
                           \ar@{}[d]|= 
           &	   \ell X  \ar@{-}[dr] 
           &               \ar@{}[d]|{\mu_X} 
           &       \ell X  \ar@{-}[dl] 
           & &     \ell X  \dr{R} 
           & &     \ell X' \ar@2{-}[d] 	
           & 
           \ \ \ 
                           \ar@{}[d]|{\stackrel{(\ref{neutralcommorph})}{=}} 
           &	   \ell X' \ar@{-}[dr] 
           &               \ar@{}[d]|{\mu_{X'}} 
           &       \ell X' \ar@{-}[dl] 
           & &     \ell X' \ar@2{-}[d] 
           & &     \ell X' \ar@2{-}[d]	
           & 
           \ \ \ 
                   \ar@{}[d]|= 
           & 
           \\
                   \ell X  \ar@{-}[dr] 
           &               \ar@{}[d]|{\mu_X} 
           &       \ell X  \ar@{-}[dl] 
           & &     \ell X' \ar@{-}[dr] 
           &               \ar@{}[d]|{\varepsilon}
           &       \ell X' \ar@{-}[dl] 	
           & & &    G 
           & & &   \ell X' \ar@{-}[dr] 
           &               \ar@{}[d]|{\varepsilon} 
           &       \ell X' \ar@{-}[dl]				
           & & &    G 
           & & &   \ell X' \ar@{-}[dr] 
           &               \ar@{}[d]|{\varepsilon} 
           &       \ell X' \ar@{-}[dl] 						
           & & 
           \\
           &        G 
           & & & & & 																   & & & & & & & & & & & & & & & & & & 
          }
$$
$$\xymatrix@C=-0.3pc @R=1.5pc{
& &   			\ell X \dr{R} & & & \ar@{-}[ld] \ar@{-}[rd] \ar@{}[d]|{\eta} & & & \ell X' \ar@2{-}[d] 					& & 						\ell X \dr{R} & & \ell X' \ar@2{-}[d]					& \ \ \ \ar@{}[dd]|= &	\ell X \dr{R} & & \ell X' \ar@2{-}[d]\\
& \ \ \ \ar@{}[d]|= & 	\ell X' \ar@2{-}[d] & & \ell X' \ar@2{-}[d] & & \ell X' \ar@{-}[dr] & \ar@{}[d]|{\varepsilon} & \ell X' \ar@{-}[dl] 	& \ \ \ \ar@{}[d]|{\stackrel{(\triangle)}{=}} & \ell X' \ar@2{-}[d] & & \ell X' \ar@2{-}[d] 				& & 			\ell X' \ar@{-}[dr] & \ar@{}[d]|{\mu_{X'}} & \ell X' \ar@{-}[dl] \\
& &		 	\ell X' \ar@{-}[dr] & \ar@{}[d]|{\mu_{X'}} & \ell X' \ar@{-}[dl] & & & & 						& & 						\ell X' \ar@{-}[dr] & \ar@{}[d]|{\mu_{X'}} & \ell X' \ar@{-}[dl]	& & 			& G. & \\
& & 			& G & & & & &  													& & 						& G &  									& & 			& & }
$$

{\it Proof of \eqref{neutralgrafdiam} $\implies$ \eqref{neutralcommorph}}:		
$$\xymatrix@C=-0.2pc @R=1pc{
\X \dr{R} & & & \op{\eta} &	& & 		\X \dd & & & \op{\eta} &		& & 		\X \dd & & & & & & & \op{\eta} &				& & 	\\
\X' & & \X' & & \X' \dd		& \ig{} &	\X \dr{R} & & \X' \dd & & \X' \dd	& \ig{(\ref{neutralgrafdiam})} &	\X \dd & & & \op{\eta} & & & \X' \dd & & \X' \dd		& \ig{} & \\
& G \cl{\mu_{X'}}& & & \X'	& & 		\X' & & \X' & & \X' \dd			& & 		\X \dd & & \X \dd & & \X \dr{R} & & \X' \dd & & \X \dd 		& & 	\\
& & & & 			& & 		& G \cl{\mu_{X'}} & & & \X'		& & 		\X & & \X & & \X' & & \X' & & \X' \dd 				& & 	\\
& & & & 			& &		& & & & 				& & 		& G \cl{\mu_X} & & & & \cl{\varepsilon} & & & \X'		& & }$$

$$\xymatrix@C=-0.2pc @R=1pc{
& 	& 	\X \dd & & & \op{\eta} & & & & & 			& & 			\X \dd & & & \op{\eta} & 		& & 		\X \dd & & & \op{\eta} & 	& & \\
& \ig{} & 	\X \dd & & \X \dd & & \X \dr{R} & & & & 		& \ig{(\triangle)} & 	\X \dd & & \X \dd & & \X \dr{R} 	& \ig{} &	\X & & \X & & \X \dr{R}		& & \\
& & 		\X & & \X & & \X' \dd & & & \op{\eta} & 		& & 			\X & & \X & & \X' \dd			& &		& G \cl{\mu_X} & & & \X'.	& & \\
& & 		& G \cl{\mu_X} & & & \X' & & \X' & & \X' \dd 	 	& & 			& G \cl{\mu_X} & & & \X' 		\\
& & 		& & & & & \cl{\varepsilon} & & & \X' }$$

\end{proof}

\pagebreak

\subsection{The Galois and the Tannakian contexts}

{\bf The Galois context.}

\begin{sinnadastandard} {\bf The localic group of automorphisms of a functor.}  \label{Aut(F)}
\end{sinnadastandard}
Let $\cc{E}ns \mr{f} \cc{E}$ be any pointed topos, with inverse image \mbox{$f^* = F$, $\cc{E} \mr{F} \cc{E}ns$.} 
The localic group of automorphisms of $F$ is defined to be the universal  $\rhd$-cone of $\ell$-bijections in the category of locales, as described in the following diagram (see \cite{D1}):
\begin{equation} \label{neutralAutF}
\xymatrix
        {
         FX \times FX \ar[rd]^{\lambda_{X}}  
                      \ar@(r, ul)[rrd]^{\phi_{X}} 
                      \ar[dd]_{F(f) \times F(f)}  
        \\
         {} \ar@{}[r]^(.3){\geq}
         & \;\; Aut(F) \;\; \ar@{-->}[r]^{\phi} 
         & \;H.  
        \\
         FY \times FY \ar[ru]^{\lambda_{Y}} 
         \ar@(r, dl)[rru]^{\phi_{Y}} 
         && {(\phi \; \text{a locale morphism})}
        } 
\end{equation}
From propositions \ref{neutraltrianguloesdiamante} and \ref{neutralextension} it immediately follows
\begin{proposition}
The localic group $Aut(F)$ exists and it is isomorphic to the localic group of automorphisms of the restriction of $F$ to any small site of definition for $\cc{E}$. \cqd
\end{proposition}

A point $Aut(F) \mr{\phi} 2$ 
corresponds exactly to the data defining a natural isomorphism of $F$.  
Given $(a,\,b) \in FX \times FX$, 
we will denote \mbox{$\langle X, \,a|b \rangle  = \lambda_X(a,\,b)$.} This element of $Aut(F)$ corresponds to the open set  $\{\phi \,|\, \phi_X(a) = b\}$ of the subbase for the product topology in the set of natural isomorphisms of $F$. For details of the construction of this locale see \cite{D1}.

\vspace{1ex}
 
The $\ell$-bijections $\lambda_X$ determine morphisms of locales 
\mbox{$Aut(FX) \mr{\mu_X} Aut(F)$,} \linebreak
\mbox{$\mu_X \langle a|b \rangle = \langle X, \,a|b \rangle$.}
It is straightforward to check that 
the following three families of arrows are $\rhd$-cones of  
$\ell$-bijections:
\begin{multline}  \label{groupAut}
FX \times FX \mr{w_X} Aut(F) \otimes Aut(F),\;\; 
w_X(a,\,b) = \displaystyle\bigvee_{x \in FX} 
\langle X, \,a|x \rangle \otimes \langle X, \,x|b \rangle,  
\\
\hspace{-26ex} FX \times FX \mr{\iota_X} Aut(F), \;\;
\iota_X(a,\,b) = \langle X, \,b|a \rangle,
\\
\hspace{-36ex} FX \times FX \mr{e_X} 2, \;\; e_X(a,\,b) = \delta_{a=b}.
\\
\end{multline}
By the universal property they determine localic morphisms with domain $Aut(F)$ which define a localic group structure on $Aut(F)$, such that 
$\mu_X$ becomes an action of $Aut(F)$ on $FX$, and such that for any $X \mr{f} Y \in \cc{E}$, $F(f)$ is a morphism of actions. 
In this way there is a lifting $\tilde{F}$  of the functor $F$ into $\beta^G$,  
$\cc{E} \mr{\tilde{F}} \beta^G$, for $G = Aut(F)$.

 \begin{sinnadastandard} {\bf The (Neutral) Tannakian context associated to pointed topos.} \label{neutraltannakacontext}

For generalities, notation and terminology  concerning Tannaka theory see appendix \ref{appendix}. We consider a  topos with a point $\cc{E}ns \mr{f} \cc{E}$, with inverse image $f^* = F$, $\cc{E} \mr{F} \cc{E}ns$. 
We have a diagram (see \ref{neutralrel}):
$$
\xymatrix
        {
           \cc{E} \ar[r] \ar[d]_F
         & \cc{R}el(\cc{E}) \ar[d]^{\cc{R}el(F)}
         \\
           \cc{E}ns \ar[r] 
         & \cc{R}el
         & \hspace{-7ex} = s \ell_0 
        }
$$
This determines a Tannakian context as in appendix \ref{appendix}, with 
$\cc{X} = \cc{R}el(\cc{E})$, $\cc{V} = s \ell$, $\cc{V}_0 = \cc{R}el = s \ell_0$ and $T = \cc{R}el(F)$. Furthermore, in this case $\cc{X}$, $\Vat$ are symmetric, $T$ is monoidal (\ref{tensoriso}, \ref{neutralrel}),
and every object of $\cc{X}$ has a right dual. Thus, the (large) coend $End^\lor(T)$ (which exists, as we shall see) is a (commutative) Hopf algebra (proposition \ref{neutralhopf}). 
\end{sinnadastandard}

The universal property which defines the coend $End^\lor(T)$ is that of a universal $\lozenge$-cone in the category of sup-lattices, as described in the following diagram: 
$$
\xymatrix
        {
         & TX \times TX \ar[rd]^{\lambda_{X}}  
                      \ar@(r, ul)[rrd]^{\phi_{X}}  
        \\
         TX \times TY \ar[rd]_{TR \times TY} 
			           \ar[ru]^{TX \times TR^{op}} 
	     & \equiv 
	     & \;\;End^\lor(T)\;\; \ar@{-->}[r]^{\phi}  
         & \;H.  
        \\
         & TY \times TY \ar[ru]^{\lambda_{Y}} 
                          \ar@(r, dl)[rru]^{\phi_{Y}} 
         && {(\phi \; \text{a linear map})}
        } 
$$

Given $(a,\,b) \in TX \!\times TX$, 
we will denote \mbox{$[X, \,a,b] = \lambda_X \langle a,\,b \rangle$.} 

From proposition \ref{neutralextension} and \ref{neutraldim1ydim2esdim} it immediately follows:
\begin{proposition} \label{neutralEndhasX}
The large coend defining $End^\lor(T)$ exists and can be computed by the coend corresponding to the restriction of $\, T$ to the full subcategory of $Rel(\Eat)$ whose objects are in any small site $\Cat$ of definition of $\cc{E}$. \cqd
\end{proposition}

By the general Tannaka theory we know that the sup-lattice $End^\lor(T)$ is a Hopf algebra in $s \ell$. The description of the multiplication $m$ and a unit $u$ given below proposition \ref{neutralbialg} yields in this case, for $X, \, Y \in \cc{X}$ (here, $F(1_\Cat) = 1_{\Ens} = \{*\}$):
\begin{equation} 
m([X,\, a,a'],\, [Y, \,b,b']) \;= \; [X \times Y,\, (a,\,b),(a',\,b')],\; \;\; u(1) \;=\; [1_\Cat, \, *,*] .
\end{equation} 
This shows that $TX \times TX \mr{\lambda_X} End^\lor(T)$ is a compatible 
$\lozenge$-cone, thus by proposition \ref{neutralcompislocale} it follows that $End^\lor(T)$ is a locale,  with top element $[1_\Cat, \, *,*]$ and infimum 
\mbox{$[X,\, a,a'] \wedge [Y, \,b,b'] =  [X \times Y,\, (a,\,b),(a',\,b')]$.}

\vspace{1ex}

We let the reader check the following:
\begin{sinnadaitalica} \label{groupEnd}
The descriptions in the general Tannaka theory of the comultiplication $w$, the counit $\varepsilon$ and the antipode $\iota$ (see appendix \ref{appendix}) yield in this case the formulae 
$$w_X(a,\,b) = \displaystyle\bigvee_{x \in FX} 
[X, \,a,x] \otimes [X, \,x,b], \quad \iota_X(a,\,b) = [X, \,b,a] \hbox{ and }
\varepsilon_X(a,\,b) = \delta_{a=b}.$$ \cqd 
\end{sinnadaitalica} 

\vspace{1ex}

\begin{sinnadastandard} 
{\bf The isomorphism $End^\lor(T) \cong Aut(F)$.}
\end{sinnadastandard} 

From propositions  \ref{neutraltrianguloesdiamante} and  \ref{neutralDiamondisbijection} it immediately follows (recall that T = F on $\cc{E}$) that $TX \times TX \mr{\lambda_X} Aut(F)$ and 
$TX \times TX \mr{\lambda_X} End^\lor(T)$ are both $\rhd$-cones and $\lozenge$-cones of $\ell$-bijections. From  proposition \ref{neutralsupisloc} and the respective universal properties it follows that they are isomorphic locales respecting the cone maps $\lambda_X$. Furthermore, by the formulae in \ref{groupAut} and \ref{groupEnd} we see that under this isomorphism the comultiplication, counit and antipode correspond. Thus, we have: 
\begin{theorem} \label{neutralG=H}
Given any pointed topos, there is a unique isomorphism of localic groups $End^\lor(T) \cong Aut(F)$ commuting with the $\lambda_X$. \cqd  
\end{theorem}

\pagebreak

\subsection{The main Theorems}
A pointed topos $\cc{E}ns \mr{f} \cc{E}$, with inverse image $f^* = F$, $\cc{E} \mr{F} \cc{E}ns$, determines a situation described in the following \mbox{diagram:}
\begin{equation} \label{neutraldiagramafundamental}
\xymatrix
        {
          \beta^G            \ar[r] \ar[rdd]
        & \cc{R}el(\beta^G)  \ar[r]^{=}
        & Cmd_0(G)          \ar[r]^{=}
        & Cmd_0(H)          \ar[ldd]
        \\
        & \cc{E}             \ar[d]^F \ar[r] \ar[lu]_{\widetilde{F}}
        & \cc{R}el(\cc{E})   \ar[ul]_{\cc{R}el(\widetilde{F})}  
                             \ar[d]^T \ar[ur]^{\widetilde{T}}
        \\
        & \cc{E}ns            \ar[r]
        & \cc{R}el 
        &  \hspace{-8ex} = s \ell_0 \subset s \ell.
       }
\end{equation}

\vspace{1ex}

\noindent
where $G = Aut(F)$, $T = \cc{R}el(F)$, $H = End^\lor(T)$ and
the two isomorphisms in the first row of the diagram are given by Theorems \ref{neutralComd=Rel} and \ref{neutralG=H}. 

\begin{theorem} \label{neutralAA}
The (Galois) lifting functor $\widetilde{F}$ is an equivalence if and only if the (Tannaka) lifting functor $\widetilde{T}$ is such. \cqd
\end{theorem} 
 Assume now that $\cc{E}$ is a connected atomic topos. The full subcategory of connected objects $\cc{C} \subset \cc{E}$ furnished with the canonical topology is a small site for $\cc{E}$. In \cite{D1} it is proved that the diagram of the functor $F$ restricted to this site $\cc{C} \mr{F} \cc{E}ns$ is a poset (\emph{This fact distinguishes atomic topoi from general locally connected topoi}),   
an explicit construction of $Aut(F)$ is given, and the following key result of localic Galois Theory is proved:
\begin{theorem}[\cite{D1} 6.9, 6.11] \label{keyLGT} ${}$

\vspace{1ex}

1) For any $C \in \cc{C}$ and $(a,\,b) \in FC \times FC$, 
$\; \langle C, \,a|b \rangle \neq 0$. 

\vspace{1ex}

2) Given any other $(a',\,b') \in FC' \times FC'$, if 
$\langle C, \,a|b \rangle \leq \langle C', \,a'|b' \rangle$, then there exists $C \mr{f} C'$ in $\cc{C}$ such that $a' = F(f)(a)$, $b' = F(f)(b)$.
\end{theorem}

\noindent The following theorem follows from \ref{keyLGT} by a formal topos theoretic reasoning. 

\begin{theorem}[\cite{D1} 8.3] \label{fundamentalLGT}
The (Galois) lifting functor $\widetilde{F}$ is an equivalence if and only if the topos $\cc{E}$ is connected atomic.\cqd
\end{theorem}
From \ref{neutralAA} and \ref{fundamentalLGT} we have:
\begin{theorem} \label{BB}
The (Tannaka) lifting functor $\widetilde{T}$ is an equivalence if and only if the topos $\cc{E}$ is connected atomic. \cqd
\end{theorem}

\pagebreak

\section{Relations, functions and bijections in a topos} \label{relintopos}

We begin this section with the first previous steps necessary to develop the results of section \ref{sec:Neutral} over an arbitrary elementary topos.

Following Joyal and Tierney in \cite{JT}, Introduction p.vii, we fix an elementary topos $\Sat$ (with subobject classifier $\Omega$), and work in this universe in the internal language of this topos, as we would in naive set theory (but without axiom of choice or law of the excluded middle). This means for us that:

\begin{itemize}
 \item We are able to consider elements $x$ of objects $X$ of $\Sat$, a situation which we denote by $x \in X$, and to apply arrows $X \mr{f} Y$ to obtain $f(x) \in Y$. 
 \item We may consider equality of such elements, given by the characteristic function $X \times X \mr{\delta_X} \Omega$ of the diagonal $X \mr{\triangle} X \times X$. We will denote $\llbracket x \! = \! x' \rrbracket = \delta(x,x') \in \Omega$, for $x,x' \in X$.
 \item A internal structure for an object $X$ in $\Sat$ becomes in this way a structure as in set theory, for example we think in this way of the Heyting algebra structure of $\Omega$ (see \cite{Johnstone}, 5.13)  
  \item Also following the work of \cite{JT}, we won't make the distinction between elements of $X^Y$ and arrows $Y \mr{} X$ in $\Sat$, though of course there is. More precisely, to prove a statement about the elements of $X^Y$, we will consider arrows $Y \mr{} X$.
 
\end{itemize}

We don't claim originality of the main results of this section up to \ref{aplication} (that $\Omega^X$ is the free sup-lattice on $X$, that functional relations correspond to arrows of $\Sat$), as most of these can be found in the references or are folkloric, but we make nevertheless a complete development including full (sketches of the) proofs since we do these in a way that can be generalized to any sup-lattice (or locale) $G$ in place of $\Omega$ in section \ref{sec:ellrel}.

\begin{remark} \label{prelim1}
 Given $x, y \in X$, $f \in Y^X$, we have $\llbracket x \! = \! y\rrbracket  \leq \llbracket f(x) \! = \! f(y) \rrbracket $. $f$ is a monomorphism if and only if the equality holds (for every $x,y \in X$).
\end{remark}

\begin{proof}
 Consider the commutative diagram $\underline{1}:$ $\vcenter{\xymatrix{X \times X \ar[r]^{(f,f)} \ar@{}[dr]|{\underline{2}} & Y \times Y \ar[r]^{\delta_Y} \ar@{}[dr]|{p.b.} & \Omega \\
					  X \ar[u]_{\triangle} \ar[r]^f & Y \ar[u]_{\triangle} \ar[r] & 1 \ar[u]_1   }}$ It yields a inclusion of subobjects of $X \times X$ that is carried via the isomorphism \linebreak $Sub(X \times X) \cong [X \times X, \Omega]$ 
					  to the desired inequality $\delta_X \leq \delta_Y \circ (f,f)$.
					  
 Also, $f$ is a monomorphism if and only if $\underline{2}$ is a pull-back, if and only if $\underline{1}$ is so, if and only if $\delta_X = \delta_Y \circ (f,f)$.
\end{proof}

\begin{\prop} \label{prelim2}
For each $a,b \in \Omega$, $a \wedge \llbracket a \! = \! b \rrbracket  \leq b$.
\end{\prop}

\begin{proof}
 Consider, as in \cite{Johnstone} p.137, $\Omega_1 \mr{} \Omega \times \Omega$ the equalizer of $\wedge$ and $\pi_1$, with classifying map $\Rightarrow$, and denote $\llbracket a \! \leq \! b \rrbracket = \ \ \Rightarrow \! \! (a,b)$. Then it is enough to show that $(*): \llbracket a \! = \! b \rrbracket  \leq \llbracket a \! \leq \! b \rrbracket$ since by the adjunction $\wedge \dashv \ \ \Rightarrow$ we have $a \wedge \llbracket a \! \leq \! b \rrbracket \leq b$. 
 
 In fact (*) holds because $\Omega \mr{\triangle} \Omega \times \Omega$ is contained in $\Omega_1 \mr{} \Omega \times \Omega$ as subobjects of $\Omega \times \Omega$, since $\triangle$ equalizes $\wedge$ and $\pi_1$. 
   This yields the inequality of the characteristic functions $\llbracket (-) \! = \! (-) \rrbracket  \leq \llbracket (-) \! \leq \! (-) \rrbracket$.
\end{proof}

\begin{sinnadastandard} \label{11} {\bf Relations in a topos.} 

 A \emph{relation} between $X$ and $Y$ is a subobject $R \mmr{} X \times Y$ or, equivalently, an arrow $X \times Y \mr{\lambda} \Omega$. Relations are composed using pullbacks and image factorizations in $\cc{S}$, or equivalently if they are given as $X \times Y \mr{\lambda} \Omega$, $Y \times Z \mr{\mu} \Omega$ by the formula \linebreak $c(x,z) = \bigvee_y \lambda(x,y) \wedge \mu(y,z)$, i.e. \emph{matrix multiplication}. This yields a category \linebreak $Rel = Rel(\cc{S})$ of relations in $\cc{S}$. We have the following correspondence:

\begin{equation} \label{tabla1}
\begin{tabular}{ccc} 
&&  \textit{(in particular for the diagonal $\triangle$)} \\
$R \hookrightarrow X \times Y$ a relation && $X \stackrel{\triangle}{\hookrightarrow} X \times X$ the identity relation \\ \cline{1-1} \cline{3-3} \noalign{\smallskip}
$X \times Y \mr{\lambda} \Omega$ a relation && $X \times X \mr{\delta} \Omega$ the identity relation\\ \cline{1-1} \cline{3-3} \noalign{\smallskip}  
$Y \mr{\lambda^*} \Omega^X$ its inverse image && $X \mr{ \{ \} } \Omega^X$ \\ \cline{1-1} \noalign{\smallskip} 
$X \mr{\lambda_*} \Omega^Y$ its direct image \\
\\
$\lambda^*(y)(x) = \lambda(x,y) = \lambda_*(x)(y)$ && $\{y\} (x) = \llbracket x \! = \! y \rrbracket = \{x\}(y)$ \\
\end{tabular}
\end{equation}

\begin{lemma} \label{prelim3}
 For each $x,x_1,x_2 \in X$, we have $$i) \bigvee_{y \in X} \igu[x]{y} = 1, \quad ii) \igu[x]{x_1} \wedge \igu[x]{x_2} \leq \igu[x_1]{x_2}.$$
\end{lemma}
\begin{proof}
 Only $ii)$ requires a proof. The pull-back of monomorphisms 
 $$\vcenter{\xymatrix{ X \ar[d]^{\triangle} \ar[r]^{\triangle} & X \times X \ar[d]^f \\
						      X \times X \ar[r]^g & X \times X \times X,} }$$ where $f(x,x_1)=(x,x_1,x)$ and $g(x,x_2)=(x,x,x_2)$, computes the intersection \linebreak 
\mbox{$\{(x,x_1,x_2)|x=x_1\} \cap \{(x,x_1,x_2)|x=x_2\} = \triangle(X) \subseteq X \times X \times X$.}

Now the commutative square $\vcenter{\xymatrix{ X \times X \times X \ar[r]^{\pi_{2,3}} & X \times X \ar[r]^{\delta_X} & \Omega \\
						  X \ar[u]_{\triangle} \ar[rr] && 1 \ar[u]_1   }}$
shows the inclusion of subobjects $\{(x,x_1,x_2)|x=x_1\} \cap \{(x,x_1,x_2)|x=x_2\} \subseteq \{(x,x_1,x_2)|x_1=x_2\}$ which, when translated to the characteristic functions, yields the desired inequality.
\end{proof}

 \begin{remark}
  Lemma \ref{prelim3} says that $X \times X \mr{\delta_X} \Omega$ is a \emph{function} in the sense of definition \ref{function} below. 
 \end{remark}

\begin{sinnadastandard} \label{onthestruct} {\bf On the structure of $\Omega^X$.} 

The \emph{power set} $PX = \Omega^X$ has the sup-lattice (and locale) structure given pointwise by the structure of $\Omega$ (\cite{JT}, I.1 p.1). Via the isomorphism  $s \ell \cong \Omega$-$Mod$ (\cite{JT}, II.1 Proposition 1 p.8), we obtain a $\Omega$-module structure for $\Omega^X$ which is given by the canonical isomorphism $\Omega \otimes \Omega^X \mr{\cdot} \Omega^X$, $(a \cdot \theta) (x) = a \wedge \theta(x)$.

\begin{lemma} \label{lema1}
 For each $\theta \in \Omega^X$, $x,y \in X$, we have $\theta(x) \wedge \llbracket x \! = \! y\rrbracket  \leq \theta(y)$
\end{lemma}
\begin{proof} Recall remark \ref{numeroarriba}. 
  $\theta(x) \wedge \llbracket x \! = \! y\rrbracket  \stackrel{\ref{prelim1}}{\leq} \theta(x) \wedge \llbracket\theta(x) \! = \! \theta(y)\rrbracket  \stackrel{\ref{prelim2}}{\leq} \theta(y)$.
\end{proof}

\begin{\prop} \label{formulainterna}
For each $\theta \in \Omega^X$, $\displaystyle \theta = \bigvee_{x \in X} \theta(x) \cdot \{x\}$. This shows how any arrow $X \mr{f} M$ into a $\Omega$-module can be extended uniquely to $\Omega^X$ as $\displaystyle f(\theta) = \bigvee_{x \in X} \theta(x) \cdot f(x)$, so the singleton $X \mr{ \{ \} } \Omega^X$ is a free $\Omega$-module structure on $X$ (i.e. a free sup-lattice structure, cf. \cite{JT}, II.1 p.8).
\end{\prop}

\begin{proof}
 We can show the equality pointwise, we have to show that for each $y \in X$, $\theta(y) = \bigvee_{x \in X} (\theta(x) \cdot \{x\})(y) = \bigvee_{x \in X} \theta(x) \wedge \igu[x]{y} $. The inequality $\geq$ is given by lemma \ref{lema1}, and the inequality $\leq$ is obtained by taking $x=y$.
\end{proof}

\begin{lemma} \label{coro2}
 For each, $x,y \in X$, we have $\igu[x]{y}  \cdot \{x\} \leq \{y\}$ in $\Omega^X$.
\end{lemma}
\begin{proof}
 We can show the inequality pointwise, for each $z \in X$, $$(\igu[x]{y} \cdot \{x\}) (z) = \igu[x]{y} \wedge \igu[x]{z} \stackrel{\ref{prelim3}}{\leq} \igu[y]{z} = \{y\}(z).$$
\end{proof}

The following lemma will be the key for many following computations.
\begin{lemma} \label{ecuacionenL}
If $L$ is a $\Omega$-module (i.e. a sup-lattice), then any arrow $f \in L^X$ satisfies $$\llbracket x \! = \! y\rrbracket  \cdot f(x) = \llbracket x \! = \! y\rrbracket  \cdot f(y).$$
\end{lemma}

\begin{proof}
By symmetry it is enough to show that $\llbracket x \! = \! y\rrbracket  \cdot f(x) \leq \llbracket x \! = \! y\rrbracket  \cdot f(y)$, i.e. \linebreak $\llbracket x \! = \! y\rrbracket  \cdot f(x) \leq f(y)$. Consider the extension $\Omega^X \mr{f} L$ as a $\Omega$-module morphism given by proposition \ref{formulainterna}. Then
$$\llbracket x \! = \! y\rrbracket  \cdot f(x) = f(\llbracket x \! = \! y\rrbracket  \cdot \{x\}) \stackrel{\ref{coro2}}{\leq} f(\{y\}) = f(y).$$
\end{proof}

\begin{remark}\label{omegaenL}
If $L$ is a locale, we have a unique locale morphism $\Omega \mr{} L$ (\cite{JT}, II.1 p.8), then we can think that the elements of $\Omega$ are in $L$, in a way that is compatible with the structure of $L$ (like we think of $\Z$ in any ring $R$).
\end{remark}

\begin{corollary} \label{uvuv'}
Given any arrow $Y \mr{f} L$ into a locale $L$, we have:
$$
f(x) \wedge f(y) \leq \llbracket x \! = \! y\rrbracket  \;\; \iff  \;\; 
f(x) \wedge f(y) = \llbracket x \! = \! y\rrbracket  \cdot f(x) = \llbracket x \! = \! y\rrbracket  \cdot f(y).
$$ \qed
\end{corollary}

\begin{lemma} \label{reldomegaY}
The singleton arrow  $Y \mr{\{\}} \Omega^Y$ determines a \emph{presentation} of the locale $\Omega^Y$ in the following sense:
$$
 i)\;  1 = \displaystyle \bigvee_{y \in Y} \{y\}, \quad \quad \quad 
 ii) \; \{x\} \wedge \{y\} \leq \llbracket x \! = \! y\rrbracket .
$$
Given any other arrow $Y \mr{f} L$ into a locale $L$ such that: 
$$
 i) \; 1 = \bigvee_y f(y), \quad \quad \quad 
 ii)\; f(x) \wedge f(y) \leq \llbracket x \! = \! y\rrbracket 
$$
there exists a unique locale morphism $\Omega^Y \mr{f} L$ such that $f(\{y\}) = f(y)$.
\end{lemma}
\begin{proof}
Equations $i)$ and $ii)$ for $\{\}$, when considered pointwise, are the equations in lemma \ref{prelim3}.

Now, given $Y \mr{f} L$ by proposition \ref{formulainterna} there is a unique $\Omega$-module morphism $\Omega^Y \mr{f} L$ such that $f(\{y\}) = f(y)$. Since equation $i)$ holds in both locales, $f$ preserves $1$. Since equation $ii)$ holds in both locales and is equivalent to $f(x) \wedge f(y) = \llbracket x \! = \! y\rrbracket  \cdot f(x)$ by corollary $\ref{uvuv'}$, $f$ preserves $\wedge$.
\end{proof}

\begin{remark} \label{reldomegaYremark} 
By looking at the proof, we see that we have proved that given any arrow $Y \mr{f} L$ into a locale, its extension as a $\Omega$-module morphism to $\Omega^Y$ preserves $1$ if and only if equation $i)$ holds in $L$, and preserves $\wedge$ if and only if equation $ii)$ holds in $L$.
\end{remark}

\end{sinnadastandard}

\begin{sinnadastandard} {\bf The four axioms for relations.}

\noindent The following axioms for relations are considered in \cite{GW}, see also \cite{DSz} and compare with \cite{F} and \cite{McLarty}, 16.3.

\begin{\de} \label{4axioms}
A relation $X \times Y \mr{\lambda} \Omega$ is:

\vspace{1ex}

\noindent ed) Everywhere defined, if for each $x \in X$, $\displaystyle \bigvee_{y \in Y} \lambda(x,y) = 1$.

\vspace{1ex}

\noindent uv) Univalued, if for each $x \in X$, $y_1,y_2 \in Y$, $\lambda(x,y_1) \wedge \lambda(x,y_2) \leq \llbracket y_1 \! = \! y_2\rrbracket $.

\vspace{1ex}

\noindent su) Surjective, if for each $y \in Y$, $\displaystyle \bigvee_{x \in X} \lambda(x,y) = 1$.

\vspace{1ex}

\noindent in) Injective, if for each $y \in Y$, $x_1,x_2 \in X$, $\lambda(x_1,y) \wedge \lambda(x_2,y) \leq \llbracket x_1 \! = \! x_2\rrbracket $.

\end{\de}

\begin{remark} 
Notice the symmetry between ed) and su), and between uv) and in). Many times in this thesis we will work with axioms ed) and uv), but symmetric statements always hold with symmetric proofs. 
\end{remark}

\begin{remark} 
By corollary \ref{uvuv'} axiom uv) is equivalent to: 

\noindent \emph{uv) for each $x \in X$, $y_1,y_2 \in Y$, $\lambda(x,y_1) \wedge \lambda(x,y_2) = \llbracket y_1 \! = \! y_2\rrbracket  \cdot \lambda(x,y_1)$.}
\end{remark}

By remark \ref{reldomegaYremark}, we obtain: 

\begin{proposition} \label{edyuvdafunction}
Consider a relation $\lambda$ and its inverse image $\Omega^Y \mr{\lambda^*} \Omega^X$. $\lambda^*$ respects $1$ if and only if $\lambda$ satisfies axiom $ed)$, and $\lambda^*$ respects $\wedge$ if and only if $\lambda$ satisfies axiom $uv)$. 
\end{proposition}

\begin{proof}
Consider $Y \mr{\lambda^*} \Omega^X$. $\lambda$ is $ed)$ if and only if for each $x \in X$, $\bigvee_{y} \lambda^*(y)(x) = 1$. Since the structure of $\Omega^X$ is given pointwise, this happens if and only if $\bigvee_{y} \lambda^*(y) = 1$, which by remark \ref{reldomegaYremark} happens if and only if the extension $\Omega^Y \mr{\lambda^*} \Omega^X$ respects $1$. We have a similar situation for axiom $uv)$ that we leave for the reader to check. 
\end{proof}

\end{sinnadastandard}

\begin{sinnadastandard} \label{invdirimage} {\bf The inverse and the direct image of a relation.}
 
Using proposition \ref{formulainterna}, we can continue the correspondences of \eqref{tabla1} as

\begin{equation} \label{tabla2}
\begin{tabular}{c} 
 \\
$X \times Y \mr{\lambda} \Omega $ a relation \\   \hline \noalign{\smallskip}
$ \Omega^Y \mr{\lambda^*} \Omega^X $ a $\Omega$-module morphism \\ \hline \noalign{\smallskip} 
$ \Omega^X \mr{\lambda_*} \Omega^Y $ a $\Omega$-module morphism    \\
\\
$\lambda^*(\{y\})(x) = \lambda(x,y) = \lambda_*(\{x\})(y)$  \\
\end{tabular}
\end{equation}

$$\lambda^* (\{y\}) = \bigvee_{x \in X} \lambda(x,y) \cdot \{x\}, \quad \lambda_* (\{x\}) = \bigvee_{y \in Y} \lambda(x,y) \cdot \{y\}$$

\begin{sinnadastandard} \label{otro1erremarkbasico}
Given $X \times Y \mr{\lambda} \Omega$, $Y \times Z \mr{\mu} \Omega$, by \eqref{tabla2} the composition $\mu_*\lambda_*$ of their direct images maps $\{x\}$ to $\displaystyle \bigvee_{z \in Z} \bigvee_{y \in Y} \lambda(x,y) \wedge \mu(y,x) \cdot \{z\}$, which is the direct image of their composition as relations defined in \ref{11}. This yields a full-and-faithful inclusion functor 
$$\xymatrix@R=.2pc{ \quad \quad Rel \quad \quad \ar@{>->}[r]^{(-)_*} & \quad \quad  s \ell \quad \quad \\
	    X \times Y \mr{\lambda} \Omega \quad \ar@{|->}[r] & \quad \Omega^X \mr{\lambda_*} \Omega^Y  }$$
\end{sinnadastandard}

\begin{remark} \label{omegayproducto}
Note that the product of $\Sat$ is not a product in $Rel$, it is instead a tensor product that is mapped via this inclusion to the tensor product $\otimes$ of $s \ell$, since \linebreak $\Omega^{X \times Y} = \Omega^X \otimes \Omega^Y$.
\end{remark}

\end{sinnadastandard}

\begin{sinnadastandard} \label{arrowsvsfunctions} {\bf Arrows versus functions.} 

Consider an arrow $X \mr{f} Y$ in the topos $\Sat$. We define its \emph{graph} \linebreak $R_f = \{(x,y) \in X \times Y \ | \ f(x) = y\}$, and denote its characteristic function by $X \times Y \mr{\lambda_f} \Omega$, $\lambda_f(x,y) =  \igu[f(x)]{y}$.

\begin{remark} \label{1erremarkbasico}
 Using the previous constructions, we can form commutative diagrams
 
 $$\vcenter{\xymatrix{\Sat \ar[r]^{\lambda_{(-)}} \ar@/_1pc/[rr]_{P} & Rel \ar[r]^{(-)_*} & s \ell & \quad}}
 \vcenter{\xymatrix{\quad & \Sat \ar[r]^{\lambda_{(-)}} \ar@/_1pc/[rr]_{\Omega^{(-)}} & Rel \ar[r]^{(-)^*} & s \ell^{op} }     }$$
 
 Given $X \mr{f} Y$, $P(f)$ is the extension of $X \mr{f} Y \mr{\{\}} \Omega^Y$ to $\Omega^X$, and $\Omega^{f}: \Omega^Y \mr{} \Omega^X$ is given by precomposition with $f$. Then we have the commutativities because for each $x \in X$, $y \in Y$,
 
 $$(\lambda_f)_*(\{x\})(y) = \lambda_f(x,y) =  \llbracket f(x) \! = \! y\rrbracket  = \{f(x)\}(y) = P(f) (\{x\}) (y), \hbox{ and}$$
 
 $$(\lambda_f)^*(\{y\})(x) = \lambda_f(x,y) =  \llbracket f(x) \! = \! y\rrbracket  = \{y\}(f(x)) = \Omega^f (\{y\})(x).$$
  
 In other words, $P(f)$ is the direct image of (the graph of) $f$, and $\Omega^f$ is its inverse image. We will use the notations $f_* := P(f) = (\lambda_f)_*$, $f^* := \Omega^f = (\lambda_f)^*$.
\end{remark}

The relations which are the graphs of arrows of the topos are characterized as follows, for example in \cite{McLarty}, theorem 16.5. 

 \begin{proposition} \label{arrowsiipi1}
Consider a relation $X \times Y \mr{\lambda} \Omega$, the corresponding subobject \linebreak $R \hookrightarrow X \times Y$ and the arrows $R \mr{p} X$, $R \mr{q} Y$ obtained by composing with the projections from the product. There is an arrow $X \mr{f} Y$ of the topos such that  $\lambda = \lambda_f$ if and only if $p$ is an isomorphism, and in this case $f = q \circ p^{-1}$.  \qed
 \end{proposition}

\begin{remark} \label{hojafinal}
 Consider a subobject $A \hookrightarrow X$ with characteristic function $\alpha$, and let $Y \mr{f} X$. Then, by pasting the pull-backs, it follows that the characteristic function of $f^{-1}A$ is $\alpha \circ f$.
 This means that the square $\vcenter{\xymatrix{Sub(X) \ar@<1ex>[d]^{f^{-1}} \ar@{}[d]|{\dashv} \ar[r]^{\cong} & \Omega^X \ar@<1ex>[d]^{f^*} \ar@{}[d]|{\dashv} \\
					Sub(Y) \ar@<1ex>[u]^{Im_f} \ar[r]^{\cong} & \Omega^Y \ar@<1ex>[u]^{\exists_f}}}$ is commutative when considering the arrows going downwards, then also when considering the left adjoints going upwards.

In particular for a relation $X \times Y \mr{\lambda} \Omega$ with corresponding subobject $R \hookrightarrow X \times Y$, and the projection $X \times Y \mr{\pi_1} X$, the commutativity of the square $\vcenter{\xymatrix{Sub(X)  \ar[r]^{\cong} & \Omega^X  \\
					Sub(X \times Y) \ar@<1ex>[u]^{Im_{\pi_1}} \ar[r]^{\cong} & \Omega^{X \times Y} \ar@<1ex>[u]^{\exists_{\pi_1}}}}$ identifies $Im_{\pi_1}(R)$ with $\exists_{\pi_1} (\lambda)$, in particular $R \mr{p} X$ is an epimorphism if and only if $\exists_{\pi_1} (\lambda)(x) = 1$ for each $x \in X$.
\end{remark}

\begin{proposition} \label{parafunctionconallegories}
Consider a relation $X \times Y \mr{\lambda} \Omega$, the corresponding subobject \linebreak $R \hookrightarrow X \times Y$ and the arrow $R \mr{p} X$. $\lambda$ is (ed) if and only if $p$ is epi, and $\lambda$ is (uv) if and only if $p$ is mono. 
\end{proposition}

\begin{proof}

The quantifier $\exists_{\pi_1}$ of remark \ref{hojafinal} is given by the suprema $\bigvee_Y$ as follows: for each $\lambda \in \Omega^{X \times Y}$, $\alpha \in \Omega^X$

\begin{center}
\begin{tabular}{c}
 $\bigvee_{y \in Y} \lambda(-,y) \leq \alpha $ \\ \hline \noalign{\smallskip}
 for each $x \in X$, $\bigvee_{y \in Y} \lambda(x,y) \leq \alpha(x)$ \\ \hline \noalign{\smallskip}
 for each $x \in X, y \in Y$, $\lambda(x,y) \leq \alpha(x)$ \\ \hline \noalign{\smallskip}
 $\lambda \leq {\pi_1}^* (\alpha)$
 \end{tabular}
\end{center}

By unicity of the adjoint, we obtain for each $\lambda \in \Omega^{X \times Y}$, $x \in X$,

\begin{equation} \label{existslambda}
 \exists_{\pi_1} \lambda (x) = \bigvee_{y \in Y} \lambda(x,y)
\end{equation}

By remark \ref{hojafinal}, we conclude that $\lambda$ is (ed) if and only if $p$ is epi.

\vspace{.3cm}

Now, by remark \ref{hojafinal}, the characteristic functions of $(X \times \pi_1)^* R$, $(X \times \pi_2)^* R $ are $\lambda_1(x,y_1,y_2) = \lambda(x,y_1)$, $\lambda_2(x,y_1,y_2) = \lambda(x,y_2)$.

Then axiom uv) is equivalent to stating that for each $x \in X$, $y_1, y_2 \in Y$, 

$$\lambda_1(x,y_1,y_2) \wedge \lambda_2(x,y_1,y_2) \leq \llbracket y_1 \! = \! y_2 \rrbracket,$$

i.e. that we have an inclusion of subobjects of $X \times Y \times Y$

$$ (X \times \pi_1)^* R \cap (X \times \pi_2)^* R \subseteq X \times \triangle_Y. $$

But this inclusion is equivalent to stating that for each $x \in X$, $y_1, y_2 \in Y$, $(x, y_1) \in R$ and $(x,y_2) \in R$ imply that $y_1 = y_2$, i.e. that $p$ is mono.

\end{proof}

\begin{definition} \label{function}
We say that a relation $X \times Y \mr{\lambda} \Omega$ is a \emph{function} if it is uv) univalued and ed) everywhere defined.
\end{definition}

Combining proposition \ref{parafunctionconallegories} with \ref{arrowsiipi1}, we obtain 

\begin{proposition} \label{functionconallegories}
 A relation $\lambda$ is a function if and only if there is an arrow $f$ of the topos such that $\lambda = \lambda_f$. \qed
\end{proposition}

\begin{remark}\label{simetria}
A symmetric work shows that if we define \emph{op-functions} as those relations which are $in)$ injective and $su)$ surjective, then a relation $\lambda$ is an op-function if and only if $\lambda^{op}$ corresponds to an actual arrow in the topos.

If we now define \emph{bijections} as those relations that are simultaneusly functions and op-functions, that is, if they satisfy the four axioms \ref{4axioms}, then a relation $\lambda$ is a bijection if and only if there are two arrows in the topos such that $\lambda = \lambda_f$, $\lambda^{op} = \lambda_g$. Then we have that for each $x \in X$, $y \in Y$, 
$$\llbracket f(x) \! = \! y\rrbracket  = \lambda_f(x,y) = \lambda(x,y) = \lambda^{op} (y,x) = \lambda_g(y,x) = \llbracket g(y) \!  =  \! x\rrbracket ,$$
i.e. $f(x)=y$ if and only if $g(y)=x$, in particular $fg(y)=y$ and $gf(x)=x$, i.e.
 $f$ and $g$ are mutually inverse. In other words, bijections correspond to isomorphisms in the topos in the usual sense.
\end{remark}

\end{sinnadastandard}

\begin{sinnadastandard} \label{imageninversausandoautodual} {\bf The autoduality of $\Omega^X$.}
We show now that $\Omega^X$ is autodual as a sup-lattice (i.e. as a $\Omega$-module). We then use this autoduality to construct the inverse (and direct) image of a relation in a different way. 

\begin{proposition} \label{autodual}
$\Omega^X$ is autodual in $s \ell$.
\end{proposition}

\begin{proof} Recall remark \ref{omegayproducto}. 

We define the $s \ell$-morphism $\Omega \mr{\eta}  \Omega^X \otimes \Omega^X$ using the diagonal $X \mr{\triangle} X \times X$, i.e. by the formula $\eta(1) = \displaystyle \bigvee_{x \in X} \{x\} \otimes \{x\}$.

We define the $s \ell$-morphism $\Omega^X \otimes \Omega^X \mr{\eps} \Omega$ using $X \times X \mr{\delta} \Omega$, i.e. by the formula $\eps(\{x\} \otimes \{y\} ) = \llbracket x  \! = \! y\rrbracket$.

We need to prove two triangular equations, we will show that the composition $\Omega^X \mr{id \otimes \eta} \Omega^X \otimes \Omega^X \otimes \Omega^X \mr{\eps \otimes id} \Omega^X$ is the identity since the other one is symmetric. Chasing a generator $\{x\}$, we have to show the equation $\{x\} = \bigvee_y \llbracket x \! = \! y\rrbracket  \cdot \{y\}$, which is immediate from \ref{formulainterna}.
\end{proof}

\begin{\prop} \label{prop:imageninversausandoautodual}
 Consider the extension of a relation $\lambda$ as a $s \ell$-morphism \linebreak $\Omega^X \otimes \Omega^Y \mr{\lambda} \Omega$ (recall remark \ref{omegayproducto}), and the corresponding $s \ell$-morphism $\Omega^Y \mr{\mu} \Omega^X$ given by the autoduality of $\Omega^X$. Then $\mu = \lambda^*$.
\end{\prop}
\begin{proof}
  $\mu$ is constructed as the composite $\Omega^Y \mr{\eta \otimes id} \Omega^X \otimes \Omega^X \otimes \Omega^Y \mr{id \otimes \lambda} \Omega^X$. Following a generator $\{y\}$ we obtain that $\mu(\{y\}) = \bigvee_{x \in X} \lambda(x,y) \cdot \{x\} \stackrel{\eqref{tabla2}}{=} \lambda^*(\{y\})$
\end{proof}

\begin{corollary} \label{dualintercambia}
 Taking dual interchanges direct and inverse image, i.e. $$\Omega^X \mr{\lambda_* = (\lambda^*)^{\lor}} \Omega^Y, \quad \Omega^Y \mr{\lambda^* = (\lambda_*)^{\lor}} \Omega^X. $$ 
 
 \vspace{-3ex}
 
  \qed
\end{corollary}

\end{sinnadastandard}

\begin{sinnadastandard} \label{aplication} {\bf An application to the inverse image.}

As an application of our previous results, we will give an elementary proof of \cite{JT}, IV.2 Proposition 1. This is a different characterization of arrows of $\Sat$: they are the relations whose inverse image is not only a sup-lattice morphism, but a locale one.

 The ``geometric aspect of the concept of Locale'' is studied by considering the category of spaces $S \! p = Loc^{op}$ (\cite{JT}, IV, p.27). If $H \in Loc$, we denote its corresponding space by $\overline{H}$, and if $X \in S \! p$ we denote its corresponding locale (of open parts) by $\cc{O}(X)$. If $H \mr{f} L$, then we denote $\overline{L} \mr{\overline{f}} \overline{H}$, and if $X \mr{f} Y$ then we denote $\cc{O}(Y) \mr{f^{-1}} \cc{O}(X)$.

\begin{\prop} \label{discretespace}
 We have a full and faithful functor $\Sat \mr{(-)_{dis}} S \! p$ that maps \linebreak $X \mapsto X_{dis} = \overline{\Omega^X}$, $f \mapsto \overline{f^*}$.
\end{\prop}

\begin{proof}
By propositions \ref{edyuvdafunction} and \ref{functionconallegories}, the functor $\Omega^{(-)}$ from remark \ref{1erremarkbasico} co-restricts to $S \! p$ as a full and faithful functor $\Sat \mr{(-)_{dis}} S \! p$.
\end{proof}

\end{sinnadastandard}
\end{sinnadastandard}

\subsection{$\ell$-relations and $\ell$-functions in a topos} \label{sec:ellrel}

We consider now a generalization of the concept of relation, that we will call $\ell$-relation, by letting $\Omega$ be any sup-lattice:

\begin{\de}
Let $G \in s \ell$. An $\ell$-relation (in $G$) is an arrow 
\mbox{$X \times Y \mr{\lambda} G$.} 
\end{\de}

\begin{\de} \label{4axiomsparaG}
The four axioms of definition \ref{4axioms}, exactly as they are written, make sense for any $\ell$-relation with values in a locale $G$. As for relations, an $\ell$-function is a \linebreak $\ell$-relation satisfying $uv)$ and $ed)$. Remark \ref{simetria} also applies here to define $\ell$-op-functions and $\ell$-bijections.
\end{\de}

\begin{assumption}
 In the sequel, whenever we consider the $\wedge$ or the $1$ of $G$, we assume implicitly that $G$ is a locale (for example, when considering any of the four axioms of \ref{4axioms} (in particular $\ell$-functions or $\ell$-bijections), or when considering $G$-modules).
\end{assumption}

\begin{sinnadastandard} {\bf On the structure of $G^X$.} 

We generalize the results of the previous section to $G^X$ instead of $\Omega^X$. When the proof of these results is the same as for $\Omega^X$, we omit it.

$G^X$ has the sup-lattice (or locale) structure given pointwise by the structure of $G$. The arrow $G \otimes G^X \mr{\cdot} G^X$ given by $(a \cdot \theta) (x) = a \wedge \theta(x)$ is a $G$-module structure for $G^X$.

We have a $G$-singleton $X \mr{ \{ \}_G } G^X$ defined by $\{x\}_G(y)= \llbracket x \! = \! y \rrbracket$ (recall remark \ref{omegaenL}).

\begin{lemma} [cf. lemma \ref{lema1}] \label{lema1paraG}
 For each $\theta \in G^X$, $x,y \in X$, we have $\theta(x) \wedge \llbracket x \! = \! y\rrbracket  \leq \theta(y)$ 
\end{lemma}

\begin{proof}
 This was shown in the proof of lemma \ref{ecuacionenL}.
\end{proof}

\begin{\prop} [cf. proposition \ref{formulainterna}] \label{formulainternaparaG} 
For each $\theta \in G^X$, $\displaystyle \theta = \bigvee_{x \in X} \theta(x) \cdot \{x\}_G$. This shows how any arrow $X \mr{f} M$ into a $G$-module can be extended uniquely to $G^X$ as $\displaystyle f(\theta) = \bigvee_{x \in X} \theta(x) \cdot f(x)$, so the $G$-singleton $X \mr{ \{ \}_G } G^X$ is a free-$G$-module structure. \qed
\end{\prop}

\begin{lemma} [cf. lemma \ref{reldomegaY}] \label{reldomegaYparaG}
The $G$-singleton arrow  $Y \mr{\{\}_G} G^Y$ determines a \emph{presentation} of the $G$-locale $G^Y$ in the following sense:
$$
 i)\;  1 = \displaystyle \bigvee_{y \in Y} \{y\}_G, \quad \quad \quad 
 ii) \; \{x\}_G \wedge \{y\}_G \leq \llbracket x \! = \! y\rrbracket .
$$
Given any other arrow $Y \mr{f} L$ into a $G$-locale $L$ such that: 
$$
 i) \; 1 = \bigvee_y f(y), \quad \quad \quad 
 ii)\; f(x) \wedge f(y) \leq \llbracket x \! = \! y\rrbracket 
$$
there exists a unique $G$-locale morphism $G^Y \mr{f} L$ such that $f(\{y\}_G) = f(y)$. \qed
\end{lemma}

\begin{remark} [cf. remark \ref{reldomegaYremark}] \label{reldomegaYremarkparaG} 
We have proved that given any arrow $Y \mr{f} L$ into a \linebreak $G$-locale, its extension as a $G$-module morphism to $G^Y$  preserves $1$ if and only if equation $i)$ holds in $L$, and preserves $\wedge$ if and only if equation $ii)$ holds in $L$.
\end{remark}

\begin{sinnadastandard} [cf. \ref{invdirimage}] {\bf The inverse and the direct image of an $\ell$-relation.} We have the correspondence between an $\ell$-relation, its direct image and its inverse image given by proposition \ref{formulainternaparaG}:

\begin{equation} \label{lambdaphipsiG} 
\begin{tabular}{c} 
$X \times Y \mr{\lambda} G$ an $\ell$-relation \\ \hline \noalign{\smallskip}
$G^Y \mr{\lambda^*} G^X$ a $G$-$Mod$ morphism \\ \hline \noalign{\smallskip}
$G^X \mr{\lambda_*} G^Y$  a $G$-$Mod$ morphism \\
\\
$\lambda^*(\{y\}_G)(x) = \lambda(x, y) = \lambda_*(\{x\}_G)(y)$
\end{tabular}
\end{equation}

$$ \lambda^*(\{y\}_G) = \bigvee_{x \in X} \lambda(x,y) \cdot \{x\}_G, \quad \lambda_*(\{x\}_G = \bigvee_{y \in Y} \lambda(x,y) \cdot \{y\}_G$$

\end{sinnadastandard}

\begin{proposition} [cf. proposition \ref{edyuvdafunction}] \label{edyuvdafunctionG} In the correspondence \eqref{lambdaphipsiG}, $\lambda^*$ respects $1$ (resp $\wedge$) if and only if $\lambda$ satisfies axiom $ed)$ (resp. $uv)$).
In particular an $\ell$-relation $\lambda$ is a 
$\ell$-function if and only if its inverse image $G^Y \mr{\lambda^*} G^X$  is a $G$-locale morphism. \qed
\end{proposition}

\begin{remark} \label{edyuvdafunctionmedioG}
 We can also replace in \eqref{tabla1} only one appearance of $\Omega$ by $G$ to obtain the equivalences 
 
 \begin{equation} 
\begin{tabular}{c} 
$X \times Y \mr{\lambda} G$ an $\ell$-relation \\ \hline  \noalign{\smallskip}
$\Omega^Y \mr{\lambda^*} G^X$ a $s \ell$ morphism \\ \hline  \noalign{\smallskip}
$\Omega^X \mr{\lambda_*} G^Y$ a $s \ell$ morphism \\
\end{tabular}
\end{equation}
 
 A proof symmetric to the one of proposition \ref{edyuvdafunction} shows that $\lambda$ is an $\ell$-op-function if and only if $\lambda_*$ is a locale morphism.
\end{remark}

\end{sinnadastandard}

\begin{sinnadastandard} [cf. \ref{imageninversausandoautodual}] \label{imageninversausandoautodualparaG} {\bf The autoduality of $G^X$.}
We show now that $G^X$ is autodual as a \linebreak $G$-module. We then use this autoduality to construct the inverse (and direct) image of an $\ell$-relation in a different way. 

\begin{remark} [cf. remark \ref{omegayproducto}]
 Given $X,Y \in \Sat$, $G^X \stackrel[G]{}{\otimes} G^Y$ is the free $G$-module on $X \times Y$, with the singleton given by the composition of $X \times Y \mr{<\{\}_G,\{\}_G>} G^X \times G^Y$ with the univeral bi-morphism $G^X \times G^Y \mr{} G^X \stackrel[G]{}{\otimes} G^Y$ (see \cite{JT}, II.2 p.8). We will denote this composition by $\{\}_G \stackrel[G]{}{\otimes} \{\}_G$.
\end{remark}

\begin{proposition} [cf. proposition \ref{autodual} ] \label{autodualparaG}
$G^X$ is autodual in $G$-Mod (in the sense of definition \ref{dualcomobimod}), with $G$-module morphisms $G \mr{\eta}  G^X \stackrel[G]{}{\otimes} G^X$, $G^X \stackrel[G]{}{\otimes} G^X \mr{\eps} G$ given by the formulae
$$\eta(1) = \bigvee_{x \in X} \{x\}_G \otimes \{x\}_G, \quad \eps(\{x\}_G \otimes \{y\}_G ) = \llbracket x \! = \! y\rrbracket .$$
\qed
\end{proposition}

\begin{\prop} [cf. proposition \ref{prop:imageninversausandoautodual}] \label{prop:imageninversausandoautodualparaG}
 Consider the extension of an $\ell$-relation $\lambda$ as a $G$-module morphism $G^X \stackrel[G]{}{\otimes} G^Y \mr{\lambda} G$, and the corresponding $G$-module morphism $G^Y \mr{\mu} G^X$ given by the autoduality of $G^X$. Then $\mu = \lambda^*$. \qed
\end{\prop}

\begin{corollary} [cf. corollary \ref{dualintercambia}] \label{dualintercambiaparaG}
 Taking dual interchanges direct and inverse image, i.e. $$G^X \mr{\lambda_* = (\lambda^*)^{\lor}} G^Y, \quad G^Y \mr{\lambda^* = (\lambda_*)^{\lor}} G^X.$$
 
  \vspace{-3ex}
 
  \qed
\end{corollary}

\end{sinnadastandard}

\pagebreak

\section{$\rhd$ and $\lozenge$ diagrams} \label{diagrams}

We now want to generalize to $\ell$-relations in an arbitrary topos $\Sat$ the work done in section \ref{sec:general}, since this will let us generalize the equivalence between the Tannaka and the Galois contexts in the next sections. We include the reference to each corresponding result in section \ref{sec:general}, and we omit the proof when it is the same as the one there. Consider the following situation (cf. \ref{neutraldiagramadiamante12}).

\begin{sinnadastandard} \label{diagramadiamante12}
Let $X \times Y \mr{\lambda} G$, \mbox{$X' \times
Y'\mr{\lambda'}G$,} be two $\ell$-relations and $X \mr{f} X'$, $Y \mr{g} Y'$ be two maps, or, more generally, consider two spans (which induce relations that we also denote with the same letters), 
$$
\xymatrix@C=3ex@R=1ex
         {
         {} & R \ar[dl]_{p} \ar[dr]^{p'}
         \\
         X & {} & X',
         }
\hspace{3ex}                  
\xymatrix@C=3ex@R=1ex
         {
         {} & S \ar[dl]_{q} \ar[dr]^{q'}
         \\
         Y & {} & Y'\;
         }
\hspace{3ex}
\xymatrix@C=2ex@R=1ex
         {
          {}
          \\
          R = p' \circ p^{op}, \;\; S = q' \circ q^{op}\,,
         }
$$
and a third $\ell$-relation $R \times S \mr{\theta} G$. 

These data give rise to the following diagrams in $Rel(\Sat)$: 
\begin{equation} \label{triangleequation}
\xymatrix@C=1.8ex@R=3ex
        {
         \hspace{-1ex} \lozenge_1 = \lozenge_1(f,g)  
         &&&& \lozenge_2 = \lozenge_2(f,g) 
         &&&& \lozenge = \lozenge(R,S)
        }
\end{equation}

\mbox{
$
       {
        \xymatrix@C=1.4ex@R=3ex
                 {
                  & X \times Y  \ar[rd]^{\lambda} 
                  \\
			        X \times Y' \ar[rd]_{f \times Y'} 
			                    \ar[ru]^{X \times g^{op}} & \equiv & G\,,
			      \\
			       & X' \times Y' \ar[ru]_{\lambda '} 
			      } 
	    }
$ 
$ \quad
       {
        \xymatrix@C=1.4ex@R=3ex
                 {
                  & X \times Y  \ar[rd]^{\lambda} 
                  \\
			        X' \times Y \ar[rd]_{X' \times g} 
			                    \ar[ru]^{f^{op} \times Y} & \equiv & G\,,
			      \\
			       & X' \times Y' \ar[ru]_{\lambda '} 
			      } 
	    }
$
$ \quad
       {
        \xymatrix@C=1.4ex@R=3ex
                 {
                  & X \times Y  \ar[rd]^{\lambda} 
                  \\
			        X \times Y' \ar[rd]_{R \times Y'} 
			                    \ar[ru]^{X \times S^{op}} & \equiv & G\,,
			      \\
			       & X' \times Y' \ar[ru]_{\lambda '} 
			      } 
	    }
$
}

\vspace{1ex}

with corresponding diagrammatic versions (see appendix \ref{ascensores})

\begin{equation} \label{ascensoresdiamante}
\xymatrix@C=-1.5ex
         {   X \did & & Y' \dcellb{g^{op}}
          \\              
             X & & Y
          \\
            & G \cl{\lambda} 
         }
\xymatrix{ \\ \;\;=\;\; }
\xymatrix@C=-1.5ex
         {   X \dcell{f} & & Y' \did
          \\              
             X' & & Y'
          \\
           & G \cl{\lambda'}
         }
\xymatrix{ \\ , \\}
\quad
\xymatrix@C=-1.5ex
         {   X' \dcellb{f^{op}} & & Y \did
          \\              
             X & & Y
          \\
           & G \cl{\lambda}
         }
\xymatrix{ \\ = \\}
\xymatrix@C=-1.5ex
         {   X' \did & & Y \dcell{g}
          \\              
             X' & & Y'
          \\
            & G \cl{\lambda'} 
         }
\xymatrix{ \\ , \\}
\quad
\xymatrix@C=-1.5ex
         {   X \did & & Y' \dcellb{S^{op}}
          \\              
             X & & Y
          \\
            & G \cl{\lambda} 
         }
\xymatrix{ \\ = \\}
\xymatrix@C=-1.5ex
         {   X \dcell{R} & & Y' \did
          \\              
             X' & & Y'
          \\
           & G \cl{\lambda'}
         }
\end{equation}

We want to write the equations expressed by the diagrams. We will do this in the case where $R,S$ are relations, therefore the monomorphisms $R \hookrightarrow X \times X'$, $S \hookrightarrow Y \times Y'$ correspond to morphisms into the subobject classifier $X \times X' \mr{\llbracket -R-\rrbracket} \Omega$, $Y \times Y' \mr{\llbracket -S-\rrbracket} \Omega$.

If we define $s \ell_0 := s \ell_0(\Sat)$ as the full subcategory of $s \ell := s \ell(\Sat)$ with objects of the form $\Omega^X$, $X \in \Sat$, then the functor $Rel \mr{(-)_{*}} s \ell_0$ that maps $X$ to the power set $PX = \Omega^X$ (see \ref{onthestruct}), $R \mapsto R_*$ is an isomorphism of categories (see    \ref{otro1erremarkbasico}). Corollary \ref{dualintercambia} implies that the opposite relation $R^{op}$ corresponds in $s \ell$ to $R^\wedge$ defined by the autoduality of $\Omega^X$. By looking at the definitions of $\eta$ and $\eps$ in the proof of proposition \ref{autodual} and chasing elements, we obtain that the previous diagrams express the equations: 

$$\lozenge_1:\, \hbox{ for each } a \in X, b' \in Y', \:
\lambda'( f(a),b' )  \,=\, \displaystyle \bigvee_{y \in Y} \llbracket g(y) \! = \! b'\rrbracket \cdot \lambda( a,y ), \quad \quad $$

\vspace{-3ex}

\begin{equation} \label{ecuacionesdiagramas}
 \lozenge_2:\, \hbox{ for each } a' \in X', b \in Y, \:
\lambda'( a',g(b) )  \,=\, \displaystyle \bigvee_{x \in X} \llbracket f(x) \! = \! a'\rrbracket \cdot \lambda( x,b ), \end{equation}
$$\lozenge: \hbox{ for each } a \in X, b' \in Y', 
           \displaystyle \bigvee_{y \in Y} \llbracket ySb'\rrbracket \cdot \lambda( a,y ) 
         = \displaystyle \bigvee_{x' \in X'} \llbracket aRx'\rrbracket \cdot \lambda'( x',b' ).$$

\end{sinnadastandard}

\begin{remark} [cf. remark \ref{diamondesmasfuerteneutral}] \label{diamondesmasfuerte} It is clear that diagrams $\lozenge_1$ and  $\lozenge_2$ are particular cases of \mbox{diagram $\lozenge$.} Take $R = f,\; S = g$, then $\lozenge_1(f,g) = \lozenge(f, g)$, and  $R = f^{op},\; S = g^{op}$, then  \mbox{$\lozenge_2(f,g) = \lozenge(f^{op}, g^{op})$.} 
\end{remark}

The general $\lozenge$ diagram follows from these two particular cases:

\begin{proposition} [cf. proposition \ref{neutralpre2rhdimpliesdiamante}] \label{pre2rhdimpliesdiamante}
Let $R$, $S$ be any two spans connected  by an $\ell$-relation $\theta$ as above. If  $\lozenge_1(p',q')$ and $\lozenge_2(p,q)$ hold, then so does 
$\lozenge(R,S)$. \qed
\end{proposition}

\begin{sinnadastandard} [cf. \ref{neutraldiagramatriangulo}] \label{diagramatriangulo}
 Two maps $X \mr{f} X'$, $Y \mr{g} Y'$ also give rise to the following diagram:
$$
\rhd = \rhd(f,g): \hspace{5ex}
\vcenter
        {
         \xymatrix@C=3ex@R=1ex
                   {
                    X \times Y  \ar[rrd]^{\lambda} 
                                \ar[dd]_{f \times g}
                    \\
                    {} & \hspace{-4ex} {}^{\!\!\geq} & G \,,
                    \\
                    X' \times Y'  \ar[rru]_{\lambda '}
                   }
        }
$$
\vspace{1ex}

\noindent expressing the equation
$\rhd:\, \hbox{for each } a \in X, b \in Y, \: \lambda ( a,b ) \leq \lambda' ( f(a),g(b) ) $.

\end{sinnadastandard}

\begin{proposition} [cf. proposition \ref{neutraldiamondimpliesrhd}] \label{diamondimpliesrhd}
If either $\lozenge_1(f,\,g)$ or $\lozenge_2(f,\,g)$ holds, then so does $\;\rhd(f,\,g)$.
\end{proposition}
\begin{proof} For each $a \in X, b \in Y, \lambda( a,b )  \leq \displaystyle \bigvee_{y \in Y} \llbracket g(y) \! = \! g(b)\rrbracket  \cdot \lambda( a,y )  =
  \lambda'( f(a),g(b) )$ using $\lozenge_1$. Clearly a symmetric arguing holds
  using $\lozenge_2$. 
\end{proof}

The reverse implication holds under some extra hypotheses.

\begin{proposition} [cf. proposition \ref{neutralrhdimpliesdiamond}] \label{rhdimpliesdiamond} 
${}$
\begin{enumerate}
	\item If $\lambda$ is ed) and $\lambda'$ is uv) (in particular, if they are $\ell$-functions), then $\rhd(f,g)$ implies $\lozenge_1(f,g)$. 
	\item If $\lambda$ is su) and $\lambda'$ is in) (in particular, if they are $\ell$-opfunctions), then $\rhd(f,g)$ implies 
$\lozenge_2(f,g)$. 
\end{enumerate}
\end{proposition}
\begin{proof}
We prove $1.$, a symmetric proof yields $2$. For each $a \in X$, $b' \in Y'$,

\vspace{1ex}

\noindent $ \displaystyle \lambda'( f(a),b' )  \stackrel{ed)_{\lambda}}{\;=\;} \lambda'( f(a),b' )  \wedge \bigvee_{y \in Y} \lambda( a,y )  \;=\; \bigvee_{y \in Y}
\lambda'( f(a),b' )  \wedge \lambda( a,y ) \stackrel{\rhd}{=}$

$$ \bigvee_{y \in Y} \lambda'( f(a),b' ) \wedge \lambda'( f(a),g(y) ) \wedge \lambda( a,y ) \! \! \stackrel{uv)_{\lambda'}}{=} \! \! \bigvee_{y \in Y} \igu[g(y)]{b'} \cdot \lambda' (f(a),g(y) ) \wedge \lambda( a,y ) $$

\flushright $ \displaystyle \stackrel{\rhd}{=} \bigvee_{y \in Y} \llbracket g(y) \! = \! b'\rrbracket  \cdot \lambda( a,y ).$

\end{proof}

\begin{sinnadastandard} [cf. \ref{neutralmoregenerally}] \label{moregenerally}
 More generally, consider two spans as in \ref{diagramadiamante12}. We have the following \mbox{$\rhd$ diagrams:} 
\begin{equation} \label{2rhd}
         \xymatrix@C=3ex@R=1ex
                   {
                    R \times S  \ar[rrd]^{\theta} 
                                \ar[dd]_{p \times q}
                    \\
                    {} & \hspace{-4ex} {}^{\!\!\geq} & G\,,
                    \\
                    X \times Y  \ar[rru]_{\lambda}
                   } 
\hspace{5ex}       
         \xymatrix@C=3ex@R=1ex
                   {
                    R \times S  \ar[rrd]^{\theta} 
                                \ar[dd]_{p' \times q'}
                    \\
                    {} & \hspace{-4ex} {}^{\!\!\geq} & G \,.
                    \\
                    X' \times Y'  \ar[rru]_{\lambda '}
                   }
\end{equation}
\end{sinnadastandard}

\begin{proposition} [cf. proposition \ref{neutral2rhdimpliesdiamante}] \label{2rhdimpliesdiamante}
We refer to \ref{diagramadiamante12}: Assume that  
$\lambda$ is in), 
$\lambda '$ is uv), and that the 
$\rhd(p,\,q)$, $\rhd(p',\,q')$ diagrams hold. Then  if 
$\theta$ is ed) and su), diagram $\lozenge(R,\,S)$ holds. \qed
\end{proposition}

\begin{sinnadastandard} [cf. \ref{neutralproductrelationdef}] \label{productrelationdef}
Given two $\ell$-relations, $X \times Y \mr{\lambda} G$, $X' \times
Y'\mr{\lambda'}G$, the product $\ell$-relation $\lambda \boxtimes \lambda'$ is defined by the composition

\begin{center}
$X \times X' \times Y \times Y' 
 \mr{X \times \psi \times Y'} X \times Y \times X' \times Y' 
 \mr{\lambda \times \lambda'} G \times G \mr{\wedge} G$ 
\end{center}

\begin{center}
$
(\lambda \boxtimes \lambda')  ( (a,a'),(b,b') )  = \lambda  ( a,b )  \wedge \lambda'  ( a',b' )$. 
\end{center}
\end{sinnadastandard}
The following is immediate and straightforward:
\begin{proposition} [cf. proposition \ref{neutralproductrelation}] \label{productrelation}
Each axiom in definition \ref{4axioms} for $\lambda$ and $\lambda'$ implies the respective axiom for the product $\lambda \boxtimes \lambda'$. \cqd 
\end{proposition}

\begin{remark} [cf. remark \ref{neutralproductistheta}] \label{productistheta}
The diagrams $\rhd$ in \ref{2rhd} mean that $$\theta \leq \lambda \boxtimes \lambda' \circ (p, p') \times (q, q').$$
In particular, there is always an $\ell$-relation $\theta$ in \ref{diagramadiamante12} such that \eqref{2rhd} holds, which may be taken as the composition 
$R \times S \mr{(p, p') \times (q, q')} X \times X' \times Y \times Y' 
\mr{\lambda \boxtimes \lambda'} G$.
However, it is important to consider an arbitrary $\ell$-relation $\theta$ (see propositions \ref{dim1ydim2esdim} and \ref{trianguloesdiamante}).

For $\theta = \lambda \boxtimes \lambda' \circ (p, p') \times (q, q')$, the converse of proposition \ref{2rhdimpliesdiamante} holds:
\end{remark}

\begin{proposition} [cf. proposition \ref{neutraldiamanteimplies2rhd}] \label{diamanteimplies2rhd}
We refer to \ref{diagramadiamante12}: Assume that $R$ and $S$ are relations, that $\lambda$, $\lambda '$ are $\ell$-bijections, and take 
\mbox{$\theta=\lambda \boxtimes \lambda' \circ (p, p') \times (q, q')$}. Then,
if $\lozenge(R,\,S)$ holds, $\theta$ is an $\ell$-bijection.
\end{proposition}
\begin{proof} First we prove axiom $uv)$. From the $\rhd$ diagrams \eqref{2rhd} we get the following equations: for each $r \in R, s_1, s_2 \in S$,

\vspace{-3ex}

$$ \theta ( r,s_1 ) \wedge \theta ( r,s_2 )  \leq \;\;\;\; \lambda ( p(r),q(s_1) ) \wedge \lambda ( p(r),q(s_2) ) \; \stackrel{uv)_{\lambda}}{\leq} \;\; \llbracket q(s_1) \! = \! q(s_2)\rrbracket ,$$

\vspace{-4ex}

$$\theta ( r,s_1 ) \wedge \theta ( r,s_2 ) \leq \lambda' ( p'(r),q'(s_1) ) \wedge \lambda' ( p'(r),q'(s_2) ) \!\! \stackrel{uv)_{\lambda'}}{\leq} \llbracket q'(s_1) \!  = \!  q'(s_2)\rrbracket .$$

\noindent Taking infima we have: $\theta ( r,s_1 ) \wedge \theta ( r,s_2 ) \leq \llbracket (q,q')(s_1) \! = \! (q,q')(s_2)\rrbracket  = \llbracket s_1 \! = \! s_2\rrbracket $, since $(q,q')$ is a monomorphism (see remark \ref{prelim1}).

We prove now axiom $ed)$. 
We can safely assume  $R \subset X \times X'$ and 
$S \subset Y \times Y'$, and \mbox{$\lambda \boxtimes \lambda' \circ (p, p') \otimes (q, q')$} to be the restriction of $\lambda \boxtimes \lambda'$ to $R \times S$.
For each $(a,a') \in R$, we compute:

$$\bigvee_{(y,y') \in S} \theta ( (a,a'),(y,y') ) = \bigvee_{y' \in Y'} \bigvee_{y \in Y} \llbracket ySy' \rrbracket  \cdot \lambda ( a,y ) \wedge \lambda' ( a',y' )  \stackrel{\lozenge}{=}$$

$$= \bigvee_{y' \in Y'} \bigvee_{x' \in X'} \llbracket aRx' \rrbracket \cdot \lambda' ( x',y' ) \wedge \lambda' ( a',y' )
\geq \bigvee_{y'} \lambda' ( a',y' ) \stackrel{ed)_{\lambda'}}{=} 1$$

The inequality is justified by taking $x'=a'$ in the supremum and using that since $(a,a') \in R$, $\llbracket aRa' \rrbracket =1$.
Axioms $in)$ and $su)$ follow in a symmetrical way.
\end{proof}

We found convenient to combine \ref{2rhdimpliesdiamante} and \ref{diamanteimplies2rhd} into:
\begin{proposition} [cf. proposition \ref{neutralcombinacion}] \label{combinacion}
Let $R \subset X \times X'$, $S \subset Y \times Y'$ be any two relations, and $X \times Y \mr{\lambda} G$, 
$X' \times Y' \mr{\lambda'} G$ be $\ell$-bijections.  Let $R \times S \mr{\theta} G$ be the restriction of $\lambda \boxtimes \lambda'$ to $R \times S$. Then, 
$\lozenge(R, S)$ holds if and only if $\theta$ is an $\ell$-bijection. \cqd
\end{proposition}

\pagebreak

\section{$\rhd$ and $\lozenge$ cones} \label{sec:cones}

In this section we generalize the results of section \ref{conesneutral} in two ways, both needed for our purpose. We work over any arbitrary topos $\Sat$ instead of over $Set$, and we develop a theory of $\rhd$ and $\lozenge$ cones for two different functors $F,F'$ instead of just one. As in the previous section, we include the reference to each corresponding result in section \ref{conesneutral} and omit the proof when it is the same as the one there.

\vspace{2ex}

\begin{sinnadastandard} [cf. \ref{neutralrel}] \label{extensionarel}
Consider a geometric morphism $\xymatrix{\cc{S} \ar[r] & \cc{E}}$, with inverse image \linebreak $\xymatrix{\cc{E} \ar[r]^{F} & \cc{S}}$. 
Consider the extension $T$ of $F$ to $Rel(\cc{E})$ as in the following commutative diagram (recall remark \ref{1erremarkbasico}):
$$
\xymatrix{\cc{E} \ar[r]^{\lambda_{(-)}} \ar[d]_F 
         & \cc{R}el(\cc{E}) \ar[d]^T 
         \\
           \cc{S} \ar[r]^{\lambda_{(-)}}  \ar@/_3ex/[rr]_P
         & \cc{R}el \; \ar@{^(->}[r]^{(-)_{*}} 
         & s \ell   }
$$
On objects $TX = FX$, and the value of $T$ in a relation $R \mmr{} X \times Y$ in $\cc{E}$ is the relation $FR \mmr{} FX \times FY$ in $\cc{S}$. In particular, for arrows $f$ in $\cc{E}$, $T(R_f) = R_{F(f)}$ (see \ref{arrowsvsfunctions}), or, if we abuse the notation by identifying $f$ with the relation given by its graph, $T(f) = F(f)$.

Since $F$ preserves products, $T$ is a tensor functor (recall \ref{otro1erremarkbasico}). From that fact (since tensor functors preserve dualities, see \ref{dualintercambia}), or immediately from the definition, we obtain that for every relation $R$ in $\cc{E}$ we have $T(R^{op}) = (TR)^{op}$.
\end{sinnadastandard}

\begin{sinnadaitalica} [cf. \ref{neutralFequivalenceiffRel(F)so}] \label{FequivalenceiffRel(F)so}
It can be seen that          
$F$ is an equivalence if and only if $T$ is so.  
\cqd
\end{sinnadaitalica} 

\vspace{2ex}

Consider now two geometric morphisms with inverse images $\xymatrix{\cc{E} \ar@<1ex>[r]^{F} \ar@<-1ex>[r]_{F'} & \cc{S}}$, and their respective extensions to the $Rel$ categories $T$, $T'$.

\vspace{2ex}

\begin{\de} [cf. definition \ref{neutraldefdeconos}] \label{defdeconos}
Let $H$ be a sup-lattice in $\Sat$. A cone $\lambda$ (with vertex $H$) is a family of $\ell$-relations $FX \times F'X \mr{\lambda_{X}} H$, one for each $X \in \cc{E}$. Note that, a priori, a cone is just a family of arrows without any particular property. This isn't standard terminology, but we do this in order to use a different prefix depending on which diagrams commute. Each arrow $X \mr{f} Y$ in 
$\cc{E}$ and each arrow 
$X \mr{R} Y$ in $\cc{R}el(\cc{E})$ (i.e relation 
$R \mmr{} X \times Y$ in $\cc{E}$), determine the following diagrams:

\begin{center}
 \begin{tabular}{cc}
  $\rhd(f) = \rhd(F(f),F'(f))$ & $\lozenge(R) = \lozenge(TR,T'R)$ \\ \noalign{\smallskip}
  $\xymatrix@C=4ex@R=3ex
        {
         FX \times F'X \ar[rd]^{\lambda_{X}}  
                      \ar[dd]_{F(f) \times F'(f)}  
        \\
         {} \ar@{}[r]^(.3){\geq}
         &  H    
        \\
         FY \times F'Y \ar[ru]_{\lambda_{Y}} 
         } $ &
  $\xymatrix@C=4ex@R=3ex
        {
         & TX \times T'X \ar[rd]^{\lambda_{X}}  
        \\
           TX \times T'Y 
               \ar[rd]_{TR \times T'Y  \hspace{2.5ex}} 
			   \ar[ru]^{TX \times T'R^{op} \hspace{2.5ex}} 
	     & \equiv 
         & H
        \\
         & TY \times T'Y \ar[ru]_{\lambda_{Y}} 
         }$
\\ \noalign{\smallskip} \noalign{\smallskip} 
$\lozenge_1(f) = \lozenge_1(F(f),F'(f))$ & $\lozenge_2(f) = \lozenge_2(F(f),F'(f)) $ \\ \noalign{\smallskip}
 ${ \xymatrix@C=1.4ex@R=3ex {
                  & FX \times F'X  \ar[rd]^{\lambda_X} 
                  \\
			        FX \times F'Y \ar[rd]_{F(f) \times F'Y \hspace{2.5ex}} 
			                    \ar[ru]^{FX \times F'(f)^{op} \hspace{2.5ex}} & \equiv & H
			      \\
			       & FY \times F'Y \ar[ru]_>>>>{\lambda_Y} 
			      } 
}$ &
$\xymatrix@C=1.4ex@R=3ex
                 {
                  & FX \times F'X  \ar[rd]^{\lambda_X} 
                  \\
			        FY \times F'X \ar[rd]_{FY \times F'(f) \hspace{2.5ex}} 
			                    \ar[ru]^{F(f)^{op} \times F'X \hspace{2.5ex}} & \equiv & H
			      \\
			       & FY \times F'Y \ar[ru]_>>>>{\lambda_Y} 
		} $
 \end{tabular}
\end{center}

We say that $\lambda$ is a \emph{$\rhd$-cone} if the $\rhd(f)$ diagrams  hold, and that it is a 
\emph{$\lozenge$-cone} if the $\lozenge(R)$ diagrams hold. Similarly we talk of \emph{$\lozenge_1$-cones} and \emph{$\lozenge_2$-cones} if the 
$\lozenge_1(f)$ and $\lozenge_2(f)$ \mbox{diagrams} hold. 
If $H$ is a locale and the $\lambda_X$ are $\ell$-functions, $\ell$-bijections, we say that we have a cone \textit{of $\ell$-functions, $\ell$-bijections}.
\end{\de}

The following proposition shows that $\lozenge_1$-cones correspond to natural transformations. 

\begin{proposition} \label{naturaligualcone}
 Consider a family of arrows $FX \mr{\theta_X} F'X$, one for each $X \in \cc{E}$. Each $\theta_X$ corresponds by the autoduality of $F'X$ (see proposition \ref{autodual}) to a function \linebreak $FX \times F'X \mr{\varphi_X} \Omega$ yielding in this way a cone $\varphi$. $\theta$ is a natural transformation if and only if $\varphi$ is a $\lozenge_1$-cone. 
\end{proposition}

\begin{proof}

As with every duality (see \eqref{lambdarhoconascensores}), the correspondence between $\theta_X$ and $\varphi_X$ is given by the diagrams:

$$\vcenter{\xymatrix{ \varphi_X:    }}
\vcenter{\xymatrix@C=.2pc{  FX \dcell{\theta_X} && F'X \did \\ 
			    F'X && F'X \\
			  & \quad \cl{\eps} }}
\vcenter{\xymatrix{ \quad\quad \quad \theta_X:    }}
\vcenter{\xymatrix@C=.2pc{  FX \did && & \op{\eta} \\
		    FX && F'X && F'X \did \\
		    & \quad  \cl{\varphi_X} &&& F'X    }}$$

Also, the naturality $N$ of theta and the $\lozenge_1$ diagrams (see \eqref{ascensoresdiamante} and recall from corollary \ref{dualintercambia} that $f^{op} = f^{\wedge}$) can be expressed as: for each $X \mr{f} Y$, 
$$\vcenter{\xymatrix{ N(f): }}
\vcenter{\xymatrix@C=-0.3pc@R=1pc{  FX \dcellbb{F(f)} \\ FY \dcellb{\theta_Y} \\ F'Y  }}
\vcenter{\xymatrix{  =  }}
\vcenter{\xymatrix@C=-0.3pc@R=1pc{  FX \dcellb{\theta_X} \\ F'X \dcellbb{F' \! (f)} & \quad ,  \\ F'Y  }}
\vcenter{\xymatrix{  \lozenge_1(f):  }}
\vcenter{\xymatrix@C=-0.3pc@R=1pc{ FX \did &&& \op{\eta} &&& F'Y \did \\              
	      FX \did && F'X \did && F'X \dcellbb{F' \! (f)} && F'Y \did \\
	      FX && F'X  && F'Y && F'Y  \\ 
            & \quad  \cl{\varphi_X}  &&&&   \quad \cl{\eps}          }}
\vcenter{\xymatrix{  =  }}
\vcenter{\xymatrix@C=-0.3pc@R=1pc{FX \dcellbb{F(f)} & & F'Y \did
          \\              
             FY & & F'Y
          \\
           & \quad  \cl{\varphi_Y} }}$$

\noindent $\underline{N(f) \Rightarrow \lozenge_1(f):}$ replace $\theta$ as in the correspondence above in $N(f)$ to obtain

$$\vcenter{\xymatrix@C=-0.3pc@R=1pc{ FX \did &&& \op{\eta} \\
			    FX && F'X && F'X \dcellbb{F' \! (f)} \\
			     & \quad \cl{\varphi_X} &&& F'Y }}
\vcenter{\xymatrix@C=-0.3pc@R=1pc{ \stackrel{N(f)}{=} }}
\vcenter{\xymatrix@C=-0.3pc@R=1pc{   FX \dcellbb{F(f)} &&& \op{\eta} \\
			    FY && F'Y && F'Y \did \\
			    & \quad \cl{\varphi_Y} &&& F'Y }} $$

Compose with $\eps$ and use a triangular identity to obtain $\lozenge_1(f)$.

\noindent $\underline{\lozenge_1(f) \Rightarrow N(f):}$ replace $\varphi$  as in the correspondence above in $\lozenge_1(f)$ to obtain
$$\vcenter{\xymatrix@C=-0.3pc@R=1pc{ FX \dcellbb{F(f)} && F'Y \did \\
				      FY \dcell{\theta_Y} && F'Y \did \\
				      F'Y && F'Y \\
				      & \cl{\eps} }}
\vcenter{\xymatrix@C=-0.3pc@R=1pc{ \stackrel{\lozenge_1(f)}{=} }}
\vcenter{\xymatrix@C=-0.3pc@R=1pc{  FX \dcell{\theta_X} &&& \op{\eta} &&& F'Y \did \\
				      F'X && F'X && F'X \dcellbb{F' \! (f)} && F'Y \did \\
				      & \cl{\eps} &&& F'Y && F'Y \\
				      &&&&& \cl{\eps} }}
\vcenter{\xymatrix@C=-0.3pc@R=1pc{ \stackrel{\triangle}{=} }}
\vcenter{\xymatrix@C=-0.3pc@R=1pc{ FX \dcell{\theta_X} && F'Y \did \\
				    F'X \dcellbb{F(f)} && F'Y \did \\
				    F'Y && F'Y \\
				    & \cl{\eps} }}$$

Compose with $\eta$ and use a triangular identity to obtain $N(f)$.
\end{proof}

\begin{sinnadastandard} \label{sin:naturaligualconeconrho}
Consider now the previous situation together with a topos over $\Sat$, 
$$\vcenter{\xymatrix@C=3.5pc@R=3.5pc{& \cc{G}  \adjuntosd{\gamma^*}{\gamma_*}  \\ \cc{E} \dosflechasr{F}{F'} & \cc{S} }}$$ A natural transformation $\gamma^* FX \mr{\theta_X} \gamma^* F'X$ corresponds to a $\lozenge_1$-cone of functions \linebreak $\gamma^* FX \times \gamma^* F'X \mr{\varphi_X} \Omega_{\cc{G}}$ in $\cc{G}$. As established in \ref{extensionarel}, $\gamma^*$ can be extended to $Rel = s \ell_0$ as a tensor functor (therefore preserving duals), then using the naturality of the adjunction $\gamma^* \dashv \gamma_*$ it follows that $\gamma^* FX \times \gamma^* F'X \mr{\varphi_X} \Omega_{\cc{G}}$ is a $\lozenge_1$-cone if and only if $FX \times F'X \mr{\lambda_X  } \gamma_* \Omega_{\cc{G}}$ is a $\lozenge_1$-cone (in $\cc{S}$). We have proved:

\begin{\prop} \label{naturaligualconeconrho}
 A family of arrows $\gamma^* FX \mr{\theta_X} \gamma^* F' X$ (one for each $X \in \cc{E}$) is a natural transformation if and only if the corresponding cone $FX \times F'X \mr{\lambda_X} \gamma_* \Omega_{\cc{G}}$ is a \linebreak $\lozenge_1$-cone. \qed 
\end{\prop}
\end{sinnadastandard}

\begin{sinnadastandard} \label{atravesdeunmorfismodetopos}
 Consider finally the previous situation together with a morphism $\cc{F} \mr{} \cc{G}$ of topoi over $\Sat$, as in the following diagram:

$$\xymatrix@R=3.5pc{& \cc{G} \df{rr}{h^*}{h_*} && \cc{F}  \\
	  \cc{E} \ar@<.5ex>[rr]^F \ar@<-.5ex>[rr]_{F'} && \cc{S} \dfbis{ul}{\gamma^*}{\gamma_*} \dfbis{ur}{f^*}{f_*}  }$$

Consider the locales in $\Sat$ of subobjects of $1$ in $\cc{G}$, resp. $\cc{F}$, $G:= \gamma_* \Omega_{\cc{G}}$, $L:= f_* \Omega_{\cc{F}}$. Since $h^*$ is an inverse image, it maps subobjects of $1$ to subobjects of $1$ and thus induces a locale morphism that we will denote $G \mr{h} L$. 
\end{sinnadastandard}

\begin{remark} \label{caractsubobjeto}
Consider the comparison morphism $h^* \Omega_{\cc{G}} \mr{\phi_1} \Omega_{\cc{F}}$, induced by the subobject $\vcenter{\xymatrix{1 \ \ \ar@{>->}[r]^{h^* (t)} & h^* \Omega_{\cc{G}}}}$. (see for example \cite{Johnstone}, A.2.1 p.69). Then, for any subobject $\vcenter{\xymatrix{M \ \ \ar@{>->}[r] \ar[d] & X \ar[d]^{\phi_M} \\
										      1 \ar[r]^{1} & \Omega_{\cc{G}}  }}$ 
by composing the pull-backs $\vcenter{\xymatrix{ h^*M \ \ \ar@{>->}[r] \ar[d] & h^*X \ar[d]^{h^*(\phi_M)} \\
				     1 \ \ \ar@{>->}[r]^{h^*(1)} \ar[d] & h^*\Omega_{\cc{G}} \ar[d]^{\phi_1} \\
				     1 \ar[r]^1 & \Omega_{\cc{F}}  }}$ it follows that the characteristic function of the subobject $h^*M$ is $\phi_1 \circ h^* (\phi_M)$.
\end{remark}

\begin{proposition} \label{paraelcoro}
In the hypothesis of \ref{atravesdeunmorfismodetopos}, for $X \in \Eat$, if $FX \times F'X \mr{\lambda} G$ corresponds to $\gamma^* FX \times \gamma^* F'X \mr{\varphi} \Omega_{\cc{G}}$ via the adjunction $\gamma^* \dashv \gamma_*$, then $FX \times F'X \mr{\lambda} G \mr{h} L$ corresponds to $f^* FX \times f^* F'X \mr{h^*(\varphi)} h^* \Omega_{\cc{G}} \mr{\phi_1} \Omega_{\cc{F}}$ via the adjunction $f^* \dashv f_*$.
\end{proposition}

\begin{proof} 

The adjunction $f^* \dashv f_*$ consists of composing the adjunctions $\gamma^* \dashv \gamma_*$ and $h^* \dashv h_*$, then we obtain:

\begin{center}
 \begin{tabular}{c}
  $h^* \gamma^* F X \times h^* \gamma^* F' X  \mr{h^*(\varphi)} h^* \Omega_{\cc{G}} \mr{\phi_1} \Omega_{\cc{F}}$ \\ \hline  \noalign{\smallskip}
  $\gamma^* F X \times \gamma^* F' X  \mr{\varphi} \Omega_{\cc{G}} \mr{\psi_1} h_* \Omega_{\cc{F}}$ \\ \hline  \noalign{\smallskip}
  $F X \times F' X  \mr{\lambda} G \mr{\gamma_*(\psi_1)} L, $
 \end{tabular}
\end{center}

\noindent where $\psi_1$ corresponds to $\phi_1$ in the adjunction $h^* \dashv h_*$. So we have to check that \linebreak $\gamma_*(\psi_1) = h$. $h$ is $h^*$ applied to a subobject $U \hookrightarrow 1$. This subobject can be considered in $G = \gamma_* \Omega_{\cc{G}} = [1, \Omega_{\cc{G}}]$ via its characteristic function $\phi_U$. Now, $\gamma_* (\psi_1) (\phi_U)$ is the composition $1 \mr{\phi_U} \Omega_{\cc{G}} \mr{\psi_1} h_* \Omega_{\cc{F}}$ in $\gamma_* h_* \Omega_{\cc{F}}$, and the corresponding arrow $1 \mr{} \Omega_{\cc{F}}$ is given by the adjunction $h^* \dashv h_*$. But this arrow is $1 \mr{h^* (\phi_U)} h^* \Omega_{\cc{G}} \mr{\phi_1} \Omega_{\cc{F}}$, which by remark \ref{caractsubobjeto} is $\phi_{h^*U}$, and we are done.

\end{proof}

\begin{corollary}  \label{naturaligualconeconrhoatravesdeunmorfismodetopos}
 In the hypothesis of \ref{atravesdeunmorfismodetopos}, consider a natural transformation \linebreak $\gamma^* FX \mr{\theta_X} \gamma^* F'X$ and the corresponding $\lozenge_1$-cone $FX \times F'X \mr{\lambda_X} G$ obtained by proposition \ref{naturaligualconeconrho}. Then the $\lozenge_1$-cone with vertex $L$ corresponding by proposition \ref{naturaligualconeconrho} to the horizontal composition $id_{h^*} \circ \theta$ of natural transformations, whose components are \linebreak $ f^* FX  \mr{h^*(\theta_X)} f^* F'X$, is $FX \times F'X \mr{\lambda_X} G \mr{h} L$. 
\end{corollary}

\begin{proof} Each $\gamma^* FX \mr{\theta_X} \gamma^* F'X$ corresponds to a relation $\gamma^* FX \times \gamma^* F'X \mr{\varphi_X} \Omega_{\cc{G}}$, which corresponds to $FX \times F'X \mr{\lambda_X} G$ via the adjunction $\gamma^* \dashv \gamma_*$. Denote by \linebreak $R_X \hookrightarrow \gamma^* FX \times \gamma^* F'X$ the subobject corresponding to $\varphi_X$.
 
  The subobject corresponding to  $f^* FX \mr{h^*(\theta_X)} f^* F'X$, is $h^*R_X \hookrightarrow f^* FX \times f^* F'X$, whose characteristic function (applying remark \ref{caractsubobjeto}) is the relation $$f^* FX \times f^* F'X \mr{h^*(\varphi_X)} h^*\Omega_{\cc{G}} \mr{\phi_1} \Omega_{\cc{F}}.$$ 
 Proposition \ref{paraelcoro} finishes the proof.
\end{proof}

The results of section \ref{diagrams} yield the following corresponding results for cones.

\begin{proposition} [cf. proposition \ref{neutraldim1ydim2esdim}] \label{dim1ydim2esdim}
A cone $FX \times F'X \mr{\lambda} H$ is a $\lozenge$-cone if and only if it is both a $\lozenge_1$ and a $\lozenge_2$-cone.
\end{proposition}
\begin{proof} The implication $\Rightarrow$ is given by remark \ref{diamondesmasfuerte}, and to prove $\Leftarrow$ given any relation $R \mmr{} X \times Y$ use proposition \ref{pre2rhdimpliesdiamante} with $R = FR$, $S = F'R$, $\lambda = \lambda_X$,  
$\lambda' = \lambda_Y$, and $\theta = \lambda_R$. 
\end{proof}

\begin{proposition} [cf. proposition \ref{neutraltrianguloesdiamante}] \label{trianguloesdiamante}
Let $H \in Loc$. A $\rhd$-cone $FX \times F'X \mr{\lambda} H$ of $\ell$-bijections is a $\lozenge$-cone (of $\ell$-bijections). \qed
\end{proposition}

Consider a topos $\cc{E}$ over $\cc{S}$, and a small site of definition $\cc{C}$ for $\cc{E}$. Let $\cc{C} \mr{F} \cc{S}$ be (the inverse image of) a point of the site, and $\cc{C}^{op} \mr{X} \cc{S}$ be a sheaf, $X \in \cc{E}$. Let 
$\Gamma_F \mr{} \cc{C}$ be the (small) diagram (discrete fibration) of $F$, recall that it is a cofiltered category whose objects are pairs $(c,C)$ with $c \in FC$, and whose arrows $(c,C) \mr{f} (d,D)$ are arrows $C \mr{f} D$ that satisfy $F(f)(c)=d$. Abuse notation and denote also by $F$, 
$\cc{E} \mr{F} \cc{S}$, the inverse image of the corresponding morphism of topoi. Recall the formulae:
\BE \label{yonedaII}
FX = X \otimes_\cc{C} F = \int^{C} XC \times FC 
\;\; \cong \;\; \colim{(c,C) \in \Gamma_F}{XC} \;\; \ml{\rho} \;\; \coprod_{C \in \cc{C}} XC \times FC
\EE
By Yoneda we have $\cc{E}(C, X) \mr{\cong} XC$, and under this identification we have, 

$$\text{for} \;\; C \mr{f} X \;\; \text{and} \;\;
c \in FC, \;\; F(f)(c) = \rho(f, c) \in FX,$$

\vspace{-5ex}

\BE \label{masvale}
\EE 

\vspace{-5ex}

$$\text{for} \;\; E \mr{h} C \text{ in } \cc{C}, \;\; X(h)(f) = fh.$$

\begin{remark} \label{rhoepi}
 Let $a \in FX$. Since $\rho$ is an epimorphism, there exist $C, \; f \in XC$ and $c \in FC$ such that $F(f)(c)=a$.
\end{remark}

\begin{remark} \label{FX}
Let $C,\; D \in \cc{C}$, $f \in XC$, $c \in FC$, and 
$g \in XD$, $d \in FD$, be such that $F(f)(c) = F(g)(d)$, i.e. $\rho(f, c) = \rho(g, d)$. Since the category 
$\Gamma_F$ is cofiltered, by construction of filtered colimits there exist $E,\; e \in FE$ and $E \mr{h} C$,
 $E \mr{\ell} D$ such that $F(h)(e) = c$, 
 $F(\ell)(e) = d$ and $X(h)(f) = X(\ell)(g)$, i.e.
 $f h = g \ell$. \cqd
\end{remark}

\vspace{2ex}

\begin{proposition} [cf. proposition \ref{neutralextension}] \label{extension} Consider a small site of definition $\cc{C}$ of the topos $\cc{E}$. 
Then suitable cones defined over $\cc{C}$ can be extended to $\cc{E}$, more precisely:

1) Let $TC \times T'C \mr{\lambda_C} H$ be a $\lozenge_1$-cone (resp. 
$\lozenge_2$-cone, resp. $\lozenge$-cone) defined over $\cc{C}$. Then, $H$ can be uniquely furnished with $\ell$-relations $\lambda_X$ for all objects $X \in \cc{E}$ in such a way to determine a $\lozenge_1$-cone (resp. $\lozenge_2$-cone, resp. $\lozenge$-cone) over $\cc{E}$ extending $\lambda$. 
 
2) If $H$ is a locale and $\lambda_C$ (one for each $C \in \cc{C}$) is a $\lozenge_1$-cone of $\ell$-functions (resp. $\lozenge_2$-cone of $\ell$-opfunctions, resp. $\lozenge$-cone of $\ell$-bijections), so is $\lambda_X$ (one for each $X \in \cc{E}$).
\end{proposition}
\begin{proof} 

1) Recall that $T = F$ on $\cc{C}$. 
Let $X \in \cc{E}$, then $TX = FX$, $T'X = F'X$ and let $(a,\,b) \in TX \times T'X$.  By \eqref{yonedaII}, \eqref{masvale} and remark \ref{rhoepi} we can take $C \mr{f} X$ and $c \in TC$ such that $a = T(f)(c) = F(f)(c)$ (see \ref{extensionarel}).
If $\lambda_X$ were defined so that the 
$\lozenge_1(f)$ diagram commutes, the equation 
$$
(1) \hspace{2ex} \lambda_X( a,\,b)  = \bigvee_{y \in T'C} \igu[T'(f)(y)]{b} \cdot \lambda_C( c,\,y)
$$ 
should hold (see \eqref{ecuacionesdiagramas}). We define $\lambda_X$ by this equation. This definition is independent of the choice of $c, \, C,$ and $f$. In fact, let $D \mr{g} X$ and $d \in TD$ be such that 
$a = T(g)(d)$. By remark \ref{FX} we can take  
$(e,\, E)$ in the diagram of $T$ (or $F$), $E \mr{h} C$, $E \mr{\ell} D$ such that 
$T(h)(e) = c$, $T(\ell)(e) = d$ and $fh=g\ell$. Then we compute

\vspace{1ex}

\noindent $ \displaystyle
 \bigvee_{y \in T'C} \igu[T'(f)(y)]{b} \cdot \lambda_C( c,\,y) \stackrel{\lozenge_1(h)}{=} $
 
\vspace{-1.5ex} 
 
$$= \bigvee_{y \in T'C} \;\;\bigvee_{w \in T'E} \igu[T'(f)(y)]{b} \cdot \igu[T'(h)(w)]{y} \cdot \lambda_E( e,\,w) \;= 
$$

\vspace{-1.5ex}

\hfill $ \displaystyle
= \bigvee_{w \in T'E} \igu[T'(fh)(w)]{b} \cdot \lambda_E( e,\,w).
$

From this and the corresponding computation with $d, \, D,$ and $\ell$ it follows:
$$
\bigvee_{y \in T'C} \igu[T'(f)(y)]{b} \cdot \lambda_C( c,\,y) = \bigvee_{y \in T'D} \igu[T'(g)(y)]{b} \cdot \lambda_{D}( d,\,y).
$$
Given $X \mr{g} Y$ in $\cc{E}$, we check that the  
$\lozenge_1(g)$ diagram commutes: Let $(a,\,b) \in TX \times T'Y$, take $C \mr{f} X$, $c \in TC$ such that $a = T(f)(c)$, and let \mbox{$d = T(g)(a) = T(gf)(c)$.} Then 

\vspace{1ex}

$ \displaystyle \lambda_Y ( d,b ) = \bigvee_{z \in T'C} \igu[T'(gf)(z)]{b} \cdot \lambda_C ( c,z ) = $

\vspace{-2ex}

\begin{flushright}
$\displaystyle = \bigvee_{z \in T'C} \bigvee_{x \in X} \igu[T'(f)(z)]{x} \cdot \igu[T'(g)(x)]{b} \cdot \lambda_C ( c,z ) = $
$\displaystyle = \bigvee_{x \in T'X} \igu[T'(g)(x)]{b} \cdot \bigvee_{z \in T'Z} \igu[T'(f)(z)]{x} \cdot \lambda_C ( c,z ) = $
$ \displaystyle \bigvee_{x \in T'X} \igu[T'(g)(x)]{b} \cdot \lambda_X ( a,x ).$
\end{flushright}

Clearly a symmetric argument can be used if we assume at the start that the $\lozenge_2$ diagram commutes. In this case, $\lambda_X$ would be defined by taking  $C \mr{f} X$ and $c \in T'C$ such that $b = T'(f)(c)$ and computing:
$$
(2) \hspace{2ex}  \lambda_X( a,\,b)  = \bigvee_{y \in TC} \igu[T(f)(y)]{a} \cdot \lambda_{C}( y,\,c).
$$  

If the $TC \times T'C \mr{\lambda_C} H$ form a $\lozenge$-cone (i.e. a $\lozenge_1$-cone and a 
$\lozenge_2$-cone), definitions (1) and (2) coincide. In fact, since they are independent of the chosen $c$, it follows they are both equal to: 

\vspace{1ex}

$ \displaystyle
\bigvee_{C \mr{f} X} \bigvee_{c \in TC} \bigvee_{y \in T'C} \igu[T(f)(c)]{a} \cdot \igu[T'(f)(y)]{b} \cdot \lambda_C( c,\,y) 
\hspace{2ex} = \hspace{2ex} $

\vspace{-1.5ex}

\hfill $ \displaystyle
\bigvee_{C \mr{f} X} \bigvee_{c \in T'C} \bigvee_{y \in TC} \igu[T'(f)(c)]{b} \cdot \igu[T(f)(y)]{a} \cdot \lambda_C( y,\,c).
$

\vspace{1ex}

2) 
It suffices to prove that if $\lambda_C$ (one for each $C \in \cc{C}$) is a $\lozenge_1$-cone of $\ell$-functions, so is $\lambda_X$ (one for each $X \in \cc{X}$). Let $X \in \cc{E}$, $a \in TX$, $b_1,b_2 \in T'X$. Take as in item 1. $C \mr{f} X$ and $c \in TC$ such that $a = T(f)(c)$.

\vspace{.5cm}

\noindent$ed) \displaystyle \! \bigvee_{b \in T'X} \lambda_X( a,\,b)  =  \! \bigvee_{b \in T'X} \! \bigvee_{y \in T'C} \igu[T'(f)(y)]{b} \cdot \lambda_C( c,\,y) = \! \bigvee_{y \in T'C} \lambda_C( c,\,y) \stackrel{ed)}{=} 1$

\vspace{.5cm}

\noindent $uv)$ $\displaystyle \lambda_X(a,b_1) \wedge \lambda_X(a,b_2) = $

\hfill $\displaystyle  \bigvee_{y_1,y_2 \in T'C} \igu[T'(f)(y_1)]{b_1} \cdot \igu[T'(f)(y_2)]{b_2} \cdot \lambda_C(c,y_1) \wedge \lambda_C(c,y_2) \stackrel{uv)}{\leq} $

\hfill $\displaystyle  \bigvee_{y_1,y_2 \in T'C} \igu[T'(f)(y_1)]{b_1} \cdot \igu[T'(f)(y_2)]{b_2} \cdot \igu[y_1]{y_2} \stackrel{\ref{prelim1}}{\leq} $

\hfill $\displaystyle \! \! \! \bigvee_{y_1,y_2 \in T'C} \! \! \! \igu[T'(f)(y_1)]{b_1} \! \cdot \!  \igu[T'(f)(y_2)]{b_2} \! \cdot \! \igu[T'(f)(y_1)]{T'(f)(y_2)} \stackrel{\ref{prelim3}}{\leq} \igu[b_1]{b_2}$.

\end{proof}

\begin{assumption}
For the rest of this section we consider a small site $\cc{C}$ (with binary products and $1$) of the topos $\cc{E}$, and cones defined over $\cc{C}$.
\end{assumption}

We now introduce the notion of compatible cone. It is a very useful notion to obtain results for locales from results for sup-lattices, as the following propositions show. Any compatible $\lozenge$-cone which covers a commutative algebra $H$ forces $H$ to be a locale, and such a cone is necessarily a cone of $\ell$-bijections (and vice versa):

\begin{definition} [cf. definition \ref{neutralcomp}] \label{comp}
Let $H$ be a commutative algebra in $s \ell$, with multiplication $*$ and unit $u$ (We consider $H \times H \mr{*} H$ bilinear and thus inducing $H \otimes H \mr{*} H$, and $u$ given by $u \in H$, i.e. $1 \mr{u} H$ inducing a linear morphism $\Omega \mr{u} H$). 

Let $TC \times T'C \mr{\lambda_C} H$ be a cone. 
We say that $\lambda$ is \emph{compatible} if the following equations hold: 

\flushleft $\displaystyle [C1] \hbox{For each } a \in TC, a' \in T'C, b \in TD, b' \in T'D,$

 \flushright $\displaystyle \lambda_C ( a,\,a' ) *  \lambda_D ( b,\,b' ) = \lambda_{C \times D}( (a,\,b),\,(a',\,b') ) \ ;$

\flushleft $\displaystyle [C2] \hspace{30ex} \lambda_1 = u.$
\end{definition}

Given a compatible cone, consider the diagonal $C \mr{\Delta} C \times C$, the arrow $C \mr{\pi} 1$, and the following $\lozenge_1$ diagrams (see \ref{diagramadiamante12}):
$$
\xymatrix@R=4ex@C=2ex
        { 
         & TC \!\times\! T'C \ar[rd]^{\lambda_C} 
         & & & TC \!\times\! T'C \ar[rd]^{\lambda_C} 
        \\
	     TC \!\times\! (T'C \!\times\! T'C) 
	                      \ar[ru]^{TC \!\times\! \Delta^{op}} 
	                      \ar[rd]_{\Delta \!\times\! (T'C \!\times\! T'C) \ \ } 
	     & \equiv 
	     & \hspace{0ex} H, 
	     & TC \!\times\! 1 \ar[ru]^{TC \!\times\! \pi^{op}} 
	                  \ar[rd]_{\pi \!\times\! 1} 
	     & \!\! \equiv 
	     & \hspace{0ex} H.
	    \\
	     & (TC \!\times\! TC) \!\times\! (T'C \!\times\! T'C) 
	                          \ar[ru]_{\lambda_{C \!\times\! C}} 
	     & & & 1 \!\times\! 1 \ar[ru]_{\lambda_1} 
	    }
$$

expressing the equations: for each $a \in TC, b_1,b_2 \in T'C,$

\vspace{1ex}

\noindent $\displaystyle \lozenge_1(\triangle) \! : \quad   \lambda_{C \times C}( (a,a),(b_1,b_2) ) \, \stackrel{}{=} \bigvee_{x \in T'C} \igu[(x,x)]{(b_1,b_2)} \cdot \lambda_C ( a,\,x)$,

\vspace{1ex} 

\noindent $\displaystyle \lozenge_1(\pi) \! : \quad   \lambda_1 = \bigvee_{x \in T'C} \lambda_C( a,\,x)  $.

\begin{lemma} \label{lemaparaconocomp}
 Let $TC \times T'C \mr{\lambda} H$ be a compatible $\lozenge_1$-cone (or $\lozenge_2$-cone, or $\lozenge$-cone) with vertex a commutative algebra $H$. Then, for each $a \in TC, b_1,b_2 \in T'C,$
\begin{enumerate}
 \item $\displaystyle \lambda_C ( a,b_1 ) *  \lambda_C ( a,b_2 ) = \igu[b_1]{b_2} \cdot \lambda_C(a,b_1)$.
 \item $\displaystyle u = \bigvee_{x \in T'C} \lambda_C( a,\,x)$.
\end{enumerate}
\end{lemma}
\begin{proof}
 2. is immediate from [C2] and $\lozenge_1(\pi)$ above. To prove 1. we compute
 
  \noindent $\displaystyle \lambda_C ( a,b_1 ) *  \lambda_C ( a,b_2 ) \stackrel{[C1]}{=} \lambda_{C \times C}( (a,a),(b_1,b_2) ) \stackrel{\lozenge_1(\triangle)}{=} \bigvee_{x \in T'C} \igu[x]{b_1} \cdot \igu[x]{b_2} \cdot \lambda_C ( a,\,x) \stackrel{\ref{ecuacionenL}}{=} $

  \flushright $\displaystyle = \bigvee_{x \in T'C} \igu[x]{b_1} \cdot \igu[b_1]{b_2} \cdot \lambda_C ( a,\,b_1) = \displaystyle \igu[b_1]{b_2} \cdot \lambda_C ( a,\,b_1) $.
  
\end{proof}

\begin{proposition} [cf. proposition \ref{neutralcompislocale}] \label{compislocale}
Let $\lambda$ be a compatible $\lozenge$-cone with vertex a commutative algebra $(H,*)$ such that the elements of the form $\lambda_C( a,\,b )$, $a \in TC, b \in T'C$ are sup-lattice generators of $H$. Then $H$ is a locale and $*=\wedge$. 
\end{proposition}
\begin{proof}
The same proof of proposition \ref{neutralcompislocale} can be used, replacing equations (1) and (2) by lemma \ref{lemaparaconocomp}
\end{proof}

\begin{proposition} [cf. proposition \ref{neutralDiamondisbijection}] \label{Diamondisbijection}  
Conider a cone $\lambda$ with vertex a locale $H$. 
\begin{enumerate}
 \item If $\lambda$ is a $\lozenge_1$-cone, then $\lambda$ is compatible if and only if it is a $\lozenge_1$-cone of $\ell$-functions.
 \item If $\lambda$ is a $\lozenge_2$-cone, then $\lambda$ is compatible if and only if it is a $\lozenge_2$-cone of \linebreak $\ell$-op-functions.
 \item If $\lambda$ is a $\lozenge$-cone, then $\lambda$ is compatible if and only if it is a $\lozenge$-cone of $\ell$-bijections.
\end{enumerate}

\end{proposition}
\begin{proof} We prove 1, 2 follows by symmetry and combining 1 and 2 we obtain 3.

($\Rightarrow$): Since $\wedge=*$ and $1=u$ in $H$, equations 1. and 2. in lemma \ref{lemaparaconocomp} become the axioms ed) and uv) for $\lambda_X$. 

($\Leftarrow$) $u = 1$ in $H$, so equation [C2] in definition \ref{comp} is axiom $ed)$ for $\lambda_1$. To prove equation [C1] we consider the projections $C \times D \mr{\pi_1} C$, $C \times D \mr{\pi_2} D$. The $\lozenge_1(\pi_1)$  and $\lozenge_1(\pi_2)$ diagrams express the equations:

\vspace{1ex}

\noindent For each $a \in TC, b \in TD, a' \in T'C, \:\:
 \lambda_C( a,a') = \displaystyle \bigvee_{y \in T'D} \lambda_{C \times D} ( (a,b),(a',y) ),$
 
\noindent For each $a \in TC, b \in TD, b' \in T'D, \:\:
 \lambda_D ( b,b') = \displaystyle \bigvee_{x \in T'C} \lambda_{C \times D} ( (a,b),(x,b') ).
$

\vspace{1ex}

Taking the infimum of these two equations we obtain for each $a \in TC, b \in TD, \linebreak a' \in T'C, b' \in T'D$:

\vspace{1ex}

\noindent
$
\lambda_C ( a,a' ) \wedge \lambda_D ( b,b' )  \;=\; \displaystyle \bigvee_{x \in T'C} \bigvee_{y \in T'D} \lambda_{C \times D} ( (a,b),(a',y) )  \wedge \lambda_{C \times D} ( (a,b),(x,b') ) =
$

\flushright $\stackrel{uv)_{\lambda_{C \times D}}}{\;=\;}  \displaystyle \bigvee_{x \in T'C} \bigvee_{y \in T'D} \igu[(a',y)]{(x,b')} \cdot \lambda_{C \times D} ( (a,b),(a',y) ) \stackrel{\ref{ecuacionenL}}{=} \lambda_{C \times D} ( (a,b),(a',b') )$

\end{proof}

Also, sup-lattice morphisms of cones with compatible domain are automatically locale morphisms:

\begin{proposition} 
Let $\lambda$ be a  compatible cone with vertex a locale $H$ such that the elements of the form $\lambda_C( a,a' )$, $a \in TC, a' \in T'C$ are sup-lattice generators of $H$. Let $\lambda$ be another compatible cone with vertex a locale $H'$. Then, any sup-lattice morphism $H \mr{\sigma} H'$ satisfying $ \sigma \lambda_C = \lambda_C$ is a locale morphism.
\end{proposition}

\begin{proof}
Equation [C2] in defintion \ref{comp} implies immediately that $\sigma u = u'$ (i.e. $\sigma$ preserves $1$).

Equation [C1] implies immediately that the infima $\wedge$ between two sup-lattice generators $\lambda_C( a,a' )$ and $\lambda_D( b,b' )$ is preserved by $\sigma$, which suffices to show that $\sigma$ preserves $\wedge$ between two arbitrary elements since $\sigma$ is a sup-lattice morphism.
\end{proof}

Combining the previous proposition with proposition \ref{Diamondisbijection} we obtain

\begin{corollary} [cf. proposition \ref{neutralsupisloc}] \label{supisloc}
Let $\lambda$ be a $\lozenge$-cone of $\ell$-bijections with vertex a locale $H$ such that the elements of the form $\lambda_C( a,\,b )$, $a \in TC, b \in T'C$ are sup-lattice generators of $H$. Let $\lambda$ be another $\lozenge$-cone of $\ell$-bijections with vertex a locale $H'$. Then, any sup-lattice morphism $H \mr{\sigma} H'$ satisfying $ \sigma \lambda_C = \lambda_C$ is a locale morphism. \qed
\end{corollary}

\pagebreak

\section{The case $\Eat = sh P$} \label{sub:EshP}

\begin{sinnadastandard} \label{enumeratedeJT} Assume now we have a \emph{base topos} $\Sat$, a locale $P \in Loc := Loc(\Sat)$ and we consider $\Eat = sh P$. We recall from \cite{JT}, VI.2 and VI.3, p.46-51, the different ways in which we can consider objects, sup-lattices and locales in $\Eat$.

\begin{enumerate}
 \item We consider the inclusion of topoi $shP \hookrightarrow \Sat^{P^{op}}$ given by the adjunction $\# \dashv i$. A sup-lattice $M \in s \ell(shP)$ yields a sup-lattice $iM \in \Sat^{P^{op}}$, in which the supremum of a sub-presheaf $S \mr{} iM$ is computed as the supremum of the corresponding sub-sheaf $\#S \mr{} M$ (see \cite{JT}, VI.1 Proposition 1 p.43). The converse actually holds, i.e. if $iM \in s \ell(\Sat^{P^{op}})$ then $M \in s \ell(shP)$, see \cite{JT}, VI.3 Lemma 1 p.49.
 
 \item We omit to write $i$ and consider a sheaf $M \in shP$ as a presheaf $P^{op} \mr{M} \Sat$ that is a sheaf, i.e. that believes covers are epimorphic families. A sup-lattice structure for $M \in shP$ corresponds in this way to a sheaf $P^{op} \mr{M} s \ell$ satisfying the following two conditions (these are the conditions 1) and 2) in \cite{JT}, VI.2 Proposition 1 p.46 for the particular case of a locale):
 
 \begin{itemize}
  \item[a)] For each $p' \leq p$ in $P$, the $s \ell$-morphism $M_{p'}^p: M(p) \mr{} M(p')$, that we will denote by $\rho_{p'}^p$, has a left adjoint $\Sigma_{p'}^p$. 
  \item[b)] For each $q \in P$, $p \leq q$, $p' \leq q$, we have $\rho_{p'}^q \Sigma_p^q = \Sigma_{p \wedge p'}^{p'} \rho_{p \wedge p'}^p$.
 \end{itemize}

 Sup-lattice morphisms correspond to natural transformations that commute with the $\Sigma$'s. 
 
  When interpreted as a presheaf, $\Omega_P(p) = P_{\leq p} := \{ q \in P | q \leq p\}$, with $\rho_q^p = (-) \wedge q$ and $\Sigma_q^p$ the inclusion. The unit $1 \mr{1} \Omega_P$ is given by $1_p = p$.
 
 \item If $M \in s \ell(\cc{S}^{P^{op}})$ (in particular if $M \in s \ell(shP)$), the supremum of a sub-presheaf $S \mr{} M$ can be computed in $\cc{S}^{P^{op}}$ as the global section $1 \mr{s} M$, $s_q = \displaystyle \bigvee_{\stackrel{p \leq q}{x \in S(p)}} \Sigma_p^q x$ (see \cite{JT}, VI.2 proof of proposition 1, p.47).
 
 \item Locales $L$ in $shP$ correspond to sheaves $P^{op} \mr{L} Loc$ such that, in addition to the $s \ell$ condition, satisfy Frobenius reciprocity: if $q \leq p$, $x \in L(p)$, $y \in L(q)$, then $\Sigma^p_{q} (\rho^p_{q} (x) \wedge y) = x \wedge \Sigma^p_{q} y$. 
 
 Note that since $\rho \Sigma = id$, Frobenius implies that if $q \leq p$, $x,y \in L(q)$ then \linebreak $\Sigma^p_{q} (x \wedge y) = \Sigma^p_{q} (\rho^p_{q} \Sigma^p_{q} (x) \wedge y) = \Sigma^p_{q} x \wedge \Sigma^p_{q} y$, in other words that $\Sigma$ commutes with 
$\wedge$.
 
 \item The direct image functor establishes an equivalence of tensor categories \linebreak $(s \ell(sh P),\otimes) \mr{\gamma_*} (P$-$Mod,\otimes_P)$ (\cite{JT}, VI.3 Proposition 1 p.49), given $G \in s \ell(sh P)$ and $p \in P$ multiplication by $p$ in $\gamma_* G = G(1)$ is given by $\Sigma_p^1 \rho_p^1$ (\cite{JT}, VI.2 Proposition 3 p.47). 
 
 The pseudoinverse of this equivalence is $P$-$Mod \mr{\widetilde{(-)}} s \ell(shP)$, $N \mapsto \widetilde{N}$ defined by $\widetilde{N}(p) = \{x \in N | p \cdot x = x\}$ for $p \in P$.
 
 \item The equivalence of item 5 restricts to an equivalence $Loc(sh P) \mr{\gamma_*} P$-$Loc$, where the last category is the category of locale extensions $P \mr{} L$ (\cite{JT}, VI.3 Proposition 2 p.51).
 
 \end{enumerate}

\end{sinnadastandard}

\begin{sinnadastandard} \label{relativizarellrelations} 
We will now consider relations in the topos $shP$ and prove that $\ell$-functions in $P$ correspond to functions in $shP$, and therefore to arrows of the topos $shP$.

\noindent The unique locale morphism $\Omega \mr{\gamma} P$ induces a topoi morphism $\xymatrix{\Sat \cong sh\Omega \ar@/^2ex/[r]^{\gamma^*} \ar@{}[r]|{\bot} & sh P \ar@/^2ex/[l]^{\gamma_*} }$. Let's denote by $\Omega_P$ the subobject classifier of $sh P$. Since $\gamma_* \Omega_P = P$, we have the correspondence

\begin{center}
 \begin{tabular}{c}
  $X \times Y \mr{\lambda} P$ an $\ell$-relation \\ \hline \noalign{\smallskip}
  $\gamma^*Y \times \gamma^*X \mr{\varphi} \Omega_P$ a relation in $sh P$
 \end{tabular}
\end{center}

\begin{\prop} \label{prop:relativizarellrelations}
In this correspondence, $\lambda$ is an $\ell$-function if and only if $\varphi$ is a function. Then, by proposition \ref{functionconallegories}, 
$\ell$-functions correspond to arrows $\gamma^*X \mr{\varphi} \gamma^*Y$ in the topos $shP$, and by remark \ref{simetria} $\ell$-bijections correspond to isomorphisms.
\end{\prop}

\begin{proof}
 
Consider the extension $\widetilde{\lambda}$ of $\lambda$ as a $P$-module, and $\widetilde{\varphi}$ of $\varphi$ as a $\Omega_P$-module, i.e. in $s \ell(shP)$ (we add the $\widetilde{(-)}$ to avoid confusion). 
We have the binatural correspondence between $\widetilde{\lambda}$ and $\widetilde{\varphi}$:

\begin{center}
 \begin{tabular}{c}
  $\xymatrix@C=3.5pc{  X \times Y \ar@/^4ex/[rr]^{\lambda}  \ar[r]_{\{\}_P \stackrel[P]{}{\otimes} \{\}_P }  & P^X \stackrel[P]{}{\otimes} P^Y \ar[r]_>>>>>>>>{\widetilde{\lambda}} & P }$   \\ \hline \noalign{\smallskip}
  $\xymatrix@C=3pc{  \gamma^* X \times  \gamma^*Y  \ar@/_4ex/[rr]_{\varphi} \ar[r]^{\{\} \otimes \{\}}  & \Omega_P^{\gamma^*X} \otimes \Omega_P^{\gamma^*Y} \ar[r]^>>>>>>>>{\widetilde{\varphi}} & \Omega_P} $
 \end{tabular}
\end{center}

given by the adjunction $\gamma^* \dashv \gamma_*$. But $\gamma_*(\Omega_P^{\gamma^*X}) = (\gamma_* \Omega_P)^X = P^X$ and $\gamma_*$ is a tensor functor, then $\gamma_*(\Omega_P^{\gamma^*X} \otimes \Omega_P^{\gamma^*Y}) = P^X \stackrel[P]{}{\otimes} P^Y$ and $\gamma_*(\widetilde{\varphi}) = \widetilde{\lambda}$.

Now, the inverse images $\lambda^*,\varphi^*$ are constructed from $\widetilde{\lambda}, \widetilde{\varphi}$ using the autoduality of $\Omega_P^{\gamma_*X}, P^X$ (see proposition  \ref{prop:imageninversausandoautodual}), and since $\gamma^*$ is a tensor functor that maps $\Omega_P^{\gamma_*X} \mapsto P^X$ we can take $\eta$, $\eps$ of the autoduality of $P^X$ as $\gamma^*(\eta')$, $\gamma^*(\eps')$ if $\eta'$, $\eps'$ are the autoduality structure of $\Omega_P^{\gamma_*X}$. It follows that $\gamma_*(\varphi^*) = \lambda^*$, then by \ref{enumeratedeJT} (item 6) we obtain that $\varphi^*$ is a locale morphism if and only if $\lambda^*$ is so. Proposition \ref{edyuvdafunctionG} finishes the proof.
\end{proof}

Consider now the situation of \ref{sin:naturaligualconeconrho} for the case $\cc{G} = shP$, i.e. assume we have 
$$\xymatrix{& shP  \adjuntosd{\gamma^*}{\gamma_*}  \\ \cc{E} \dosflechasr{F}{F'} & \cc{S}. }$$
Combining proposition \ref{prop:relativizarellrelations} with \ref{naturaligualconeconrho} we obtain:

\begin{corollary} \label{conodeellrelationsesnattransf}
 There is a bijective correspondence given by the adjunction $\gamma^* \dashv \gamma_*$ between $\lozenge_1$-cones of $\ell$-functions (resp $\ell$-bijections) $FX \times F'X \mr{\lambda_X} P$ and natural transformations (resp. isomorphisms) $\gamma^* F \Mr{\varphi} \gamma^* F'$. \qed
\end{corollary}

\begin{remark}
 Though we will not use the result with this generality, we note that proposition \ref{prop:relativizarellrelations} (and therefore corollary \ref{conodeellrelationsesnattransf}) also holds for an arbitrary topos $\cc{G}$. Consider $P = \gamma_* \Omega_{\cc{G}}$, the hyperconnected factorization $\vcenter{\xymatrix{\cc{G} \ar[rr]^q \ar[rd]_{\gamma} && shP \ar[dl]^{\gamma} \\ & \cc{S} }}$ (see \cite{JT}, VI. 5 p.54) and recall that $q_* \Omega_{\cc{G}} \cong \Omega_{P}$ and that the counit map $q^* q_* \Omega_{\cc{G}} \mr{} \Omega_{\cc{G}}$ is, up to isomorphism, the comparison morphism $q^* \Omega_{P} \mr{} \Omega_{\cc{G}}$ of remark \ref{caractsubobjeto} (see \cite{JohnstoneFactorization}, 1.5, 1.6). The previous results imply that the correspondence between relations $X \times Y \mr{} \Omega_P$ and relations $q^* X \times q^* Y \mr{} \Omega_{\cc{G}}$ given by the adjunction $q^* \dashv q_*$ is simply the correspondence between a relation $R \hookrightarrow X \times Y$ in $shP$ and its image by the full and 
faithful morphism $q^*$, therefore functions correspond to functions. Since by proposition \ref{prop:relativizarellrelations} we know that the same happens for $shP \mr{\gamma} \cc{S}$, by composing the adjunctions we obtain it for $\cc{G} \mr{\gamma} \cc{S}$.
 \end{remark}

\end{sinnadastandard}

\begin{\de} \label{abusosvarios}
Let $p \in P$, we identify by Yoneda $p$ with the representable presheaf $p = [-,p]$. If $q \in P$, then $[q,p] = \llbracket q \! \leq \! p \rrbracket \in \Omega$. In particular if $a \leq p$ then $[a,p] = 1$. 

For $a \leq p \in P$, $x \in X(p)$, consider $X(p) \mr{X_a^p} X(a)$ in $\Sat$. We will denote $x|_a := X_a^p(x)$. 
\end{\de}

We describe now the sup-lattice structure of the exponential $G^X$. Recall that as a presheaf, $G^X(p) = [p \times X, G]$, and note that if $\theta \in G^X(p)$, and $a \leq p$, by definition \ref{abusosvarios} we have $X(a) \mr{\theta_a} G(a)$.

$\theta$ corresponds via the exponential law to $X \mr{\hat{\theta}} G^p$, $$X(q) \mr{\hat{\theta}_q} G^p(q) \cong [q \wedge p, G] \cong G(q \wedge p)$$ by Yoneda lemma. Following $\theta$ through this correspondences, it follows that \linebreak $X(q) \mr{\hat{\theta}_q} G(q \wedge p)$ is defined by $\hat{\theta}_q(x) = \theta_{q \wedge p} (x|_{q \wedge p})$. 

This implies that $\theta \in G^X(p)$ is completely characterized by its components $\theta_a$ for $a \leq p$. From now on we make this identification, i.e. we consider $\theta \in G^X(p)$ as a family $\{X(a) \mr{\theta_a} G(a)\}_{a \leq p}$ natural in $a$. Via this identification, if $q \leq p$, it can be checked that the morphism $G^X(p) \mr{\rho_q^p} G^X(q)$ is given by $\{X(a) \mr{\theta_a} G(a)\}_{a \leq p} \mapsto \{X(a) \mr{\theta_a} G(a)\}_{a \leq q}$.

\begin{lemma} \label{exponencialenshP}
 Let $X \in shP$, $G \in s \ell(shP)$. Then the sup-lattice structure of $G^X$ is given as follows: 
 \begin{enumerate}
  \item For each $p \in P$, $G^X(p) = \{ \{X(a) \mr{\theta_a} G(a)\}_{a \leq p} \hbox{ natural in a}\}$ is a sup-lattice pointwise. 
  \item If $q \leq p$ the morphisms $\xymatrix{G^X(q) \adjuntos[\Sigma_q^p]{\rho_q^p} & G^X(p)}$ are defined by the formulae \linebreak (for $\theta \in G^X(p)$, $\xi \in G^X(q)$):

 \noindent $F\rho) \quad \quad (\rho_q^p \theta)_{a} (x) = \theta_{a} (x)$ for $x \in X(a)$, $a \leq q$.
 
 \noindent $F\Sigma) \quad \quad (\Sigma_q^p \xi)_{a} (x) = \Sigma_{a \wedge q}^{a} \xi_{a \wedge q} (x|_{a \wedge q})$ for $x \in X(a)$, $a \leq p$.
\end{enumerate}
 \end{lemma}

\begin{proof}
 We have already showed above that $\rho_q^p$ satisfies $F\rho)$.
 
 We have to prove that if $\Sigma_q^p$ is defined by $F\Sigma)$ then the adjunction holds, i.e. that $A: \Sigma_q^p \xi \leq \theta$ if and only if $B: \xi \leq \rho_q^p \theta$. 
 
 \noindent By $F\Sigma)$, $A$ means that for each $a \leq p$, for each $x \in X(a)$ we have $\Sigma_{a \wedge q}^{a} \xi_{a \wedge q} (x|_{a \wedge q}) \leq \theta_{a}(x)$ in $G(a)$.
 
 \noindent By $F\rho)$, $B$ means that for each $a \leq q$, for each $x \in X(a)$ we have $\xi_{a}(x) \leq \theta_{a}(x)$. 
 
 Then $A$ implies $B$ since if $a \leq q$ then $a \wedge q = a$, and $B$ implies $A$ since for each $a \leq p$, for each $x \in X(a)$, by the adjunction $\Sigma \dashv \rho$ for $G$, $\Sigma_{a \wedge q}^{a} \xi_{a \wedge q} (x|_{a \wedge q}) \leq \theta_{a}(x)$ holds in $G(a)$ if and only if $\xi_{a \wedge q} (x|_{a \wedge q}) \leq \rho_{a \wedge q}^{a} \theta_{a}(x)$ holds in $G(a \wedge q)$, but this inequality is implied by $B$ since by naturality of $\theta$ we have $\rho_{a \wedge q}^{a} \theta_{a}(x) = \theta_{a \wedge q} (x|_{a \wedge q})$.
\end{proof}

\begin{\prop} 
 If $X \in shP$, $G \in Loc(shP)$, then the sup-lattice structure of $G^X$ defined above satisfies Frobenius reciprocity as in \ref{enumeratedeJT} item 4, yielding in this way a locale structure for $G^X$.
\end{\prop}

\begin{proof}
For $q \leq p$, $\theta \in G^X(p)$, $\xi \in G^X(q)$, we have to check  $\Sigma^p_{q} (\rho^p_{q} (\theta) \wedge \xi) = \theta \wedge \Sigma^p_{q} \xi$. By $F\rho)$ and $F\Sigma)$ above, it suffices to check that for each $a \leq p$, $x \in X(a)$,

$$\Sigma_{a \wedge q}^a (\theta_{a \wedge q} (x|_{a \wedge q}) \wedge \xi_{a \wedge q}(x|_{a \wedge q})) = \theta_a(x) \wedge \Sigma_{a \wedge q}^a \xi_{a \wedge q}(x|_{a \wedge q}),$$

which follows from Frobenius reciprocity (for $G$) with $x = \theta_a(x)$, $y = \xi_{a \wedge q}(x|_{a \wedge q})$.
\end{proof}

\begin{remark} \label{1deGalaX}
If $X \in shP$, $G \in Loc(shP)$, the unit $1 \in G^X$ is a global section that corresponds to the arrow $X \mr{} 1 \mr{1} \Omega_P \mr{} G$, which by \ref{enumeratedeJT} item 2 maps $1_p(x) = p$ for each $p \in P$, $x \in X(p)$.
\end{remark}

\begin{sinnadastandard} For the remainder of this section, the main idea (that shouldn't be lost in the computations) is to consider some of the situations defined in section \ref{sec:ellrel} for the topos $shP$, and to ``translate'' them to the base topos $\Sat$. In particular we will translate the four axioms for an $\ell$-relation in $shP$ (which are expressed in the internal language of the topos $shP$) to equivalent formulae in the language of $\Sat$ (proposition \ref{propaxiomparamodulos}), and also translate the autoduality of $G^X$, if $G \in s \ell(shP)$, to an autoduality of $P$-modules (proposition \ref{formulaeetaeps}). All this will be needed later in section \ref{equiv:objects}. 
\end{sinnadastandard}

 Consider $X \in shP$, $G \in s \ell(shP)$ and an arrow $X \mr{\alpha} G$. We want to compute the internal supremum $\displaystyle \bigvee_{x \in X} \alpha(x) \in G$. This supremum is the supremum of the subsheaf of $G$ given by the image of $\alpha$ in $shP$, which is computed as $\#S \hookrightarrow G$, where $S$ is the sub-presheaf of $G$ given by $S(p) = \{\alpha_p(x) \ | \ x \in X(p)\}$. Now, by \ref{enumeratedeJT} item 1 (or, it can be easily verified), this supremum coincides with the supremum of the sub-presheaf $S \hookrightarrow G$, which by \ref{enumeratedeJT} item 3 is computed as the global section $1 \mr{s} G$, $s_q = \displaystyle \bigvee_{\stackrel{p \leq q}{x \in X(p)}} \Sigma_p^q \alpha_p(x)$. Applying the equivalence $\gamma_*$ of \ref{enumeratedeJT}, item 5 we obtain:

\begin{\prop} \label{calcsup} Let $X \in shP$, $G \in s \ell(shP)$ and an arrow $X \mr{\alpha} G$. Then at the level of $P$-modules, the element $s \in G(1)$ corresponding to the internal supremum $\displaystyle \bigvee_{x \in X} \alpha(x)$ is $\displaystyle \bigvee_{\stackrel{p \in P}{x \in X(p)}} \Sigma_p^1 \alpha_p(x) $. \qed
 \end{\prop}

\begin{definition} \label{defdeXd}
 Given $X \in shP$, recall that we denote by $\Omega_P$ the object classifier of $shP$ and consider the sup-lattice in $shP$, $\Omega_P^X$ (that is also a locale). We will denote by $X_d$ the $P$-module (that is also a locale extension $P \mr{} X_d$) corresponding to $\Omega_P^X$, in other words $X_d := \gamma_*(\Omega_P^X) = \Omega_P^X(1)$.
 
 Given $p \in P$, $x \in X(p)$ we define the element $\delta_x := \Sigma_p^1 \{x\}_p \in X_d$.
\end{definition}

Consider now $\theta \in X_d$, that is $\theta \in \Omega_P^X(1)$, i.e. $X \mr{\theta} \Omega_P$ in $shP$. Let $\alpha$ be $X \mr{\theta \cdot \{\}} \Omega_P^X$, $\alpha(x) = \theta(x) \cdot \{x\}$. Then proposition \ref{formulainterna} states that $\theta = \displaystyle \bigvee_{x \in X} \alpha(x)$ (this is internally in $shP$). Appyling proposition \ref{calcsup} we compute in $X_d$:

$$\theta = \bigvee_{\stackrel{p \in P}{x \in X(p)}} \Sigma_p^1 (\theta_p(x) \cdot \{x\}_p) = \bigvee_{\stackrel{p \in P}{x \in X(p)}} \theta_p(x) \cdot \Sigma_p^1 \{x\}_p = \bigvee_{\stackrel{p \in P}{x \in X(p)}} \theta_p(x) \cdot \delta_x.$$

We have proved the following:

\begin{\prop} \label{formula}
The family $\{\delta_x\}_{p \in P, x \in X(p)}$ generates $X_d$ as a $P$-module, and furthermore,
for each $\theta \in X_d$, we have $\theta = \displaystyle \bigvee_{\stackrel{p \in P}{x \in X(p)}} \theta_p(x) \cdot \delta_x$. \qed
\end{\prop}

\begin{remark} \label{pavada}
 Given $q \leq p \in P$, $x \in X(p)$, by naturality of $X \mr{\{\}} \Omega_P^X$ we have $\{x|_q\}_q = \rho_q^p \{x\}_p$.
\end{remark}

\begin{lemma} \label{restringirlosdelta}
 For $p,q \in P$, $x \in X(p)$, we have $q \cdot \delta_x = \delta_{x|_{p \wedge q}}$. In particular $p \cdot \delta_x = \delta_x$.
\end{lemma}

\begin{proof}
 Recall that multiplication by $a \in P$ is given by $\Sigma_a^1 \rho_a^1$, and that $\rho_a^1 \Sigma_a^1 = id$. Then $p \cdot \delta_x = \Sigma_p^1 \rho_p^1 \Sigma_p^1 \{x\}_p = \Sigma_p^1 \{x\}_p = \delta_x$, and
 
 \flushleft $q \cdot \delta_x = q \cdot p \cdot \delta_x = (p \wedge q) \cdot \delta_x = \Sigma_{p \wedge q}^1 \rho_{p \wedge q}^1 \Sigma_p^1 \{x\}_p = $
 
 \hfill $ = \Sigma_{p \wedge q}^1 \rho_{p \wedge q}^p \rho_p^1 \Sigma_p^1 \{x\}_p =  \Sigma_{p \wedge q}^1 \rho_{p \wedge q}^p \{x\}_p \stackrel{\ref{pavada}}{=} \Sigma_{p \wedge q}^1 \{x|_{p \wedge q}\}_{p \wedge q} = \delta_{x|_{p \wedge q}}. $ 
 \end{proof}

\begin{corollary} \label{pyqsaltan}
 For $X,Y \in shP$, $p,q \in P$, $x \in X(p)$, $y \in Y(q)$, we have \linebreak $\delta_x \otimes \delta_y = \delta_{x|_{p \wedge q}} \otimes \delta_{y|_{p \wedge q}}$ in $X_d \otimes_P Y_d$.
\end{corollary}

\begin{proof}
 $\delta_x \otimes \delta_y = p \cdot \delta_x \otimes q \cdot \delta_y = q \cdot \delta_x \otimes p \cdot \delta_y =  \delta_{x|_{p \wedge q}} \otimes \delta_{y|_{p \wedge q}}$.
\end{proof}

\begin{\de} \label{notacioncorchetes}
Consider now $X \times X \mr{\delta_X} \Omega_P$ in $shP$, for each $a \in P$ we have $$X(a) \times X(a) \mr{{\delta_X}_a} \Omega_P(a) \stackrel{\ref{enumeratedeJT}, \ item \ 2.}{=} P_{\leq a}.$$
If $x \in X(p)$, $y \in X(q)$ with $p,q \in P$, we denote $$\llbracket x \! = \! y \rrbracket_P := \Sigma_{p \wedge q}^1 {\delta_X}_{p \wedge q} (x|_{p \wedge q}, y|_{p \wedge q}) \in P.$$
\end{\de}

This shouldn't be confused with the internal (in $shP$) notation $\llbracket x \! = \! y \rrbracket$ introduced in section \ref{11}, though it is similar to how one would compute it using sheaf semantics; here all these computations are ``external'', i.e. in $\Sat$. 

\begin{corollary} \label{ecuacionenOmegaX}
 For $p,q \in P$, $x \in X(p)$, $y \in X(q)$, we have $ \llbracket x \! = \! y\rrbracket_P  \cdot \delta_x = \llbracket x \! = \! y\rrbracket_P  \cdot \delta_y$. 
\end{corollary}
\begin{proof}
 Applying lemma \ref{ecuacionenL} to $X \mr{\{ \}} \Omega_P^X$ it follows that for each $p,q \in P$, $x \in X(p)$, $y \in X(q)$, $${\delta_X}_{p \wedge q} (x|_{p \wedge q}, y|_{p \wedge q})  \cdot \{x|_{p \wedge q}\}_{p \wedge q} = {\delta_X}_{p \wedge q} (x|_{p \wedge q}, y|_{p \wedge q})  \cdot \{y|_{p \wedge q}\}_{p \wedge q}$$
 in $\Omega_P^X(p \wedge q)$, where $\cdot$ is the $p \wedge q$-component of the natural isomorphism $\Omega_P \otimes \Omega_P^X \mr{\cdot} \Omega_P^X$. Apply now $\Sigma_{p \wedge q}^1$ and use that $\cdot$ is a $s \ell$-morphism (therefore it commutes with $\Sigma$) to obtain 
 $$\llbracket x \! = \! y\rrbracket_P  \cdot \delta_{x|_{p \wedge q}} = \llbracket x \! = \! y\rrbracket_P  \cdot \delta_{y|_{p \wedge q}}.$$
 Then, by lemma \ref{restringirlosdelta},
 $$\llbracket x \! = \! y\rrbracket_P  \cdot q \cdot \delta_{x} = \llbracket x \! = \! y\rrbracket_P  \cdot p \cdot \delta_{y},$$
 which since $\llbracket x \! = \! y\rrbracket_P \leq p \wedge q$ is the desired equation.
 \end{proof}

Let $X,Y \in shP, G \in Loc(shP)$, then we have the correspondence

\begin{equation} \label{axiomasparamodulos}
\begin{tabular}{c}
 $X \times Y \mr{\lambda} G$ an $\ell$-relation \\ \hline \noalign{\smallskip}
 $\Omega_P^X \otimes \Omega_P^Y \mr{\lambda} G$ a $s \ell$-morphism \\ \hline \noalign{\smallskip}
 $X_d \otimes_P Y_d \mr{\mu} G(1)$ a morphism of $P$-$Mod$
\end{tabular}
\end{equation}

The following propositions show how $\mu$ is computed from $\lambda$ and vice versa.

\begin{proposition} \label{muapartirdelambda}
 In the correspondence \eqref{axiomasparamodulos}, for each $p,q \in P$, $x \in X(p)$, $y \in Y(q)$, $$\mu(\delta_x \otimes \delta_y) = \Sigma_{p \wedge q}^1 \lambda_{p \wedge q} (x|_{p \wedge q},y|_{p \wedge q}).$$
\end{proposition}
\begin{proof}
 $\mu(\delta_x \otimes \delta_y) \stackrel{\ref{pyqsaltan}}{=} \lambda_1(\delta_{x|_{p \wedge q}} \otimes \delta_{y|_{p \wedge q}}) = \lambda_1 \Sigma_{p \wedge q}^1 (\{x|_{p \wedge q}\}_{p \wedge q} \otimes \{y|_{p \wedge q}\}_{p \wedge q} ) =$
 
 \vspace{1ex}
 
\hfill $  =   \Sigma_{p \wedge q}^1 \lambda_{p \wedge q} (\{x|_{p \wedge q}\}_{p \wedge q} \otimes \{y|_{p \wedge q}\}_{p \wedge q} ) = \Sigma_{p \wedge q}^1 \lambda_{p \wedge q} (x|_{p \wedge q},y|_{p \wedge q}).$
\end{proof}

\begin{corollary} \label{lambdaapartirdemu}
 Applying $\rho_{p \wedge q}^1$ and using that $\rho_{p \wedge q}^1 \Sigma_{p \wedge q}^1 = id$, we obtain the reciprocal computation
 $$\lambda_{p \wedge q} (x|_{p \wedge q},y|_{p \wedge q}) = \rho_{p \wedge q}^1 \mu(\delta_x \otimes \delta_y). $$
 
 \vspace{-4ex}
 
 $\hfill \qed$
\end{corollary}

\begin{remark} \label{notacionlinda}
In the correspondence \eqref{axiomasparamodulos} above, if $\lambda = \delta_X: X \times X \mr{} \Omega$, then $\mu(\delta_{x_1} \otimes \delta_{x_2}) = \igu[x_1]{x_2}_P$ (recall definition \ref{notacioncorchetes}).
\end{remark}

\begin{lemma} \label{multporr}
  In the correspondence \eqref{axiomasparamodulos}, for each $p,q,r \in P$, $x \in X(p)$, $y \in Y(q)$, 
   $$r \cdot \mu(\delta_x \otimes \delta_y) = \Sigma_{p \wedge q \wedge r}^1 \rho_{p \wedge q \wedge r}^{p \wedge q} \lambda_{p \wedge q} (x|_{p \wedge q},y|_{p \wedge q}) = \Sigma_{p \wedge q \wedge r}^1 \lambda_{p \wedge q \wedge r} (x|_{p \wedge q \wedge r},y|_{p \wedge q \wedge r}).$$
\end{lemma}

\begin{proof} The second equality is just the naturality of $\lambda$. To prove the first one, we compute:

\flushleft $r \cdot \mu(\delta_x \otimes \delta_y) \stackrel{\ref{muapartirdelambda}}{=} \Sigma_r^1 \rho_r^1 \Sigma_{p \wedge q}^1 \lambda_{p \wedge q} (x|_{p \wedge q},y|_{p \wedge q}) \stackrel{\ref{enumeratedeJT} \ item \ 2.b)}{=}$
 
\hfill $ = \Sigma_r^1 \Sigma_{p \wedge q \wedge r \ }^r \rho_{p \wedge q \wedge r \ }^{p \wedge q} \lambda_{p \wedge q} (x|_{p \wedge q},y|_{p \wedge q}) = \Sigma_{p \wedge q \wedge r \ }^1 \rho_{p \wedge q \wedge r \ }^{p \wedge q} \lambda_{p \wedge q} (x|_{p \wedge q},y|_{p \wedge q})$.
\end{proof}

The following proposition expresses the corresponding formulae for the four axioms of an $\ell$-relation $X \times Y \mr{\lambda} G$ in $shP$ (see definitions \ref{4axioms}, \ref{4axiomsparaG}), at the level of $P$-modules.

\begin{\prop} \label{propaxiomparamodulos}
Let $X,Y \in shP, G \in Loc(shP)$, and an $\ell$-relation $X \times Y \mr{\lambda} G$. Consider the corresponding $P$-module morphism $X_d \otimes_P Y_d \mr{\mu} G(1)$ as in \eqref{axiomasparamodulos}. Then $\lambda$ is $ed, uv, su, in$ resp. if and only if:
\begin{itemize}
 \item ed) for each $p \in P$, $x \in X(p)$, $\displaystyle \bigvee_{\stackrel{q \in P}{y \in Y(q)}} \mu(\delta_x \otimes \delta_y) = p$.
 \item uv) for each $p,q_1,q_2 \in P$, $x \in X(p)$, $y_1 \in Y(q_1)$, $y_2 \in Y(q_2)$, 
 
 \hfill $\mu(\delta_x \otimes \delta_{y_1}) \wedge \mu(\delta_x \otimes \delta_{y_2}) \leq \llbracket y_1 \! = \! y_2\rrbracket_P.$
 \item su) for each $q \in P$, $y \in Y(q)$, $\displaystyle \bigvee_{\stackrel{p \in P}{x \in X(p)}} \mu(\delta_x \otimes \delta_y) = q$.
 \item in) for each $p_1,p_2,q \in P$, $x_1 \in X(p_1)$, $x_2 \in X(p_2)$, $y \in Y(q)$, 

 \hfill $\mu(\delta_{x_1} \otimes \delta_y) \wedge \mu(\delta_{x_2} \otimes \delta_y) \leq \llbracket x_1 \! = \! x_2\rrbracket_P.$ 
\end{itemize}
\end{\prop}

\begin{proof}

By proposition \ref{edyuvdafunctionG} and remark \ref{reldomegaYremarkparaG}, $\lambda$ is $ed)$ if and only if $\displaystyle \bigvee_{y \in Y} \lambda^*(y) = 1$ in $G^X$.
By proposition \ref{calcsup} and remark \ref{1deGalaX}, this is an equality of global sections $\displaystyle \bigvee_{\stackrel{q \in P}{y \in Y(q)}} \Sigma_q^1 \lambda^*_q (y) = 1$ in $G^X(1) \stackrel[\textcolor{white}{M}]{}{=} [X,G]$. Then $\lambda$ is $ed)$ if and only if for each $p \in P$, $x \in X(p)$, \linebreak $\displaystyle \bigvee_{\stackrel{q \in P}{y \in Y(q)}} (\Sigma_q^1 \lambda^*_q (y))_p (x) = p$ in $G(p)$. But by $F\Sigma)$ in lemma \ref{exponencialenshP} we have \linebreak $(\Sigma_q^1 \lambda^*_q (y))_p (x) \stackrel[\textcolor{white}{M}]{}{=}  \Sigma_{p \wedge q}^p (\lambda^*_q (y))_{p \wedge q} (x|_{p \wedge q}) = \Sigma_{p \wedge q}^p \lambda_{p \wedge q} (x|_{p \wedge q}, y|_{p \wedge q}),$ where last equality holds since by definition of $\lambda^*$ we have $(\lambda^*_q (y))^{\textcolor{white}{M}}_{p \wedge q} (x|_{p \wedge q}) = \lambda_{p \wedge q} (x|_{p \wedge q}, y|_{p \wedge q})$.

\vspace{.2cm}

We conclude that $\lambda$ is $ed)$ if and only if for each $p \in P$, $x \in X(p)$, \linebreak $\displaystyle \bigvee_{\stackrel{q \in P}{y \in Y(q)}} \Sigma_{p \wedge q}^p \lambda_{p \wedge q} (x|_{p \wedge q}, y|_{p \wedge q}) = p$ in $G(p)$. Since $\rho_p^1 \Sigma_p^1 = id$, this holds if and only if it holds after we apply $\Sigma_p^1$. Then, proposition \ref{muapartirdelambda} yields the desired equivalence.

\vspace{.3cm}

\noindent We now consider axiom uv):

$\lambda$ is $uv)$ if and only if for each $p,q_1,q_2 \in P$, $x \in X(p)$, $y_1 \in Y(q_1)$, $y_2 \in Y(q_2)$,

\vspace{.2cm}

\noindent $ \displaystyle \rho_{p \wedge q_1 \wedge q_2}^{p \wedge q_1} \lambda_{p \wedge q_1} (x|_{p \wedge q_1}, y_1|_{p \wedge q_1}) \wedge \rho_{p \wedge q_1 \wedge q_2}^{p \wedge q_2} \lambda_{p \wedge q_2} (x|_{p \wedge q_2}, y_2|_{p \wedge q_2}) \leq $

\hfill $\rho_{p \wedge q_1 \wedge q_2}^{q_1 \wedge q_2} {\delta_Y}_{q_1 \wedge q_2} (y_1|_{q_1 \wedge q_2},y_2|_{q_1 \wedge q_2} )$.

\vspace{.2cm}

We apply $\Sigma_{p \wedge q_1 \wedge q_2}^1$ and use that it commutes with $\wedge$ to obtain that this happens if and only if 

\vspace{.2cm}

\noindent $ \displaystyle \Sigma_{p \wedge q_1 \wedge q_2}^1 \rho_{p \wedge q_1 \wedge q_2}^{p \wedge q_1} \lambda_{p \wedge q_1} (x|_{p \wedge q_1}, y_1|_{p \wedge q_1}) \wedge \Sigma_{p \wedge q_1 \wedge q_2}^1 \rho_{p \wedge q_1 \wedge q_2}^{p \wedge q_2} \lambda_{p \wedge q_2} (x|_{p \wedge q_2}, y_2|_{p \wedge q_2}) $

\hfill $ \leq \Sigma_{p \wedge q_1 \wedge q_2}^1 \rho_{p \wedge q_1 \wedge q_2}^{q_1 \wedge q_2} {\delta_Y}_{q_1 \wedge q_2} (y_1|_{q_1 \wedge q_2},y_2|_{q_1 \wedge q_2} )$,

\vspace{.2cm}

\noindent which by lemma \ref{multporr} (see remark \ref{notacionlinda}) is equation 

$$ \displaystyle q_2 \cdot \mu(\delta_x \otimes \delta_{y_1} ) \wedge q_1 \cdot \mu(\delta_x \otimes \delta_{y_2} ) \leq p \cdot \igu[y_1]{y_2}_P,$$
but since $q_i \cdot \delta_{y_i} = \delta_{y_i}$ ($i=1,2$), by corollary \ref{pyqsaltan} this is equivalent to the equation 
$\mu(\delta_{x} \otimes \delta_{y_1} ) \wedge \mu(\delta_{x} \otimes \delta_{y_2} ) \leq p \cdot \llbracket y_1 \! = \! y_2 \rrbracket_P$.

This equation is equivalent to the one in the proposition since the right term is lower or equal than $\llbracket y_1 \! = \! y_2 \rrbracket_P$, and multiplying by $p$ the left term doesn't affect it.

\end{proof}

\begin{\de} \label{definicionesparamu}
Let $X,Y \in shP, G \in Loc(shP)$, and an $\ell$-relation $X \times Y \mr{\lambda} G$. Consider the corresponding $P$-module morphism $X_d \otimes_P Y_d \mr{\mu} G(1)$. We say that $\mu$ is $ed, uv, su, in$ resp. if it satisfies the conditions of proposition \ref{propaxiomparamodulos} above. We say that $\mu$ is an \linebreak $\ell$-function, resp. $\ell$-op-function, resp. $\ell$-bijection if it is $ed$ and $uv$, resp. $su$ and $in$ resp. the four conditions.

Note that $\mu$ has each of the properties defined above if and only if $\lambda$ does.
\end{\de}

Consider now the autoduality of $\Omega_P^X$ in $s \ell(shP)$ given by proposition $\ref{autodual}$. Applying the tensor equivalence $s \ell(shP) \mr{\gamma_*} P$-$Mod$ it follows that $X_d$ is autodual in $P$-$Mod$, in the sense of definition \ref{dualcomobimod}. We will now give the formulae for the $\eta$, $\eps$ of this duality.

\begin{\prop} \label{formulaeetaeps}
 The $P$-module morphisms $P \mr{\eta} X_d \stackrel[P]{}{\otimes} X_d$, $X_d \stackrel[P]{}{\otimes} X_d \mr{\eps} P$ are given by the formulae $\eta(1) = \displaystyle \bigvee_{\stackrel{p \in P}{x \in X(p)}} \delta_x \otimes \delta_x$, $\eps(\delta_x \otimes \delta_y) = \llbracket x \! = \! y \rrbracket_P$ for each $p,q \in P$, $x \in X(p)$, $y \in X(q)$.
\end{\prop}

\begin{proof}
 The internal formula for $\eta$ given in the proof of proposition $\ref{autodual}$, together with remark \ref{calcsup} yield the formula for $\eta$. The internal formula for $\eps$, together with our definition of the notation $\llbracket x \! = \! y \rrbracket_P$ yield that for each $p,q \in P$, $x \in X(p)$, $y \in X(q)$, we have $\eps_{p \wedge q}(\{x|_{p \wedge q}\}_{p \wedge q} \otimes \{y|_{p \wedge q}\}_{p \wedge q}) = \llbracket x \! = \! y \rrbracket_P$ in $\Omega_P(p \wedge q)$. Apply $\Sigma_{p \wedge q}^1$, use that it \linebreak commutes with the $s \ell$-morphism $\eps$ and recall remark \ref{1deGalaX} to obtain \linebreak
 $\eps_1( \delta_{x|_{p \wedge q}} \otimes \delta_{y|_{p \wedge q}}) = \llbracket x \! = \! y \rrbracket_P$ in $P$, which by corollary \ref{pyqsaltan} is the desired equation.
 \end{proof}

\subsection{A particular type of $\ell$-relation} \label{particular}

Assume $P$ is the coproduct of two locales, $P = H \otimes L$. Then the inclusions into the coproduct yield projections from the product of topoi $shH \stackrel{\pi_1}{\longleftarrow} sh(H \otimes L) \stackrel{\pi_2}{\longrightarrow} shL$. 

Consider now $X \in shH, Y \in shL$, $G \in Loc(sh(H \otimes L))$. We can consider an \linebreak $\ell$-relation $\pi_1^*X \times \pi_2^*Y \mr{\lambda} G$, and the corresponding $(H \otimes L)$-module morphism \linebreak $(\pi_1^*X)_d \stackrel[H \otimes L]{}{\otimes} (\pi_2^*Y)_d \mr{\mu} G(1)$.

To compute $(\pi_1^*X)_d$, note that $X_d$ is the $H$-module corresponding to the locale of open parts of the discrete space $X_{dis}$ (recall corollary \ref{discretespace}). By \cite{JT}, VI.3 Proposition 3, p.51, $H \mr{} X_d$ is the morphism of locales corresponding to the etale (over $\overline{H}$) space $X_{dis} = \Omega_H^X$. 
Then we have the following pull-back of spaces (push-out of locales)

$$\xymatrix{(\pi_1^*X)_{dis} \ar[r] \ar[d] & X_{dis} \ar[d] \\
	      \overline{H \otimes L} \ar[r] & \overline{H}} \quad \quad \quad \quad
 \xymatrix{(\pi_1^*X)_d & X_d \ar[l]  \\
	      H \otimes L \ar[u] & H \ar[l] \ar[u] }$$       

\noindent which shows that $(\pi_1^*X)_d = X_d \otimes L$, and similarly $(\pi_2^*Y)_d = H \otimes Y_d$. Then we have 
$(\pi_1^*X)_d \stackrel[H \otimes L]{}{\otimes} (\pi_2^*Y)_d = (X_d \otimes L) \stackrel[H \otimes L]{}{\otimes} (H \otimes Y_d) \cong X_d \otimes Y_d$,
where the last tensor product is the tensor product of sup-lattices in $\Sat$, i.e. as $\Omega$-modules. The isomorphism maps $\delta_x \otimes \delta_y \mapsto (\delta_x \otimes 1) \otimes (1 \otimes \delta_y)$, then we have the following instance of proposition \ref{propaxiomparamodulos}. 

\begin{\prop} \label{propaxiomparamodulos2}
Let $X \in shH$, $Y \in shL, G \in Loc(sh(H \otimes L))$, and an $\ell$-relation \linebreak $\pi_1^*X \times \pi_2^*Y \mr{\lambda} G$. Consider the corresponding $(H \otimes L)$-module morphism \linebreak $X_d \otimes Y_d \mr{\mu} G(1)$. Then $\lambda$ is $ed, uv, su, in$ resp. if and only if: 
\begin{itemize}
 \item ed) for each $h \in H$, $x \in X(h)$, $\displaystyle \bigvee_{\stackrel{l \in L}{y \in Y(l)}} \mu(\delta_x \otimes \delta_y) = h$.
 \item uv) for each $h \in H$, $l_1,l_2 \in L$, $x \in X(h)$, $y_1 \in Y(l_1)$, $y_2 \in Y(l_2)$, 
 
 \hfill $\mu(\delta_x \otimes \delta_{y_1}) \wedge \mu(\delta_x \otimes \delta_{y_2}) \leq \llbracket y_1 \! = \! y_2\rrbracket_P.$
 \item su) for each $l \in L$, $y \in Y(l)$, $\displaystyle \bigvee_{\stackrel{h \in H}{x \in X(h)}} \mu(\delta_x \otimes \delta_y) = l$.
 \item in) for each $h_1,h_2 \in H$, $l \in L$, $x_1 \in X(h_1)$, $x_2 \in X(h_2)$, $y \in Y(l)$, 
 
 \hfill $\mu(\delta_{x_1} \otimes \delta_y) \wedge \mu(\delta_{x_2} \otimes \delta_y) \leq \llbracket x_1 \! = \! x_2\rrbracket_P.$
\end{itemize}
\flushright \cqd
\end{\prop}

\begin{sinnadastandard} \label{conosconlambdaomu}
 Consider the situation of \ref{diagramadiamante12} in the topos $shP$, with $P = H \otimes L$. Since the $\lozenge_2(f,g)$ diagram is a diagram in $Rel(shP) \subseteq s \ell(shP) \cong P$-$Mod$, it is equivalent to a corresponding diagram in $P$-$Mod$ that we will also denote $\lozenge_2$ (recall that $g^{op}$ corresponds to $g^{\wedge}$, see \ref{dualintercambia}),
 $$\lozenge_2(f,g): \vcenter
       {
        \xymatrix@C=1.4ex@R=3ex
                 {
                  & X_d \otimes Y_d  \ar[rd]^{\mu} 
                  \\
			        X_d' \otimes Y_d \ar[ru]^{f^{\wedge} \otimes Y_d} 
			                    \ar[rd]_{X_d' \otimes g} & \equiv & G(1)\, ,
			      \\
			       & X_d' \otimes Y_d' \ar[ru]_{\mu'} 
			      } 
	    }
$$
expressing the equation: for each $h \in H$, $l \in L$, $x' \in X'(h)$, $y \in Y(l)$
\begin{equation} \label{equationdeconoconmu}
\lozenge_2(f,g): \mu'(\delta_{x'} \otimes \delta_{g(l)}) = \bigvee_{\stackrel{b \in H}{x \in X(b)}} \igu[f(x)]{x'}_H \cdot \mu(\delta_{x} \otimes \delta_l) 
\end{equation}

\end{sinnadastandard}

\pagebreak

\section{The equivalence $Cmd_0(\O(G)) = Rel(\beta^G)$} \label{sec:Cmd=Rel}

\begin{sinnadastandard} \label{sinnadadeequivalence0}
We fix throughout this section a localic groupoid $G$ (i.e. groupoid object in \linebreak $S \! p=Loc^{op}$), with subjacent structure of localic category (i.e. category object in $S \! p$) given by (see \cite{JT}, VIII.3 p.68) $$\xymatrix{G \stackrel[G_0]{\textcolor{white}{G_0}}{\times} G \ar[r]^>>>>>>{\circ} & G \ar@<1.3ex>[r]^{\partial_0} \ar@<-1.3ex>[r]_{\partial_1} & G_0 \ar[l]|{\ \! i \! \ }}$$ 

\vspace{-.75cm}

$$\hbox{(we abuse notation by using the same letter } G \hbox{ for the object of arrows of } G).$$

We denote by $L = \O(G)$, $B = \O(G_0)$ their corresponding locales of open parts, and think of them as (commutative) algebras in the monoidal category $s \ell$. The locale morphisms $\xymatrix{B \ar@<1ex>[r]^{s = \partial_0^{-1}} \ar@<-1ex>[r]_{t = \partial_1^{-1}} & L}$ give $L$ a structure of $B$-bimodule. We establish, following \cite{De}, that $B$ acts on the left via $t$ and on the right via $s$. This is consistent with the pull-back $G \times_{G_0} G$ above which is thought of as the pairs $\{(f,g) \in G \times G | \partial_0(f)=\partial_1(g)\}$ of composable arrows, in the sense that $\O(G \times_{G_0} G) = L \otimes_B L$ (the push-out corresponding to the pull-back above is the tensor product of $B$-bimodules).

In this way, the unit $\xymatrix{G_0 \ar[r]^i & G}$ corresponds to a counit $L \mr{e} B$, and the multiplication (composition) $G \times_{G_0} G \mr{\circ} G$ corresponds to a comultiplication $L \mr{c} L \otimes_B L$. Therefore $L$ is a coalgebra in the category $B$-bimod, i.e. a \emph{cog\`ebro\"ide agissant sur B}. In other words, a localic category structure for $G$ is the same as a cog\`ebro\"ide structure for $L$.

We define a \emph{localic Hopf algebroid} as the exact formal dual structure of a localic groupoid. The inverse $G \mr{(-)^{-1}} G$ of a localic groupoid corresponds to an \emph{antipode} $L \mr{a} L$. As was observed by Deligne in \cite{De}, p.117, the structure of cog\`ebro\"ide is the subjacent structure of a Hopf algebroid which is used to define its representations (see definition \ref{defcmd}), exactly like the subjacent localic category structure of the groupoid is the subjacent structure required to define $G$-spaces as $S \! p$-valued functors, namely, actions of the category object on a family (internal) $X \mr{} G_0$ (see definition \ref{defdeaction}).
\end{sinnadastandard}

\subsection{The category $\beta^G$}

Groupoid objects $G$ in $S \! p$ act on spaces over $G_0$, $X \mr{} G_0$, as groupoids (or categories with object of objects $G_0$) act on families over $G_0$ in $S \! ets$, defining an internal functor. We consider $G \times_{G_0} X$, the pull-back of spaces over $G_0$ constructed using $\partial_0$, as a space over $G_0$ using $\partial_1$:

\begin{\de} \label{defdeaction}
An \emph{action} of a localic groupoid $G$ in a space over $G_0$, $X \mr{} G_0$, is a morphism $G \times_{G_0} X \mr{\theta} X$ of spaces over $G_0$ such that the following diagrams commute.
$$A1: \vcenter{\xymatrix{G \stackrel[G_0]{\textcolor{white}{G_0}}{\times} G \stackrel[G_0]{\textcolor{white}{G_0}}{\times} X \ar[r]^>>>>>{\circ \times X} \ar[d]_{G \times \theta} & G \stackrel[G_0]{\textcolor{white}{G_0}}{\times} X \ar[d]^{\theta} \\
													G \stackrel[G_0]{\textcolor{white}{G_0}}{\times} X \ar[r]^{\theta} & X  } }
					A2: \vcenter{ \xymatrix{G \stackrel[G_0]{\textcolor{white}{G_0}}{\times} X \ar[r]^{\theta} & X \\
												G_0 \stackrel[G_0]{\textcolor{white}{G_0}}{\times} X \ar[u]^{i \times X} \ar[ur]_{\cong} }} $$

Given two actions that we will denote by $G \curvearrowright X$, $G \curvearrowright X'$, an action morphism (which corresponds to a natural transformation between the functors) is a morphism $f$ of spaces over $G_0$ such that the following diagram commute.
$$AM: \; \vcenter{ \xymatrix{G \stackrel[G_0]{\textcolor{white}{G_0}}{\times} X \ar[r]^{\theta} \ar[d]_{G \times f} & X \ar[d]^f \\
						G \stackrel[G_0]{\textcolor{white}{G_0}}{\times} X' \ar[r]^{\theta'} & X' }}$$
\end{\de}

\begin{remark}
The reader can easily check that these definitions are equivalent to the ones of \cite{JT}, VIII.3, p.68. 
\end{remark}

\begin{remark} Recall from \cite{JT}, VI.3 p.51, Proposition 3 (see also proposition \ref{discretespace} and \ref{enumeratedeJT}, item 5), that the functor $$sh B \mr{(-)_{dis}} S \! p(shB) \mr{\gamma_*} B \hbox{-} Loc^{op}$$ \vspace{-3ex} $$\xymatrix@C=3pc{\quad Y \quad \ar@{|->}[rr] && (\overline{Y_d} \rightarrow \overline{B}),}$$ where $Y_d = \gamma_*(\Omega^Y) = \gamma_*\cc{O}(Y_{dis})$ (recall definition \ref{defdeXd} and proposition \ref{discretespace}), yields an equivalence of categories $sh B \mr{} Et_B$, where $Et_B$ is the category of etale spaces over $\overline{B}$, i.e. $X \mr{p} \overline{B}$ satisfying that $p$ and the diagonal $X \mr{\triangle} X \times_{\overline{B}} X$ are open (see \cite{JT}, V.5 p.41). 
\end{remark}
 
\begin{\de}
An action $G \curvearrowright X$ is \emph{discrete} if $X \mr{} G_0$ is etale, i.e. in view of last remark if $X = \overline{Y_{d}}$ (or equivalently $\O(X) = Y_d$) with $Y \in shB$. We denote by $\beta^G$ the category of discrete actions of $G$.
\end{\de}

\begin{sinnadastandard} \label{rel}
Consider $s \ell_0(shB)$ the full subcategory of $s \ell(shB)$ with objects of the form $\Omega_B^Y$. Then we have the equivalence $: Rel(shB) \mr{(-)_*} s \ell_0(shB)$. consider also the restriction of the equivalence $s \ell(shB) \cong B$-$Mod$ to $s \ell_0(shB) = (B$-$Mod)_0$. Combining both we obtain $Rel(shB) = (B$-$Mod)_0$, mapping $Y \leftrightarrow Y_d$.
\end{sinnadastandard}

The objective of this section is to prove the following theorem:

\begin{theorem} \label{Comd=Relprevio}
For any localic groupoid $G$, there is an equivalence of categories making the triangle commutative ($T,F$ are forgetful functors):
$$
\xymatrix@C=0ex
        {
          Cmd_0(L) \ar[rr]^{=} \ar[rd]_T 
      & & \cc{R}el(\beta^G) \ar[ld]^(.4){\cc{R}el(F)} 
        \\
         & (B\hbox{-}Mod)_0 \cong \cc{R}el(shB).
        }       
$$ 
\end{theorem}

\subsection{The equivalence at the level of objects} \label{equiv:objects} 

Consider an etale space $X \mr{} G_0$, and assume $\O(X) = Y_d$, with $Y \in shB$. A (discrete) action $G \curvearrowright X$ ($G \times_{G_0} X \mr{\theta} X$ satisfying A1, A2) corresponds exactly to a $B$-locale morphism $Y_d \mr{\rho} L \otimes_B Y_d$ satisfying C1, C2 (in definition \ref{defcmd}). Therefore, to establish an equivalence between discrete actions $G \curvearrowright X$ and comodules $Y_d \mr{\rho} L \otimes_B Y_d$ we need to prove 

\begin{\prop} \label{locmorphgratis}
Every comodule structure $Y_d \mr{\rho} L \otimes_B Y_d$ is automatically a locale morphism (when $L$ is the cog\`ebro\"ide subjacent to a localic groupoid).   
\end{\prop}

Next we prove this proposition (see \ref{esquema} below for a clarifying diagram). In order to do this, we will work in the category of $B \otimes B$-modules. Since $B$ is commutative, we have an isomorphism of categories $B$-bimod $\cong B \otimes B$-mod, but 
we consider the tensor product $\stackrel[B \otimes B]{}{\otimes}$ of $B \otimes B$-modules as in \ref{tensorprodofmodules}, not to be confused with the tensor product $\otimes_B$ as $B$-bimodules. Via this isomorphism, $L$ is a $B \otimes B$-module whose structure is given by $B \otimes B \mr{(t,s)} L$.

We first notice that $L \stackrel[B]{}{\otimes} Y_d \cong L \stackrel[B \otimes B]{}{\otimes} (B \otimes Y_d)$, and via extension of scalars (using the inclusion $B \mr{} B \otimes B$ in the first copy), $\rho$ corresponds to a $B \otimes B$-module morphism $Y_d \otimes B \mr{\rho} L \stackrel[B \otimes B]{}{\otimes} (B \otimes Y_d)$. 
From the equivalence of tensor categories recalled in section \ref{enumeratedeJT} items 5,6, with $P = B \otimes B$, $\rho$ corresponds to a morphism $\varphi$ in $s \ell(sh(B \otimes B))$, $\rho = \gamma_*(\varphi)$, and $\rho$ is a locale morphism if and only if $\varphi$ is so.

From the results of section \ref{particular}, $Y_d \otimes B = {(\pi_1^* Y)}_d = \gamma_*(\Omega_{B \otimes B}^{\pi_1^*Y})$, and similarly \linebreak $B \otimes Y_d = \gamma_*(\Omega_{B \otimes B}^{\pi_1^*Y})$, where $\Omega_{B \otimes B}$ is the subobject classifier of $sh(B \otimes B)$. Then  
$$L \stackrel[B \otimes B]{}{\otimes} (B \otimes Y_d) \stackrel{\ref{enumeratedeJT}}{=} \gamma_*(\widetilde{L} \stackrel{(1)}{\otimes} \Omega_{B \otimes B}^{\pi_2^*Y}) \stackrel{(2)}{=} \gamma_* (\widetilde{L}^{\pi_2^*Y}),$$ 

where $\widetilde{L}$ is as in \ref{enumeratedeJT} item 5, $\gamma_* \widetilde{L} = L$, the tensor product marked with $(1)$ is as sup-lattices in $sh(B \otimes B)$ and the equality marked with $(2)$ holds since $\widetilde{L} \otimes \Omega_{B \otimes B}^{\pi_2^*Y}$ and $\widetilde{L}^{\pi_2^*Y}$ are the free $\widetilde{L}$-module in $\pi_2^*Y$ (see proposition \ref{formulainternaparaG}).

Then $\varphi$ is $\Omega_{B \otimes B}^{\pi_1^*Y} \mr{\varphi} \widetilde{L}^{\pi_2^*Y}$, therefore by remark \ref{edyuvdafunctionmedioG} there is an $\ell$-relation \linebreak  $\pi_1^*Y \times \pi_2^*Y \mr{\lambda} \widetilde{L}$ in the topos $sh(B \otimes B)$ such that $\varphi = \lambda_*$ and, to see that $\rho$ is a locale morphism, we can prove that $\lambda$ is an $\ell$-op-function.

\begin{sinnadastandard} \label{esquema}
 We schematize the previous arguing in the following correspondence
 
 \vspace{.2cm}
 
 \noindent \begin{tabular}{ccc}
    $Y_d \mr{\rho} L \stackrel[B]{}{\otimes} Y_d$ & $B$-module morphism & $B$-locale morphism \\ \noalign{\smallskip} \hline \noalign{\smallskip}
    $Y_d \otimes B \mr{\rho} L \stackrel[B \otimes B]{}{\otimes} (B \otimes Y_d)$ & $(B \otimes B)$-module morphism & $(B \otimes B)${-locale morphism} \\  \noalign{\smallskip}\hline \noalign{\smallskip}
    $\Omega_{B \otimes B}^{\pi_1^*Y} \mr{\varphi} \widetilde{L}^{\pi_2^*Y}$ & $s \ell${ morphism in} $sh(B \otimes B)$ & locale morphism \\ \noalign{\smallskip}\hline \noalign{\smallskip}
    $\pi_1^*Y \times \pi_2^*Y \mr{\lambda} \widetilde{L}$ & $\ell${-relation in} $sh(B \otimes B)$ & $\ell$-op-function 
   \end{tabular}

\end{sinnadastandard}

\begin{\prop} \label{bijectiongratis}
The $\ell$-relation $\pi_1^*Y \times \pi_2^*Y \mr{\lambda} \widetilde{L}$ corresponding to a comodule structure $Y_d \mr{\rho} L \otimes_B Y_d$, where $L$ is the cog\`ebro\"ide subjacent to a localic groupoid, is an $\ell$-bijection.
\end{\prop}
\begin{proof} We will use the analysis of this particular kind of $\ell$-relations that we did in section \ref{particular}. We have seen that $\lambda$ corresponds to a $B$-bimodule morphism $Y_d \otimes Y_d \mr{\mu} L$. We have also seen, in proposition \ref{propaxiomparamodulos2}, which conditions in $\mu$ are equivalent to the axioms for $\lambda$. 

Since any duality induces an internal-hom adjunction and $\Omega^Y$ is autodual, $\mu$ corresponds to $\rho$ via the duality of modules described in $\ref{adjunciondeldual}$. Then by lemma \ref{lemadeunicidad}, the $B1$ and $B2$ subdiagrams in the following diagram are commutative. Also, the pentagon subdiagram $\pentagon$ is commutative by definition of the localic groupoid $G$, where $a$ is the antipode corresponding to the inverse of $G$ (cf. proof of proposition \ref{monoidaction}).          

\begin{equation} \label{diagramapentagonal}
\xymatrix@C=3pc{
          \ar@{}[dr]|>>>>>>>{B2} & Y_d \otimes Y_d \ar[r]^(.35){Y_d \otimes \eta \otimes Y_d} \ar[d]^\mu \ar[dl]_\eps  & Y_d \otimes Y_d  \stackrel[B]{}{\otimes} Y_d \otimes Y_d \ar@{}[dl]|{B1} \ar[d]^{\mu \otimes_B \mu} \\
          B \ar@<.75ex>[dr]^t \ar@<-.75ex>[dr]_s  & L \ar@{}[dr]|{\pentagon} \ar[r]^c \ar[l]_e  & L \stackrel[B]{}{\otimes} L \ar@<-1ex>[d]_{a \otimes L}
                        \ar@<1ex>[d]^{L \otimes a} \\
          & L & L \stackrel[B \otimes B]{}{\otimes} L. \ar[l]_\wedge  }
\end{equation}

To prove axiom $ed)$, let $b_0 \in B$, $x \in Y(b_0)$. 
Chasing $\delta_x \otimes \delta_x$ in diagram \eqref{diagramapentagonal} all the way down to $L$ using the arrow $L \otimes a$ we obtain (recall our formulae for $\eta$, $\eps$ in proposition \ref{formulaeetaeps})
$\displaystyle \bigvee_{\stackrel{b \in B}{y \in Y(b)}} \mu  ( \delta_x \otimes \delta_y )  \wedge a \mu  ( \delta_y \otimes \delta_x )  = b_0$, which implies the inequality $\displaystyle \bigvee_{\stackrel{b \in B}{y \in Y(b)}} \mu  ( \delta_x \otimes \delta_y ) \geq b_0$, i.e. $\geq$ in $ed)$ in proposition \ref{propaxiomparamodulos2}, but the inequality $\leq$ always holds.

To prove axiom $uv)$, let $b_0, b_1, b_2 \in B$, $x \in Y(b_0)$, $y_1 \in Y(b_1)$, $y_2 \in Y(b_2)$. Chasing $\delta_{y_1} \otimes \delta_{y_2}$, but this time using the arrow $a \otimes L$, we obtain $$\displaystyle \bigvee_{\stackrel{c \in B}{w \in Y(c)}} a  \mu  ( \delta_{y_1} \otimes \delta_w )  \wedge \mu  ( \delta_w \otimes \delta_{y_2} )  = \llbracket y_1 \! = \! y_2 \rrbracket_{B} ,$$
then in particular $ \quad (1) \quad  a  \mu  ( \delta_{y_1} \otimes \delta_x)  \wedge \mu  ( \delta_x \otimes \delta_{y_2} )  \leq \llbracket y_1 \! = \! y_2 \rrbracket_{B} . $

\vspace{1ex}

To deduce $uv)$ from (1) we need to see that $a  \mu  ( \delta_{y_1} \otimes \delta_x) = \mu  ( \delta_x \otimes \delta_{y_1})$. Since $a^2 = id$, it is enough to prove $\leq$:

\vspace{1ex}
 
\noindent $a  \mu  ( \delta_{y_1} \otimes \delta_x) \stackrel{\ref{restringirlosdelta}}{=}  a  \mu  ( \delta_{y_1} \otimes b_0 \cdot \delta_x) = a  \mu  ( \delta_{y_1} \otimes \delta_x) \wedge b_0  \stackrel{ed)}{=} a  \mu  ( \delta_{y_1} \otimes \delta_x) \wedge \displaystyle\bigvee_{\stackrel{b \in B}{y \in Y(b)}}  \mu  ( \delta_x \otimes \delta_y )$

\noindent $ =  \displaystyle\bigvee_{\stackrel{b \in B}{y \in Y(b)}}  a  \mu  ( \delta_{y_1} \otimes \delta_x) \wedge \mu  ( \delta_x \otimes \delta_y ) \stackrel{(1)}{=} 
 \displaystyle\bigvee_{\stackrel{b \in B}{y \in Y(b)}}  a  \mu  ( \delta_{y_1} \otimes \delta_x) \wedge \mu  ( \delta_x \otimes \delta_y ) \wedge \llbracket y_1 \! = \! y \rrbracket_{B}
 \stackrel{\ref{ecuacionenOmegaX}}{=} $

\hfill $ =  a  \mu  ( \delta_{y_1} \otimes \delta_x) \wedge \mu  ( \delta_x \otimes \delta_{y_1} ) $. 

\vspace{1ex}
 
Axioms $su)$ and $in)$ follow symetrically.
\end{proof}

We have finished the proof of proposition \ref{locmorphgratis}. For future reference, we record the results of this section:

\begin{\prop} \label{equivdeaccion}
Given a localic groupoid $G$ over $G_0$, with subjacent cog\`ebro\"ide $L$ sur $B$, and $Y \in shB$
, the following structures are in a bijective correspondence:
\begin{itemize}
 \item Discrete actions $G \stackrel[G_0]{}{\times} \overline{Y_{d}} \mr{\theta} \overline{Y_{d}}$.
 \item $\ell$-relations $\pi_1^*Y \times \pi_2^*Y \mr{\lambda} \widetilde{L}$ with a corresponding $B$-bimodule morphism \linebreak $Y_d \otimes Y_d \mr{\mu} L$ such that the following diagrams commute:
 $$B1: \vcenter{\xymatrix{Y_d \otimes Y_d \ar[r]^{\mu} \ar[d]_{Y_d \otimes \eta \otimes Y_d} & L \ar[d]^{c} \\
	    Y_d \otimes Y_d \stackrel[B]{}{\otimes} Y_d \otimes Y_d \ar[r]^>>>>{\mu \otimes_B \mu} & L \stackrel[B]{}{\otimes} L}}
B2: \vcenter{\xymatrix{Y_d \otimes Y_d \ar[r]^{\mu} \ar[dr]_{\eps} & L \ar[d]^{e} \\ & B}}$$
  \item Comodule structures $Y_d \mr{\rho} L \otimes_B Y_d$. \qed
\end{itemize}
\end{\prop}

\begin{remark} \label{coinciden1}
 By proposition \ref{monoidaction}, in the case where $G$ is a localic group, actions $Aut(X) \mr{} G$ defined as in \cite{D1} (see \ref{thecategorybetaG}) also correspond to the previous structures.
\end{remark}

\begin{notation}
 We fix until the end of this thesis the following notation: we use the symbols $\theta$, $\rho$, $\lambda$, $\mu$ only for the arrows in the correspondence above, adding a $(-)'$ if neccessary.
\end{notation}

\begin{remark}
 In \cite{De}, the comodule structure considered is the opposite of \ref{defcmd}, i.e. right $L$-comodules $Y_d \mr{\rho} Y_d \otimes_B L$. By considering the inverse image $\lambda^*$ we obtain that this structure is also equivalent to the other three.
\end{remark}

\subsection{The equivalence at the level of arrows} \label{equiv:arrows}

We start this section with some results that allow us to better understand the category $Rel(\beta^G)$. We begin with a proposition that relates action morphisms with $\lozenge_2$-cones as in section \ref{sec:cones}. 

\begin{\prop} \label{accionsiidiam2}
Given two discrete actions $G \times_{G_0} \overline{Y_{d}} \mr{\theta} \overline{Y_{d}}$, $G \times_{G_0} \overline{Y'_{d}} \mr{\theta'} \overline{Y'_{d}}$, a space morphism $\overline{Y_{d}} \mr{f} \overline{Y'_{d}}$ is an action morphism if and only if the corresponding arrow $Y \mr{g} Y'$ in $shB$ satisfies 

$$\vcenter{\xymatrix{\lozenge_2(g): \ \ }}
\vcenter{\xymatrix@C=-1.5ex
         {   Y' \dcellb{g^{op}} & & Y \did 
          \\              
             Y & & Y
          \\
            & \ G \ \cl{\lambda} 
         }}
\vcenter{\xymatrix{ \;=\; }}
\vcenter{\xymatrix@C=-1.5ex{   Y' \did & & Y \dcell{g}
          \\              
             Y' & & Y'
          \\
           & \ G \ \cl{\lambda'}         }}
\vcenter{\xymatrix{\hbox{ i.e. by }\ref{conosconlambdaomu}: \quad    }}           
\vcenter{\xymatrix@C=-1.5ex
         {  Y'_d \dcellb{g^{\wedge}} & &  Y_d \did
          \\              
             Y_d & & Y_d
          \\
            & \ L \ \cl{\mu} 
         }}
\vcenter{\xymatrix{ \;=\; }}
\vcenter{\xymatrix@C=-1.5ex{  Y'_d \did & & Y_d \dcell{g} 
          \\              
             Y'_d & & Y'_d
          \\
           & \ L \ \cl{\mu'}         }}$$

\end{\prop}

\begin{proof}
$f^{-1}$ (the formal dual of $f$) is the $B$-locale morphism $Y_d \mr{g^{\wedge}} Y_d$, which is computed with the autoduality of $Y_d$ (see \ref{discretespace} and \ref{autodual}), and the correspondence between $\theta$ and $\mu$ in proposition \ref{equivdeaccion} is also given by this duality, i.e.

$$\vcenter{\xymatrix{f^{-1} = g^{\wedge}: \quad   }}
\vcenter{\xymatrix@C=-0.3pc@R=1pc{  Y'_d \did &&& \op{\eta} \\
					Y'_d \did && Y_d \dcell{g} \Brr & \,\,\, & Y_d \did \\
					Y'_d && Y'_d \Brr && Y_d \did \\
					& B \cl{\eps} & \Brr && Y_d }}
\vcenter{\xymatrix{\quad \quad \theta^{-1} = \rho: \quad }}
\vcenter{\xymatrix@C=-0.3pc@R=1pc{Y'_d \did &&& \op{\eta} \\
				  Y'_d && Y'_d \Brr & \,\,\, & Y'_d \did \\
				  & L \cl{\mu} & \Brr && Y'_d }}$$

Then the commutativity of the diagram AM in definition \ref{defdeaction}, expressing that $f$ is an action morphism, is equivalent when passing to the formal dual to the equality of the left and right terms of the equation (and therefore to the equality marked with an (*))
$$\vcenter{\xymatrix@C=-0.3pc@R=1pc{Y'_d \did &&& \op{\eta} \\
	 Y'_d && Y'_d \Brr & \,\,\, & Y'_d \did &&& \op{\eta} \\
	 & L \cl{\mu'} \did & \Brr && Y'_d \did && Y_d \dcell{g} \Brr & \,\,\, & Y_d \did \\
	 & L \did & \Brr && Y'_d && Y'_d \Brr && Y_d \did \\
	 & L & \Brr &&& B \cl{\eps} & \Brr && Y_d }}
\vcenter{\xymatrix{\stackrel{\triangle}{=} }}
\vcenter{\xymatrix@C=-0.3pc@R=1pc{Y'_d \did &&& \op{\eta} \\
	 Y'_d \did && Y_d \dcell{g} \Brr & \,\,\, & Y_d \did \\
	 Y'_d && Y'_d \Brr && Y_d \did \\
	 & L \cl{\mu'} & \Brr && Y_d}}
\vcenter{\xymatrix{\stackrel{(*)}{= }}}
\vcenter{\xymatrix@C=-0.3pc@R=1pc{Y'_d \did &&& \op{\eta} \\
	 Y'_d \did && Y_d \dcell{g} \Brr & \,\,\, & Y_d \did &&& \op{\eta} \\
	 Y'_d && Y'_d \Brr && Y_d && Y_d \Brr & \,\,\, & Y_d \did \\
	 & B \cl{\eps} & \Brr &&&	L \cl{\mu} & \Brr && Y_d}} $$
But the equality (*) is $\lozenge_2(g)$ composed with $\eta$, to recover $\lozenge_2(g)$ compose with $\eps$.
\end{proof}

\begin{corollary}
Using last proposition and proposition \ref{equivdeaccion} we can think of the category $\beta^G$ of discrete actions of $G$ in a purely algebraic way (without considering spaces over ${G_0}$) as follows: an action is a $B$-bimodule morphism $Y_d \otimes Y_d \mr{\mu} L$ satisfying B1, B2, and an action morphism is an arrow $Y \mr{g} Y'$ in $shB$ such that $\lozenge_2(g)$ holds.
\end{corollary}

\begin{remark} \label{coinciden2}
Since $\mu$ is an $\ell$-bijection, $\lozenge_2(g)$ holds if and only if $\rhd(g)$ does, so the definitions of action morphism of \cite{JT} (see \ref{defdeaction}) and \cite{D1} (see \ref{thecategorybetaG}) coincide.
\end{remark}

\begin{remark} \label{monoenshB}
Since the forgetful functor $\beta^G \mr{F} shB$, $G \curvearrowright \overline{Y_{d}} \mapsto Y$, is left exact, a monomorphism of discrete $G$-actions $Z \mr{g} Y$ is also a monomorphism in $shB$.
\end{remark}

\begin{lemma} \label{monocancela}
Given two actions $Y_d \otimes Y_d \mr{\mu} L$ and $Z_d \otimes Z_d \mr{\mu'} L$ and a monomorphism $Z \mr{g} Y$ of actions, for each $\delta_z$, $\delta_w$ generators of $Z_d$, $\mu'(\delta_z \otimes \delta_w) = \mu(\delta_{g(z)} \otimes \delta_{g(w)})$.
\end{lemma}

\begin{proof} 
$\mu(\delta_{g(z)} \otimes \delta_{g(w)}) \stackrel{\ref{accionsiidiam2}, \eqref{equationdeconoconmu}}{=} \displaystyle \bigvee_{\stackrel{b \in B}{x \in Y(b)}} \igu[g(x)]{g(z)}_B \cdot \mu' (\delta_x \otimes \delta_w) \stackrel{\ref{monoenshB}, \ref{prelim1}}{=} $

\hfill $= \displaystyle \bigvee_{\stackrel{b \in B}{x \in Y(b)}} \igu[x]{z}_B \cdot \mu' (\delta_x \otimes \delta_w) \stackrel{\ref{ecuacionenOmegaX}}{=} \mu' (\delta_z \otimes \delta_w). $
\end{proof}

\begin{lemma} \label{unicaaccionposible}
Given an action $Y_d \otimes Y_d \mr{\mu} L$ and a monomorphism $Z \mr{f} Y$, if the restriction of the action to $Z$ is an $\ell$-bijection, then it is an action. This is the only possible action on $Z$ that makes $f$ a morphism of $G$-actions.
\end{lemma}

\begin{proof}
Unicity is clear from the previous lemma. We have to check B1 and B2 in proposition \ref{equivdeaccion} for $Z_d \otimes Z_d \mr{\mu} L$. The only one that requires some care is B1. By hypothesis we have for $b_0,b'_0 \in B$, $x \in Y(b_0), w \in Y(b'_0)$, $$c \mu (\delta_x \otimes \delta_w) = \displaystyle \bigvee_{\stackrel{b \in B}{y \in Y(b)}} \mu (\delta_x \otimes \delta_y) \stackrel[B]{}{\otimes} \mu (\delta_y \otimes \delta_w)$$

\begin{center}
(we specify in the notation if the tensor product is over $B$).
\end{center}

We have to see that when $x \in Z(b_0), w \in Z(b'_0)$, this equation still holds when restricting the supremum to $Z$. In fact, in this case we have 

$$\displaystyle \bigvee_{\stackrel{b \in B}{y \in Y(b)}} \mu (\delta_x \otimes \delta_y) \stackrel[B]{}{\otimes} \mu (\delta_y \otimes \delta_w) \stackrel{\ref{restringirlosdelta}}{=}
\displaystyle \bigvee_{\stackrel{b \in B}{y \in Y(b)}} b_0 \cdot \mu (\delta_x \otimes \delta_y) \stackrel[B]{}{\otimes} \mu (\delta_y \otimes \delta_w) \cdot b'_0 \stackrel{ed), \ su)}{=}  $$

$$= \displaystyle \bigvee_{\stackrel{b \in B}{y \in Y(b)}} \bigvee_{\stackrel{b_1 \in B}{z_1 \in Z(b_1)}} \bigvee_{\stackrel{b_2 \in B}{z_2 \in Z(b_2)}} 
\mu (\delta_x \otimes \delta_{z_1}) \wedge \mu (\delta_x \otimes \delta_y) \stackrel[B]{}{\otimes} \mu (\delta_y \otimes \delta_w) \wedge \mu (\delta_{z_2} \otimes \delta_w)
\stackrel{uv), \ in), \ \ref{ecuacionenOmegaX}}{=} $$

\hfill $ \displaystyle = \bigvee_{\stackrel{b \in B}{z \in Z(b)}} \mu (\delta_x \otimes \delta_z) \stackrel[B]{}{\otimes} \mu (\delta_z \otimes \delta_w).$
\end{proof}

We are ready to prove theorem \ref{Comd=Relprevio}.

\begin{theorem} \label{Comd=Rel}
For any localic groupoid $G$ as in \ref{sinnadadeequivalence0}, there is an equivalence of categories making the triangle commutative ($T,F$ are forgetful functors):
$$
\xymatrix@C=0ex
        {
          Cmd_0(L) \ar[rr]^{\cong} \ar[rd]_T 
      & & \cc{R}el(\beta^G) \ar[ld]^(.4){\cc{R}el(F)} 
        \\
         & (B\hbox{-}Mod)_0 \cong \cc{R}el(shB).
        }       
$$ 
The identification between relations $R \subset Y \times Y'$ in $shB$ and $B$-module morphisms \linebreak $Y_d \mr{R} Y'_d$ lifts to the upper part of the triangle.
\end{theorem} 
\begin{proof}
Since the equivalence $(B\hbox{-}Mod)_0 \cong \cc{R}el(shB)$ maps $Y_d \leftrightarrow Y$, proposition \ref{equivdeaccion} yields a bijection between the objects of $Cmd_0(L)$ and $Rel(\beta^G)$.

We have to show that this bijection respects the arrows of the categories. 
Using the lemma \ref{unicaaccionposible}, it is enough to see that for $Y$, $Y'$ any two objects of $\beta^G$, and $R \subset Y \times Y'$ a relation in $shB$, the restriction $\theta$ of the product action $\lambda \boxtimes \lambda'$ to $R$ is a bijection if and only if the corresponding $B$-module map 
$R: Y_d \rightarrow Y'_d$ is a comodule morphism.

We claim that the diagram expressing that $R: Y_d \rightarrow Y_d'$ is a comodule morphism is equivalent to the diagram $\lozenge(R,R)$ in \ref{diagramadiamante12} (cf. proof of proposition \ref{conclaim}). The proof follows then by proposition \ref{combinacion}. We prove the claim using the elevators calculus \linebreak described in appendix \ref{ascensores}, \ref{ascensoresconB}:

The comodule morphism diagram is the equality
\begin{equation} \label{commorph}
\vcenter{\xymatrix@C=-0.3pc@R=1.5pc {Y_d \dcell{R} &&& \op{\eta} \\
														Y'_d && Y'_d \Brr & \,\,\, & Y'_d \did \\
														& G \cl{\mu'} & \Brr && Y'_d}}
\vcenter{\xymatrix@C=-0.3pc@R=1.5pc{= }}
\vcenter{\xymatrix@C=-0.3pc@R=1.5pc {Y_d \did &&& \op{\eta} \\
														Y_d && Y_d \Brr & \,\,\, & Y_d \dcell{R} \\
														& G \cl{\mu} & \Brr && Y'_d}}
\end{equation}
while the diagram $\lozenge$ is
\begin{equation} \label{grafdiam}
\vcenter{\xymatrix@C=-0.3pc@R=1.5pc         {  Y_d \did &&& \op{\eta} &&& Y'_d \did \\
																			Y_d \did && Y_d \did \Brr & \,\,\, & Y_d \dcell{R} && Y'_d \did \\
																			Y_d && Y_d \Brr && Y'_d && Y'_d \\
																			& G \cl{\mu} & \Brr &&& H \cl{\eps} }}
\vcenter{\xymatrix@C=-0.3pc@R=1.5pc         { =} }
\vcenter{\xymatrix@C=-0.3pc@R=1.5pc         {  Y_d \dcell{R} && Y'_d \did \\
																			Y'_d && Y'_d \\
																			& G \cl{\mu'} }}
\end{equation}

{\it Proof of \eqref{commorph} $\implies$ \eqref{grafdiam}}:
$$
\vcenter{\xymatrix@C=-0.3pc@R=1.5pc         {  Y_d \did &&& \op{\eta} &&& Y'_d \did \\
																			Y_d \did && Y_d \did \Brr & \,\,\, & Y_d \dcell{R} && Y'_d \did \\
																			Y_d && Y_d \Brr && Y'_d && Y'_d \\
																			& G \cl{\mu} & \Brr &&& H \cl{\eps} }}
\vcenter{\xymatrix@C=-0.3pc@R=1.5pc         { \stackrel{\eqref{commorph}}{=} }}
\vcenter{\xymatrix@C=-0.3pc@R=1.5pc         {  Y_d \dcell{R} &&& \op{\eta} &&& Y'_d \did \\
																			Y'_d && Y'_d \Brr & \,\,\, & Y'_d && Y'_d \\
																			& G \cl{\mu'} & \Brr &&& H \cl{\eps} }}
\vcenter{\xymatrix@C=-0.3pc@R=1.5pc         { \stackrel{(\triangle)}{=} }}
\vcenter{\xymatrix@C=-0.3pc@R=1.5pc         {  Y_d \dcell{R} && Y'_d \did \\
																			Y'_d && Y'_d \\
																			& G \cl{\mu'} }}
$$
																			
{\it Proof of \eqref{grafdiam} $\implies$ \eqref{commorph}}:		
$$
\vcenter{\xymatrix@C=-0.3pc@R=1.5pc {Y_d \dcell{R} &&& \op{\eta} \\
														Y'_d && Y'_d \Brr & \,\,\, & Y'_d \did \\
														& G \cl{\mu'} & \Brr && Y'_d}}
\vcenter{\xymatrix@C=-0.2pc @R=1pc{ \stackrel{\eqref{grafdiam}}{=} }}
\vcenter{\xymatrix@C=-0.3pc@R=1.5pc {	Y_d \did &&& \op{\eta} &&&& \op{\eta} \\
																			Y_d \did && Y_d \did \Brr & \,\,\, & Y_d \dcell{R} && Y'_d \did \Brr & \,\,\, & Y'_d \did \\
																			Y_d && Y_d \Brr && Y'_d && Y'_d \Brr && Y'_d \did \\
																			& G \cl{\mu} & \Brr &&& H \cl{\eps} & \Brr && Y'_d}}
\vcenter{\xymatrix@C=-0.2pc @R=1pc{ \stackrel{(\triangle)}{=} }}
\vcenter{\xymatrix@C=-0.3pc@R=1.5pc {Y_d \did &&& \op{\eta} \\
														Y_d && Y_d \Brr & \,\,\, & Y_d \dcell{R} \\
														& G \cl{\mu} & \Brr && Y'_d}}
$$
\end{proof}

\pagebreak

\section{The Galois and the Tannaka contexts} \label{sec:Contexts}

\textbf{The Galois context associated to a topos.} Consider an arbitrary topos over $\Sat$, \linebreak $\Eat \mr{} \Sat$. In \cite{JT}, VII.3 p.59-61, VIII.3 p.68-69 the following is proved. There is a spatial cover of $\Eat$, this is an open surjection of topos $\Xat \mr{q} \Eat$ with $\Xat = shG_0$ for a \mbox{$G_0 \in sp$.} The 2-kernel pair of $q$, $\xymatrix{\Xat \stackrel[\Eat]{{{}^{\ }}^{\ }}{\times} \Xat \ar@<1ex>[r]^>>>>>>{p_1} \ar@<-1ex>[r]_>>>>>>{p_2} & \Xat}$ satisfies that there is a localic groupoid \linebreak
$G= \xymatrix{G \stackrel[G_0]{\textcolor{white}{G_0}}{\times} G \ar[r]^>>>>>>{\circ} & G \ar@<1.3ex>[r]^{\partial_0} \ar@<-1.3ex>[r]_{\partial_1} & G_0 \ar[l]|{\ \! i \! \ }}$ 
such that 

\begin{equation} \label{igualdadJT} 
\xymatrix{\Xat \stackrel[\Eat]{{{}^{\ }}^{\ }}{\times} \Xat \ar@<1ex>[r]^>>>>>>{p_1} \ar@<-1ex>[r]_>>>>>>{p_2} & \Xat & = & shG \ar@<1ex>[r]^{\partial_0^*} \ar@<-1ex>[r]_{\partial_1^*} & shG_0}
\end{equation}

\begin{center}
(we use in this section the notation of $\cite{JT}$ for sheaves on a space, $shG = sh (\cc{O}(G))$, $shG_0 = sh (\cc{O}(G_0)) \ \ $ ). 
\end{center}

Joyal and Tierney use this to prove the equivalence $\Eat \cong \beta^G$ (Galois recognition theorem, see theorem \ref{fundamentalGT}) via descent techniques. They don't construct $G$ (they don't need to), though they make the remark (in p.70 of op. cit.) that (in the case $\cc{X} = \cc{S} \mr{} \cc{E}$, with $\cc{E}$ an atomic topos, corresponding to the neutral case of Tannaka, see \cite{DSz}) $G$ is the spatial group of automorphisms of (a model of the structure classified by the codomain of) the point. This idea was developed by Dubuc in \cite{D1}, who constructed $G$. Our work in this section can be considered as a generalization of those results of \cite{DSz}, \cite{D1}, and of the remark of \cite{JT}, since we construct $G$ in the general case and describe it as a universal $\rhd$-cone of $\ell$-bijections (which is a generalization of the description of the localic group of automorphisms of a point made in \cite{D1}, 4, p.152-155).

Equation \eqref{igualdadJT} means that the 2-kernel pair of $q$ can be computed as the following 2-push out in the 2-category of topoi with inverse images.

\begin{equation} \label{2po}
	\xymatrix@C=.5pc@R=.5pc{\Eat \ar[rrr]^{q^*} \ar[ddd]_{q^*} &&& shG_0 \ar[ddd]^{\partial_0^*}  \ar@/^3ex/[dddddrrr]^{f_0^*} \\ 
			&& \ar@{=>}[dl]^{\cong}_{\varphi}  &&&&&&  (\hbox{for each } \cc{F}, f_0^*, f_1^*, f_0^*q^* \Mr{\psi} f_1^*q^*,  \\
			&&&&& \ar@{=>}[dl]^{\cong}_{\psi}  &&&		\hbox{there exists a unique } \ell^* \hbox{ such that }   \\
			shG_0 \ar[rrr]_{\partial_1^*} \ar@/_2ex/[ddrrrrrr]_{f_1^*} &&& shG \ar@{.>}[ddrrr]|{\exists ! \ell^*} &&&&& \ell^* \partial_i^* = f_i^* \hbox{ and } id_{\ell^*} \circ \varphi = \psi)\\
			\\
			&&&&&& \cc{F} } 
\end{equation}

\begin{sinnadastandard} \label{cosasconA}
Take, as in section \ref{sec:Cmd=Rel}, $B = \O({G_0})$. By \ref{enumeratedeJT}, items 5,6, $(B \otimes B)$-locales \mbox{$B \otimes B \mr{g = (g_0,g_1)} A$} correspond to locales $\widetilde{A} \in Loc(sh(B \otimes B))$, $\gamma_* \widetilde{A} = A$ and the following diagram commutes.

\begin{equation} \label{dosdanuna}
\xymatrix{shB \ar[r]^{g_0^*} \ar[rd]_{\pi_1^*} & sh\widetilde{A} & shB \ar[l]_{g_1^*} \ar[dl]^{\pi_2^*} \\
						& sh(B\otimes B) \ar[u]^{g^*} }
\end{equation}

Consider also the following commutative diagram 

$$\xymatrix{ sh\widetilde{A}  \\
	    sh(B \otimes B) \ar[u]^{g^*} \\
	    \cc{S} \ar@/_8ex/[uu]_{\gamma^*} \ar[u]^{\gamma^*} }$$

Since the composition of spatial morphisms is spatial (see for example \cite{JohnstoneFactorization}, 1.1), then $sh\widetilde{A}$ is spatial (over $\cc{S}$), i.e. $sh\widetilde{A} \cong sh(\gamma_* \Omega_{\widetilde{A}} )$. But $\gamma_* \Omega_{\widetilde{A}} = \gamma_* {g}_* \Omega_{\widetilde{A}} = \gamma_* \widetilde{A} = A$.

\begin{equation} \label{abusotilde}
 \hbox{In the sequel, we make no distinction between } shA \hbox{ and } sh\widetilde{A}.
\end{equation}
\end{sinnadastandard}

\begin{sinnadastandard} \label{spatialreflection}
Recall from \cite{JT}, VI.5 p.53-54 the fact that
there is a left adjoint $F$ to the full and faithful functor $Loc^{op}(\cc{S}) \stackrel{sh}{\hookrightarrow} Top/\cc{S}$, 
that maps $\cc{E} \mr{p} \cc{S}$ to $F(\cc{E}) = \overline{p_*(\Omega_{\cc{E}})}$. 
\end{sinnadastandard}

\begin{lemma} \label{2poigual2polocalic}
 The universal property defining the 2-push out \eqref{2po} is equivalent to the following universal property for localic topoi:
 
\begin{equation} \label{2polocalic}
	\xymatrix@C=.5pc@R=.5pc{\Eat \ar[rrr]^{q^*} \ar[ddd]_{q^*} &&& shG_0 \ar[ddd]^{\partial_0^*}  \ar@/^3ex/[dddddrrr]^{g_0^*} \\ 
			&& \ar@{=>}[dl]^{\cong}_{\varphi}  &&&&&&  (\hbox{for each } A, g_0^*, g_1^*, g_0^*q^* \Mr{\phi} g_1^*q^*,  \\
			&&&&& \ar@{=>}[dl]^{\cong}_{\phi}  &&&		\hbox{there exists a unique } h^* \hbox{ such that }   \\
			shG_0 \ar[rrr]_{\partial_1^*} \ar@/_2ex/[ddrrrrrr]_{g_1^*} &&& shG \ar@{.>}[ddrrr]|{\exists ! h^*} &&&&& h^* \partial_i^* = g_i^* \hbox{ and } id_{h^*} \circ \varphi = \phi)\\
			\\
			&&&&&& shA } 
\end{equation}
 
\end{lemma}

\begin{proof}
 Of course \eqref{2po} implies \eqref{2polocalic}. To show the other implication, given $\cc{F}, f_0^*, f_1^*, \psi$ as in \eqref{2po}, consider $\cc{F}$ as a topos over $sh (G_0 \times G_0)$ via $\cc{F} \mr{f = (f_0,f_1)} sh (G_0 \times G_0)$ and apply $F$ as in \ref{spatialreflection}. Then $\cc{O}(F(\cc{F})) = f_* \Omega_{\cc{F}}$ is a locale in $sh (G_0 \times G_0)$. Take $A = \gamma_* f_* \Omega_{\cc{F}}$ the corresponding locale over $B \otimes B$, $B \otimes B \mr{g = (g_0,g_1)} A$, i.e. $\widetilde{A} = f_* \Omega_{\cc{F}}$, then we have the commutative diagram \eqref{dosdanuna}.
 
 The hyperconnected factorization of $f$ is $\vcenter{\xymatrix{\cc{F} \ar[rd]_f \ar[rr]^{\eta} && shA \ar[ld]^g \\ & sh(G_0 \times G_0)}}$, where $\eta$ is the unit of the adjunction described in \ref{spatialreflection}. $\eta$ is hyperconnected (see \cite{JT}, VI.5 p.54), in particular $\eta^*$ is full and faithful (see \cite{JohnstoneFactorization}, 1.5). Then $\eta^* g_0^* q^* \Mr{\psi} \eta^* g_1^* q^*$ determines uniquely $g_0^* q^* \Mr{\phi} g_1^* q^*$ such that $id_{\eta^*} \circ \phi = \psi$ and applying \eqref{2polocalic} we obtain

  $$\xymatrix@C=.5pc@R=.5pc{\Eat \ar[rrr]^{q^*} \ar[ddd]_{q^*} &&& shG_0 \ar[ddd]^{\partial_0^*}  \ar@/^3ex/[dddddrrr]^{g_0^*}  \ar@/^6ex/[dddddddrrrrr]^{f_0^*} \\ 
			&& \ar@{=>}[dl]^{\cong}_{\varphi}  &&&&&&    \\
			&&&&& \ar@{=>}[dl]^{\cong}_{\phi}  &&&		  \\
			shG_0 \ar[rrr]_{\partial_1^*} \ar@/_2ex/[ddrrrrrr]_{g_1^*} \ar@/_6ex/[ddddrrrrrrrr]_{f_1^*} &&& shG \ar@{.>}[ddrrr]|{\exists ! h^*} &&&&& \\
			\\
			&&&&&& shA \ar[rrdd]^{\eta^*} \\ 
			\\
			&&&&&&&& \cc{F}} $$

Now, by the adjunction described in \ref{spatialreflection}, since taking sheaves is full and faithful, we have a bijective correspondence between morphisms $h^*$ and $\ell^*$ in the following commutative diagram:

\begin{equation} \label{adjuncionconeta}
\xymatrix{\cc{F} & shA \ar[l]_{\eta^*} \\
	  & shG \ar[u]_{h^*} \ar[ul]_{\ell^*} \\
	  & sh(G_0 \times G_0) \ar[uul]^{f^*} \ar[u]_{\partial^*} \ar@/_5ex/[uu]_{g^*}  }
\end{equation}

To end the proof, we have to show that under this correspondence the conditions of \eqref{2po} and \eqref{2polocalic} are equivalent. The equivalence between $l^* \partial^* = f^*$ and $h^* \partial^* = g^*$ is immediate considering \eqref{adjuncionconeta}, and the equivalence between $id_{l^*} \circ \varphi = \psi$ and $id_{h^*} \circ  \varphi = \phi$ follows from $id_{\eta^*} \circ \phi = \psi$ using that $\eta^*$ is full and faithful.

\end{proof}

\begin{sinnadastandard} \label{correspondanceparagalois} 
Consider a $B\otimes B$-locale $A$ as in \ref{cosasconA}. We have the correspondence

\renewcommand{\arraystretch}{-3}
\begin{tabular}{c c}
$g_0^* q^* \Mr{\phi} g_1^* q^* \hbox{ a natural isomorphism}$  & ${}_{_{\textstyle \hbox{by } \eqref{dosdanuna} }}$ \\ \cline{1-1} \noalign{\smallskip}
$g^* \pi_1^*  q^* \Mr{\phi} g^* \pi_2^* q^* \hbox{ a natural isomorphism}$ & ${}_{_{\textstyle \hbox{by } \ref{conodeellrelationsesnattransf} }}$ \\ \cline{1-1} \noalign{\smallskip}
$\hbox{A } \lozenge_1\hbox{-cone } \pi_1^* q^* X \times \pi_2^* q^* X \mr{\alpha_X} \widetilde{A} \; \hbox{ of } \ell \hbox{-bijections (in } sh(B \otimes B))$ & ${}_{_{\textstyle \hbox{by } \ref{dim1ydim2esdim} \hbox{, } \ref{trianguloesdiamante} }}$ \\ \cline{1-1} \noalign{\smallskip}
$\hbox{A } \rhd\hbox{-cone } \pi_1^* q^* X \times \pi_2^* q^* X \mr{\alpha_X} \widetilde{A} \; \hbox{ of } \ell \hbox{-bijections (in } sh(B \otimes B))$
\end{tabular}

In particular for $L = \O(G)$, the locale morphisms $\xymatrix{B \ar@<1ex>[r]^{s=\partial_0^{-1}} \ar@<-1ex>[r]_{t=\partial_1^{-1}} & L}$ induce a locale morphism $\xymatrix{B \otimes B \ar[r]^>>>>>{\gamma=(b,s)} & L}$, and $\partial_0^* q^* \Mr{\varphi} \partial_1^* q^*$ correspond to a $\rhd$-cone $\pi_1^* q^* X \times \pi_2^* q^* X \mr{\lambda_X} \widetilde{L}$ of $\ell$-bijections.

\end{sinnadastandard}

\begin{theorem} \label{JTGeneral}
Given the previous data, \eqref{2po} is a $2$-push out if and only if $\lambda$ is universal as a $\rhd$-cone of $\ell$-bijections (in the topos $sh(B \otimes B)$) in the following sense:

\begin{equation} \label{AutF}
\xymatrix
        {
         \pi_1^* q^* X \times \pi_2^* q^* X \ar[rd]^{\lambda_{X}}  
                      \ar@(r, ul)[rrd]^{\alpha_{X}} 
                      \ar[dd]_{\pi_1^* q^* (f) \times \pi_2^* q^* (f)}  
        \\
         {} \ar@{}[r]^(.3){\geq}
         & \;\; \widetilde{L} \;\; \ar@{-->}[r]^{\exists ! h} 
         & \;\widetilde{A}.  
        \\
         \pi_1^* q^* Y \times \pi_2^* q^* Y \ar[ru]^{\lambda_{Y}} 
         \ar@(r, dl)[rru]^{\alpha_{Y}} 
         && {(h \; \text{a locale morphism})}
        } 
\end{equation}
\end{theorem}

\begin{proof} By lemma \ref{2poigual2polocalic} it suffices to show that \eqref{2polocalic} is equivalent to \eqref{AutF}. We have shown in \ref{correspondanceparagalois} that $\varphi$, $\phi$ in \eqref{2polocalic} correspond to $\lambda$, $\alpha$ in \eqref{AutF}.
Since taking sheaves is full and faithful, a morphism $\widetilde{L} \mr{h} \widetilde{A}$ of locales in $sh(B \otimes B)$ corresponds to the inverse image $shL \mr{h^*} shA$ (recall \eqref{abusotilde}) of a topoi morphism over $sh(B \otimes B)$, i.e. $h^*$ as in \eqref{2polocalic} satisfying $h^* \partial_i^* = f_i^*$, $i=0,1$. It remains to show that $h \lambda_X = \alpha_X$ for each $X$ in \eqref{AutF} if and only if $id_{h^*} \circ \varphi = \psi$ in \eqref{2polocalic}.

In the correspondence between $h$ and $h^*$ above, $\widetilde{L} \mr{h} \widetilde{A}$ is given by the value of $h^*$ in the subobjects of $1$ ($\widetilde{L} = \gamma_* \Omega_{shL}$, $\widetilde{A} = f_* \Omega_{shA}$), then we are in the hypothesis of \ref{atravesdeunmorfismodetopos} as the following diagram shows 

$$\xymatrix@R=3.5pc{& shL  \df{rr}{h^*}{h_*} && shA  \\
	  \cc{E} \ar@<.5ex>[rr]^{\pi_1^* q^*} \ar@<-.5ex>[rr]_{\pi_2^* q^*} && sh(B \otimes B), \dfbis{ul}{\gamma^*}{\gamma_*} \dfbis{ur}{f^*}{f_*}  }$$

and the proof finishes by corollary \ref{naturaligualconeconrhoatravesdeunmorfismodetopos}.
\end{proof}

\begin{remark} \label{levdegalois}
 From proposition \ref{equivdeaccion}, we have that for each $X \in \Eat$, $\pi_1^* q^* X \times \pi_2^* q^* X \mr{\lambda_{X}} \widetilde{L}$ is equivalent to a discrete action $G \times_{G_0} X_{dis} \mr{\theta} X_{dis}$. In this way we can construct a lifting $\Eat \mr{\widetilde{q^*}} \beta^G$. This is the lifting $\Eat \mr{\phi} Des(q)$ of \cite{JT}, VIII.1 p.64, composed with the equivalence $Des(q) \mr{\cong} \beta^G$ given by the correspondence in \ref{correspondanceparagalois} for each $X$ (see \cite{JT}, VIII.3 proof of theorem 2, p.69).
\end{remark}

\begin{sinnadastandard} {\bf The Tannakian context associated to a topos.} \label{tannakacontext}

For generalities, notation and terminology  concerning Tannaka theory see appendix \ref{sec:Tannaka}. 
Consider the fiber functor associated to the topos $\cc{E}$ (see \ref{rel}): $$\cc{R}el(\cc{E}) \mr{F = \cc{R}el(q^*)} (B \hbox{-} Mod)_0, \quad FX = (q^*X)_d.$$
This determines a Tannakian context as in appendix \ref{sec:Tannaka}, with 
$\cc{X} = \cc{R}el(\cc{E})$, $\cc{V} = s \ell$.
\end{sinnadastandard}

The universal property which defines the coend $End^\lor(F)$ is that of a universal $\lozenge$-cone in the category of $(B \otimes B)$-modules, as described in the following diagram: 
$$
\xymatrix
        {
         & FX \otimes FX \ar[rd]^{\mu_{X}}  
                      \ar@(r, ul)[rrd]^{\phi_{X}}  
        \\
         FX \otimes FY \ar[rd]_{F(R) \otimes FY \quad} 
			           \ar[ru]^{FX \otimes F(R)^{\wedge} \;} 
	     & \equiv 
	     & \;\;End^\lor(F)\;\; \ar@{-->}[r]^{\phi}  
         & \;Z.  
        \\
         & FY \otimes FY \ar[ru]^{\mu_{Y}} 
                          \ar@(r, dl)[rru]^{\phi_{Y}} 
         && {(\phi \; \text{a linear map})}
        } 
$$

Via the equivalence $B\otimes B$-$Mod \cong s \ell(sh(B \otimes B))$, we can also think of this coend internally in the topos $sh(B \otimes B)$ as

$$
\xymatrix
        {
         & \pi_1^* F X \times \pi_2^* F X \ar[rd]^{\lambda_{X}}  
                      \ar@(r, ul)[rrd]^{\phi_{X}}  
        \\
         \pi_1^* F X \times \pi_2^* F Y \ar[rd]_{\pi_1^*F(R) \times \pi_2^*FY \quad} 
			           \ar[ru]^{\pi_1^*FX \times \pi_2^*F(R)^{\wedge} \;} 
	     & \equiv 
	     & \;\;End^\lor(F)\;\; \ar@{-->}[r]^{\phi}  
         & \;Z.  
        \\
         & \pi_1^* F Y \times \pi_2^* F Y \ar[ru]^{\lambda_{Y}} 
                          \ar@(r, dl)[rru]^{\phi_{Y}} 
         && {(\phi \; \text{a linear map})}
        } 
$$

Depending on the context, it can be convenient to think of $End^{\lor}(F)$ as a \linebreak $(B \otimes B)$-module or in $s \ell(sh(B \otimes B))$: to use general Tannaka theory, we consider modules, but to use the theory of $\lozenge$-cones developed in section \ref{sec:cones} we work internally in the topos $sh(B \otimes B)$.
We apply proposition \ref{extension} to obtain: 
\begin{proposition} \label{EndhasX}
The large coend defining $End^\lor(F)$ exists and can be computed by the coend corresponding to the restriction of $\, F$ to the full subcategory of $Rel(\Eat)$ whose objects are in any small site $\Cat$ of definition of $\cc{E}$. \cqd
\end{proposition}

We fix a small site $\Cat$ (with binary products and $1$) of the topos $\cc{E}$. Then $End^\lor(F)$ can be constructed as a $(B \otimes B)$-module with generators $\mu_C(\delta_a \otimes \delta_b)$, where $\delta_a$, $\delta_b$ are the generators of $FC = (q^*C)_d$ (see proposition \ref{formula}), subject to the relations that make the $\lozenge$-diagrams commute. We will denote $[C,\delta_a,\delta_b] = \mu_C ( \delta_a \otimes \delta_b )$.

By the general Tannaka theory we know that $End^\lor(F)$ is a cog\`ebro\"ide agissant sur $B$ and a $(B \otimes B)$-algebra. 
The description of the multiplication $m$ and the unit $u$ given below proposition \ref{bialg} yields in this case, for $C, \, D \in \cc{C}$ (here, $F(I) = F(1_\Cat) = B$):
\begin{equation} 
m([C,\, \delta_a,\delta_{a'}],\, [D, \,\delta_b,\delta_{b'}]) \;= \; [C \times D,\, (\delta_a \otimes \delta_b),(\delta_{a'} \otimes \delta_{b'})],\; \;\; u = \lambda_1.
\end{equation} 
When interpreted internally in $sh(B \otimes B)$, this shows that $\pi_1^* q^* C \times \pi_2^* q^* C \mr{\lambda_C} End^\lor(F)$ is a compatible 
$\lozenge$-cone, with $End^\lor(F)$ generated as a sup-lattice in $sh(B \otimes B)$ by the elements $\lambda_C ( a, b )$, thus by proposition \ref{compislocale} it follows that $End^\lor(F)$ is a locale.

By proposition \ref{hopf}, we obtain that $End^\lor(F)$ is also a (localic) \textit{Hopf cog\`ebro\"ide}, i.e. the dual structure in $Alg_{s \ell}$ of a localic groupoid.

\vspace{1ex}

\begin{sinnadastandard} 
{\bf The construction of $G$.} 
\end{sinnadastandard} 

\begin{\prop}
 Take $L = End^{\lor}(F)$. Then $G = \overline{L}$ satisfies \eqref{AutF}, i.e. (by theorem \ref{JTGeneral}) satisfies \eqref{2po}.
\end{\prop}

\begin{proof}
 Given a $\rhd$-cone of $\ell$-bijections over a locale $A$, by proposition \ref{trianguloesdiamante} it factors uniquely through a $s \ell$-morphism which by proposition \ref{supisloc} is a locale morphism.
\end{proof}

We show now that $G$ is the localic groupoid considered by Joyal and Tierney. By theorem \ref{JTGeneral}, the dual $L$ of a groupoid $G$ satisfying \eqref{2po} is unique as a locale in  \linebreak $sh(B \otimes B)$, and so are the $\lambda_X$ corresponding to the $\varphi$ in \eqref{2po}.

Now, remark \ref{remarkdeunicidad}, interpreted for $G = \overline{L}$ using proposition \ref{equivdeaccion}, states that $i = \overline{e}$, $\circ = \overline{c}$ are the only possible localic groupoid structure (with inverse given as $(-)^{-1} = \overline{a}$, see proposition \ref{hopf}) such that the lifting $\widetilde{q^*}$ lands in $\beta^G$ (see remark \ref{levdegalois}). We have proved:

\begin{theorem} \label{G=H}
Given any topos $\Eat$ over a base topos $\Sat$, and a spatial cover \mbox{$sh{G_0} \mr{q} \Eat$,} the localic groupoid $G = \xymatrix{G \stackrel[G_0]{\textcolor{white}{G_0}}{\times} G \ar[r]^>>>>>>{\circ} & G \ar@<1.3ex>[r]^{\partial_0} \ar@<-1.3ex>[r]_{\partial_1} & G_0 \ar[l]|{\ \! i \! \ }}$ considered in \cite{JT} is unique and can be constructed as $G = \overline{End^{\lor}(Rel(q^*))}$, with $i = \overline{e}$, $\circ = \overline{c}$ and inverse $(-)^{-1} = \overline{a}$. The lifting $\Eat \mr{\widetilde{q^*}} \beta^G$ is also unique and defined as in remark \ref{levdegalois}. \cqd  
\end{theorem}

\pagebreak

\section{$s \ell$-Tannakian Categories} \label{theorems}
A spatial cover of a topos $shB \mr{q} \cc{E}$, with inverse image $\cc{E} \mr{q^*} shB$, determines by theorem \ref{G=H} a situation described in the following diagram (cf. \eqref{neutraldiagramafundamental})

\begin{equation} \label{diagramacompleto}
\xymatrix
        {
          \beta^G            \ar[r] \ar[rdd]
        & \cc{R}el(\beta^G)  \ar[rr]^{\cong} \ar[rdd]
        &           
        & Cmd_0(L)          \ar[ldd]
        \\
        & \cc{E}             \ar[d]^{q^*} \ar[r] \ar[lu]_{\widetilde{q^*}}
        & \cc{R}el(\cc{E})   \ar[ul]_{\cc{R}el(\widetilde{q^*})}  
                             \ar[d]^F \ar[ur]^{\widetilde{F}}
        \\
        & shB            \ar[r]
        & Rel(shB) \cong (B\hbox{-}Mod)_0 .
        }
\end{equation}

\vspace{1ex}

\noindent
where $F = \cc{R}el(q^*)$, $L = End^\lor(F)$, $G = \overline{L}$ and the isomorphism in the first row of the diagram is given by Theorem \ref{Comd=Rel}. 

\begin{theorem} [cf. theorem \ref{neutralAA}] \label{nonneutralAA}
The (Galois) lifting functor $\widetilde{q^*}$ is an equivalence if and only if the (Tannaka) lifting functor $\widetilde{F}$ is such. \cqd
\end{theorem}

From \cite{JT}, VIII.3, theorem 2, p.68 (see also remark \ref{levdegalois}), we have

\begin{theorem} \label{fundamentalGT}
The (Galois) lifting functor $\widetilde{q^*}$ is an equivalence. \cqd
\end{theorem}

We obtain

\begin{theorem} \label{AA}
The (Tannaka) lifting functor $\widetilde{F}$ is an equivalence. \cqd
\end{theorem} 

We make now the first developments of a theory that we call $s \ell$-tannakian theory. Theorem \ref{AA} yields the first examples of non-neutral $s \ell$-tannakian categories, the categories of relations of Grothendieck topoi.

We begin with some considerations regarding size, that will let us construct under some hypothesis the coend $End^{\lor}$ of a $s \ell$-enriched functor.

We work as before over a base topos $\Sat$, and denote $s \ell = s \ell(\cc{S})$, $Rel \cong s \ell_0 = s \ell_0(\cc{S})$. Let $\cc{A}$ be a $s \ell$-enriched category. Let $T,T': \cc{A} \mr{} s \ell_0$ be two $s \ell$-functors, $L \in s \ell$. Then we define a \emph{$\lozenge$-cone over $L$} as a family $TX \otimes T'X \mr{\lambda_X} L$, for $X \in \cc{A}$ such that for each $X \mr{f} Y$ in $\cc{A}$, the $\lozenge(f)$ diagram 
$$\xymatrix@C=4ex@R=3ex
        {
         & TX \times T'X \ar[rd]^{\lambda_{X}}  
        \\
           TX \times T'Y 
               \ar[rd]_{Tf \times T'Y \hspace{2.5ex}} 
			   \ar[ru]^{TX \times T'f^{\wedge}\hspace{2.5ex}} 
	     & \equiv 
         & L
        \\
         & TY \times T'Y \ar[ru]_{\lambda_{Y}} 
         }$$
commutes. This generalizes definition \ref{defdeconos}, if $\cc{A} = Rel(\cc{E})$ they coincide.

We consider now the following concepts from \cite{Pitts}, 1.7. p.442

\begin{\de}
 A collection $\cc{B}$ of objects of a $s \ell$-category $\cc{A}$ is $s \ell$-generating if, for each $X \in \cc{A}$, 
 \begin{equation} \label{eqdegenerating}
1_X = \bigvee_{(f,r) \in \cc{F}_X} f \circ r, \hbox{ where }  
 \end{equation}
 $\cc{F}_X = \{ (f,r) \hbox{ arrows of } \cc{A} \hbox{ s.t. } cod(r) = dom(f) \in \cc{B} \hbox{ and } f \circ r \leq 1_X \} $.
 
 $\cc{A}$ is bounded if it possesses a small collection of $s \ell$-generators.
\end{\de}

The motivating example is given by sites of topos, see \ref{ejemplodeDCR} below or \cite{Pitts}, 1.8. p.443.

\begin{\prop} \label{extension2} [cf proposition \ref{extension}] For a $s \ell$-generating collection $\cc{B}$ of objects of a $s \ell$-category $\cc{A}$, suitable cones defined in $\cc{B}$ (considered as a full subcategory of $\cc{A}$) can be extended to $\cc{A}$, more precisely:

Let $TC \times T'C \mr{\lambda_C} L$ be a $\lozenge$-cone defined in $\cc{B}$. Then, there are unique $TX \times T'X \mr{\lambda_X} L$ for all objects $X \in \cc{A}$ in such a way to determine a $\lozenge$-cone extending $\lambda$. 
\end{\prop}

\begin{proof}
 Let $X \in \cc{A}$, then by \eqref{eqdegenerating} we have $\displaystyle 1_X =  \bigvee_{(f,r) \in \cc{F}_X} f \circ r$, therefore \linebreak $\displaystyle 1_{TX} =  \bigvee_{(f,r) \in \cc{F}_X} T(f) \circ T(r)$, i.e. for each $x \in FX$, (1) $\displaystyle x = \bigvee_{(f,r) \in \cc{F}_X} T(f) T(r) x$.
 
 If $TX \times T'X \mr{\lambda_X} L$ is a $\lozenge$-cone extending $\lambda$, in particular $\lozenge(f)$ should hold: if $dom(f)=C$ then for each $c \in TC$, $x' \in T'X$,
 
 $$\lozenge(f): \ \lambda_X(T(f)c \otimes x' ) = \bigvee_{c' \in T'C} \igu[T'(f)c']{x'} \cdot \lambda_C (c \otimes c').$$
 
 Then $\displaystyle \lambda_X(x \otimes x' ) \stackrel{(1)}{=} \bigvee_{(f,r) \in \cc{F}_X} \lambda_X( T(f) T(r) x \otimes x' ) \stackrel{\lozenge(f)}{=} $
 
 \hfill $\displaystyle \bigvee_{(f,r) \in \cc{F}_X}  \bigvee_{c' \in T'C}  \igu[T'(f)c']{x'} \cdot \lambda_C (T(r)x \otimes c')$.
 
 That is the only possible definition of $\lambda_X$. Let's check that it is in fact a $\lozenge$-cone, if $X \mr{\varphi} Y$ is an arrow in $\cc{A}$ we must show that $\lozenge(\varphi)$ holds: for each $x \in TX$, $y' \in T'Y$, 
 $$\lozenge(\varphi): \ \lambda_Y(T(\varphi)x \otimes y' ) = \bigvee_{x' \in T'X} \igu[T'(\varphi)x']{y'} \cdot \lambda_X (x \otimes x').$$
 
\noindent We make the following previous computations:
 
 (1) Since $\displaystyle 1_X =  \bigvee_{(f,r) \in \cc{F}_X} f \circ r$, then $\displaystyle \varphi =  \bigvee_{(f,r) \in \cc{F}_X} \varphi \circ f \circ r$.
 
 (2) Since $\displaystyle 1_Y =  \bigvee_{(g,q) \in \cc{F}_Y} g \circ q$, then $\displaystyle \varphi =  \bigvee_{(g,q) \in \cc{F}_Y} g \circ q \circ \varphi$.
 
 (3) For each $(f,r) \in \cc{F}_X$, $(g,q) \in \cc{F}_Y$ we consider $\psi= q \circ \varphi \circ f$, then \linebreak $dom(\psi) = dom(f) = C \in \cc{B}$, $cod(\psi) = cod(q) = D \in \cc{B}$ and therefore $\lozenge(\psi)$ holds by hypothesis, i.e. for each $c \in TC$, $d' \in T'D$ 
 $$\lozenge(\psi): \ \lambda_D(T(q)T(\varphi)T(f)c \otimes d' ) = \bigvee_{c' \in T'C} \igu[T'(q)T'(\varphi)T'(f)c']{d'} \cdot \lambda_C (c \otimes c').$$
 
 We now compute

\flushleft $\displaystyle \lambda_Y(T(\varphi)x \otimes y' ) \stackrel{def.}{=} \bigvee_{(g,q) \in \cc{F}_Y}  \bigvee_{d' \in T'D}  \igu[T'(g)d']{y'} \cdot \lambda_C (T(q)T(\varphi)x \otimes d') \stackrel{(1)}{=}$
 
 \hfill $\displaystyle \bigvee_{(g,q) \in \cc{F}_Y}  \bigvee_{(f,r) \in \cc{F}_X}   \bigvee_{d' \in T'D}  \igu[T'(g)d']{y'} \cdot \lambda_D (T(q)T(\varphi)T(f)T(r)x \otimes d') \stackrel{\lozenge(\psi) \; in \; (3)}{=}$
 
 \hfill $\displaystyle  \bigvee_{(g,q) \in \cc{F}_Y}  \bigvee_{(f,r) \in \cc{F}_X}   \bigvee_{d' \in T'D}   \bigvee_{c' \in T'C}  \igu[T'(g)d']{y'}  \cdot  \igu[T'(q)T'(\varphi)T'(f)c']{d'}  \cdot  \lambda_C (T(r)x  \otimes  c') = $
 
  \hfill $\displaystyle  = \bigvee_{(g,q) \in \cc{F}_Y}  \bigvee_{(f,r) \in \cc{F}_X}  \bigvee_{c' \in T'C} \igu[T'(g)T'(q)T'(\varphi)T'(f)c']{y'} \cdot \lambda_C (T(r)x \otimes c') \stackrel{(2)}{=} $
 
 \hfill $\displaystyle  \bigvee_{(f,r) \in \cc{F}_X}  \bigvee_{c' \in T'C} \igu[T'(\varphi)T'(f)c']{y'} \cdot \lambda_C (T(r)x \otimes c') =  $
 
 \hfill $\displaystyle  \bigvee_{(f,r) \in \cc{F}_X}  \bigvee_{c' \in T'C} \bigvee_{x' \in T'X} \igu[T'(\varphi)x']{y'} \cdot \igu[T'(f)c']{x'} \cdot \lambda_C (T(r)x \otimes c') \stackrel{def.}{=}  $
 
 \hfill $\displaystyle  \bigvee_{x' \in T'X}    \igu[T'(\varphi)x']{y'} \cdot \lambda_X(x \otimes x')$.
 
\end{proof}

\begin{sinnadastandard} {\bf Tannaka theory for DCRs}
We will now express theorem \ref{AA} as a tannakian recognition-type theorem for some special type $s \ell$-enriched categories, distributive categories of relations (DCR), that generalize the categories of relations $Rel(\cc{E})$ of topoi. 
We recall from \cite{Pitts}, chapter 2 p.443-451 the following definitions and constructions:

\begin{\de}
 A distributive category of relations (DCR) $\cc{A}$ is a cartesian $s \ell$-category in which every object is discrete (see \cite{Pitts}, 2.1 p.444 for details). A morphism of DCRs is a $s \ell$-functor which preserves this structure (see \cite{Pitts}, 2.4 p.447 for details). A DCR $\cc{A}$ is complete if it has small coproducts (as a category) and if all symmetric idempotents in $\cc{A}$ split (see \cite{Pitts}, p.448 for details).
\end{\de}

\begin{sinnadastandard} \label{ejemplodeDCR}
The motivating example is: if $\cc{E}$ is a (Grothendieck) topos, then any full subcategory of $Rel(\cc{E})$ whose objects are closed under finite products in $\cc{E}$ is a DCR. In 
\cite{CW}, p.31, Theorem 6.3 (see also \cite{Pitts}, 2.5 p.448), Carboni and Walters prove that a DCR $\cc{A}$ is isomorphic to $Rel(\Eat)$ for a Grothendieck topos $\cc{E}$ if and only if $\cc{A}$ is bounded and complete. Furthermore, $Rel$ yields an equivalence of 2-categories between the dual of the category of Grothendieck topoi and the category of bounded, complete DCRs (\cite{Pitts}, 2.5).
\end{sinnadastandard}

\begin{sinnadastandard} \label{completion}
For any DCR $\cc{A}$ there exists its completion $\widehat{\cc{A}}$ (see \cite{Pitts}, 2.6 p.448). This is a complete DCR together with a full and faithful morphism of DCRs $\cc{A} \mr{\eta} \widehat{\cc{A}}$ (the counit of the inclusion of the category of complete DCRs into the category of DCRs) such that the objects in the image of $\eta$ are $s \ell$-generating in $\widehat{\cc{A}}$. $\eta$ is an equivalence of categories if and only if $\cc{A}$ is already complete.
\end{sinnadastandard}

Let now $\cc{A}$ be a bounded DCR. Since its completion $\widehat{\cc{A}}$ is a bounded and complete DCR, there exists a topos $\cc{E}$ such that $\widehat{\cc{A}} \cong Rel(\cc{E})$. We consider the situation of diagram \eqref{diagramacompleto} for this $\cc{E}$. We think of the $s \ell$-functor $\cc{A} \mr{\eta} \widehat{\cc{A}} \mr{F} (B$-$Mod)_0$ as a tannakian fiber functor and obtain the following tannakian recognition-type theorem for bounded DCRs:

\begin{theorem} \label{recognitionforsl}
Let $\cc{A}$ be a bounded DCR. Then the coend $L = End^{\lor}(F \circ \eta)$ exists, and the lifting $\cc{A} \mr{\widehat{F \circ \eta}} Cmd_0(L)$ is an equivalence of categories if and only if $\cc{A}$ is complete.
\end{theorem}

\begin{proof}
 Since the objects in the image of $\eta$ are $s \ell$-generating in $\widehat{\cc{A}}$, by proposition \ref{extension2} we obtain $L = End^{\lor}(F)$. Then the lifting $\cc{A} \mr{\widehat{F \circ \eta}} Cmd_0(L)$ is equal to \linebreak $\cc{A} \mr{\eta} \widehat{\cc{A}} \mr{\widetilde{F}} Cmd_0(L)$, and by theorem \ref{AA} we obtain that $\widehat{F \circ \eta}$ is an equivalence of categories if and only if $\eta$ is, which by \ref{completion} happens if and only if $\cc{A}$ is complete.
\end{proof}
\end{sinnadastandard}

\begin{sinnadastandard} {\bf A general recognition theorem for $s \ell$-tannakian categories (future work).}

We end this thesis by briefly describing the contents of a theory that arises naturally as a result of our work relating Galois and Tannaka theories. This is future work, we pose conjectures which we plan to investigate.

We have shown that theorem \ref{AA} (and therefore theorem \ref{fundamentalGT}, i.e. Theorem 2 of \cite{JT}, VIII.3) corresponds to a tannakian recognition theorem for a particular case of $s \ell$-enriched categories. In a sense, this theorem combines a recognition (the lifting is an equivalence) theorem and an ``additional structure'' ($G$ is a localic groupoid instead of just a localic category) theorem (see our introduction, On Tannaka Theories). But we can also consider a tannakian context for a general $s \ell$-enriched category, not neccesarily the category of relations of a topos:

\begin{sinnadastandard} \label{tannakiancontextsl}
Let $B \in Alg_{s \ell}$, $\cc{A}$ a $s \ell$-enriched bounded category, $\cc{A} \mr{F} (B$-$Mod)_0$ a functor. Define $L$ as in definition \ref{defdeL}, then it exists by proposition \ref{extension2} and we have the lifting $\widetilde{F}$ as in proposition \ref{lifting}.
\end{sinnadastandard}

The fundamental property of the open spatial cover $\cc{X} \mr{q} \cc{E}$ that is used in \cite{JT}, VIII.3 to prove Theorem 2 is that $q$ is an open surjection. In \cite{Pitts}, lemma 4.3, it is shown that under the equivalence $Top^{op} \mr{Rel} DCR$ (the inverse image of) an open surjection corresponds to an open morphism of DCRs (see \cite{Pitts} 2.4 (ii)) that is faithful as a functor. The definition of an open morphism between DCRs (\cite{Pitts}, 4.1) uses only their underlying structure of $s \ell$-enriched categories, therefore we may consider open faithful $s \ell$-functors between $s \ell$-categories.

Based on our previous developments, we make the conjecture that the following more general recognition theorem should hold (or that at least it is worth researching), of which theorem \ref{AA} (and therefore theorem \ref{fundamentalGT}) would be a particular case, for \linebreak $s \ell$-enriched categories and comodules of a (not neccesarily Hopf) cog\"ebro\`ide.

\begin{conjecture} \label{conjetura}
 Under the hypothesis of \ref{tannakiancontextsl}, if $F$ is a $s \ell$-enriched open and faithful functor, then $\widetilde{F}$ is an equivalence. 
\end{conjecture}

Theorems on the existence of fiber functors (i.e. functors for which the lifting is an equivalence) are also common to both Galois and Tannaka theories. In \cite{De}, 7, an internal characterization of tannakian categories is given, constructing under some hypotheses (see \cite{De}, 7.1) a fiber functor (see \cite{De}, 7.18). In \cite{JT} VII. 3, for any  Grothendieck topos $\cc{E}$ its spatial cover is constructed, which we have showed that can be considered as a fiber functor. Since the Diaconescu cover of $\cc{E}$ is also an open surjection, it can also be considered as a fiber functor.   

We think it would also be worth researching which conditions should be satisfied by a $sl$-enriched category $\cc{A}$ so that there is an algebra $B$ and a fiber functor (i.e. a $s \ell$-enriched open and faithful functor if conjecture \ref{conjetura} holds) $\cc{A} \mr{} (B$-$Mod)_0$. Such a result would be analogous to the one of \cite{De} for the case of sup-lattices, and (if the conditions are weaker than those that make $\cc{A}$ the category of relations of a Grothendieck topos) would generalize the construction of the Diaconescu cover mentioned above.

\end{sinnadastandard}

\pagebreak

\begin{appendices}
 
\section{Neutral Tannaka theory} \label{appendix}

{\bf The Hopf algebra of automorphisms of a $\cc{V}$-functor.} 

(For details see for example \cite{SP}, \cite{S}). Let $\Vat$ be a cocomplete monoidal closed category with tensor product $\otimes$, unit object $I$ and internal hom-functor $hom$. By definition for every object $V \in \cc{V}$,  $hom(V, -)$ is right adjoint to $(-) \otimes V$. That is, for every $X,\;Y$, $hom(X \otimes V,\, Y) = hom(X,\, hom(V,\, Y))$.

A pairing between two objects $V$, $W$ is a pair of arrows $W \otimes V \mr{\varepsilon} I$ and $I \mr{\eta} V \otimes W$ satisfying the usual triangular equations. We say that $W$ is the \emph{left} dual of $V$, and denote $W = V^\vee$, and that $V$ is \emph{right} dual of $W$ and denote $V = W^\wedge$. When $X$ has a left dual, then $X^\vee = hom(X,\, I)$.

The following are basic equations:

If $X$ has a right dual: \hfill 
$Y$ has a left dual $\iff$ 
$hom(Y,\, X)^\wedge = Y \otimes X^\wedge$,

\hfill $X = X^{\wedge^{\scriptstyle\vee}}$,
$hom(X^\wedge,\,Y) = Y \otimes X$.   

\vspace{1ex}

If $X$ has a left dual:  \hfill
$X = X^{\vee^{\scriptstyle\wedge}}$,
$hom(X,\,Y) = Y \otimes X^\vee$.

\vspace{1ex}

Recall that the object of natural transformations between $\cc{V}$-valued functors $L,\,T: \cc{X} \to \cc{V}$, is given, if it exists, by the following end 
\begin{equation}
Nat(L,\,T) = \displaystyle\int_X hom(LX,TX)\,.
\end{equation}

We consider a pair $(\cc{V}_0,\, \cc{V})$, where $\cc{V}_0 \subset \cc{V}$ is a full subcategory such that all its objects have a right dual. 

Let $\cc{X}$ be a $\cc{V}$-category such that for any two functors  
$\cc{X} \mr{L} \cc{V}$ and $\cc{X} \mr{T} \cc{V}_0$ the coend in the following definition exists in $\cc{V}$ (for example, if $\cc{X}$ is small). Then, we define (in Joyal's terminology) the \emph{Nat predual} as follows:
\begin{equation} \label{predual}
Nat^\lor(L,T) = \int^X LX \otimes (TX)^\wedge = 
                                         \int^X hom(LX,\, TX)^\wedge\,.
\end{equation} 

However, the last expression is valid only if $LX$ has a left dual for every $X$ (for example, if $\cc{X} \mr{L} \cc{V}_0$ and every object in 
$\cc{V}_0$ also has a left dual).

\vspace{.2cm}

\noindent Given $V \in \cc{V}$, recall that there is a functor 
$\cc{X} \mr{V \otimes T} \cc{V}$ defined by \mbox{$(V \otimes T)(X) = V \otimes TX$.}

\begin{proposition} \label{predualprop}
Given $T \in {\Vat_0}^{\cc{X}}$, we have a $\cc{V}$-adjunction $$\xymatrix { {\Vat}^{\cc{X}} \ar@/^/[r]^{Nat^{\lor}(-,T)}_\bot & \Vat \ar@/^/[l]^{(-) \otimes T} }.$$ 
\end{proposition}

\vspace{-3ex}

\begin{proof} ${}$

\noindent
$hom(Nat^\vee(L,\,T), V) 
= hom(\displaystyle\int^X LX \otimes TX^\wedge,\, V) 
= \displaystyle\int_X hom(LX \otimes TX^{\wedge}, \, V)$ 

\noindent \mbox{$ = \displaystyle\int_X hom(LX,\, hom(TX^\wedge, \, V) 
= \displaystyle\int_X hom(LX,\, V \otimes TX) 
= Nat(L,\, V \otimes T)$.}
\end{proof}

\noindent In particular we have that the end $Nat(L,\, T)$ exists and \mbox{$Nat(L,\, T) = hom(Nat^\lor(L,\,T), \,I)$.}
It follows that $Nat^\vee(L,\,T)$ classifies natural transformations $L \implies T$ in the sense that they correspond to arrows $Nat^\vee(L,\,T) \mr{} I$ in $\cc{V}$. This does not mean that $Nat(L,\, T)$ is the left dual of $Nat^\vee(L,\,T)$,  
which in general will not have a left dual. 
\emph{Passing from $Nat^\vee(L,\,T)$ to $Nat(L,\, T)$ looses information.}

\vspace{1ex}

The unit of the adjunction $L \Mr{\eta} Nat^\lor (L,\,T) \otimes T$ is a \mbox{coevaluation,} and if $\cc{X} \mr{H} \Vat_0$, it induces (in the usual manner) a \mbox{cocomposition} 
$$Nat^\lor (L,\,H) \mr{w} Nat^\lor(L,\,T) \otimes Nat^\lor (T,\,H).$$ There is a counit 
$Nat^\lor (T,\,T) \mr{\varepsilon} I$ determined by the arrows $TC \otimes TC^\lor \mr{\varepsilon} I$ of the duality.
All the preceding means exactly that the functors $\cc{X} \mr{} \cc{V}_0$ are the objects of a $\Vat$-cocategory.

\vspace{1ex}

We define $End^\lor(T) = Nat^\lor (T,\,T)$, which is therefore a coalgebra in $\Vat$. The coevaluation in this case becomes a $End^\lor(T)$-comodule structure \mbox{$TC\mr{\eta_C} End^\lor (T) \otimes TC$} on $TC$. In this way there is a lifting of the functor $T$ into $Cmd_0(H)$, $\cc{X} \mr{\tilde{T}} Cmd_0(H)$, for $H = End^\lor(T)$, and $Cmd_0(H)$ the full subcategory of comodules with underlying object in $\cc{V}_0$.

\vspace{1ex}

\begin{proposition}\label{neutralbialg}
 If $\Xat$ and $T$ are monoidal, and $\Vat$ has a symmetry, then $End^\lor(T)$ is a bialgebra. If in addition $\Xat$ has a symmetry and $T$ respects it, $End^\lor(T)$ is commutative (as an algebra). \cqd
\end{proposition}

We will not prove this proposition here, but show how the multiplication and the unit are constructed, since they are used explicitly in \ref{neutraltannakacontext}.
The multiplication $$End^\lor(T) \otimes End^\lor(T) \mr{m} End^\lor(T)$$ is induced by the composites
$$
m_{X,Y}: TX \otimes TX^\wedge \otimes TY \otimes TY^\wedge \mr{\cong} T(X \otimes Y) \otimes T(X \otimes Y)^\wedge \mr{\lambda_{X \otimes Y}} End^\lor(T).
$$
The unit is given by the composition
$$
u: I \rightarrow I \otimes I^\wedge \mr{\cong} T(I) \otimes T(I)^\wedge \mr{\lambda_{I}} End^\lor(T).
$$
\begin{proposition}\label{neutralhopf}
If in addition to the hypothesis of \ref{neutralbialg} every object of $\Xat$ has a right dual, then $End^\lor(T)$ is a Hopf algebra. \cqd
\end{proposition}

The antipode $End^\lor(T) \mr{\iota} End^\lor(T)$ is induced by the composites
$$
\iota_X: TX \otimes TX^\wedge  \mr{\cong} T(X^\wedge) \otimes TX \mr{\lambda_{X^\wedge}} End^\lor(T).
$$

\pagebreak

\section{Elevators calculus} \label{ascensores}

This is a graphic notation\footnote{Invented by Dubuc in 1969 (which has remained for private draft use for understandable typographical reasons).} 
to write equations in a monoidal category $\cc{V}$, ignoring \linebreak associativity and suppressing the tensor  symbol $\otimes$ and the neutral object $I$. Arrows are 
written as cells, the identity arrow as a double line, and the symmetry as crossed double lines. The notation, in particular, exhibits clearly the permutation associated 
to a composite of symmetries, allowing to see if any two composites are the same simply by checking that they codify the same permutation\footnote
         { This is justified by a simple coherence theorem for symmetrical categories (\cite{S} Proposition 2.3), particular case of \cite{JS2} Corollary 2.2 for braided categories.
         }. Compositions are read from top to bottom. 

\begin{sinnadastandard} \label{ascensoresconB}
 Given an algebra $B$ in the monoidal category $\cc{V}$, we specify with a $\xymatrix@C=0pc{{\ar@{}[rr]|*+<.6ex>[o][F]{\scriptscriptstyle{B}}} && }$ the tensor product $\otimes_B$ over $B$, and leave the tensor product $\otimes$ of $\cc{V}$ unwritten.
\end{sinnadastandard}         

Given arrows $C \mr{f} D$, $C' \mr{f'} D'$, the bifunctoriality of the tensor product is the basic equality:
\begin{equation} \label{ascensor}
\xymatrix@C=0ex
         {
             C \dcell{f} & C' \did
          \\
             D \did & C' \dcell{f'}
          \\
             D  &  D' 
         }
\xymatrix@R=6ex{\\ \;\;\;=\;\;\; \\}
\xymatrix@C=0ex
         {
             C \did & C'\dcell{f'}
          \\
             C \dcell{f} & D' \did
          \\
             D & D' 
         }
\xymatrix@R=6ex{ \\ \;\;\;=\;\;\; \\}
\xymatrix@C=0ex@R=0.9ex
         {
             {} & {}
          \\
               C   \ar@<4pt>@{-}'+<0pt,-6pt>[ddd] 
                   \ar@<-4pt>@{-}'+<0pt,-6pt>[ddd]^{f}
             & C'  \ar@<4pt>@{-}'+<0pt,-6pt>[ddd] 
                   \ar@<-4pt>@{-}'+<0pt,-6pt>[ddd]^{f'}
          \\ 
             {} & {}
          \\ 
             {} & {}
          \\
             D & D'.
         }
\end{equation}
This allows to move cells up and down when there are no obstacles, as if they were elevators. 
The naturality of the symmetry is the basic equality:
\begin{equation} \label{swap}
\xymatrix@C=0ex
         {
             C \dcell{f} & C' \did
          \\
             D \did & C' \dcell{f'}
          \\
             D \ar@{=} [dr] & D' \ar@{=} [dl]
          \\
             D' & D 
         }
\xymatrix@R=10ex{ \\ \;\;\;=\;\;\; \\}
\xymatrix@C=0ex
         {
             C \dcell{f} & C' \did
          \\
             D \ar@{=} [dr] & C' \ar@{=} [dl]
          \\
             C' \dcell{f'} & D \did
          \\
             D' & D 
         }
\xymatrix@R=10ex{ \\ \;\;\;=\;\;\; \\}
\xymatrix@C=0ex
         {
             C \ar@{=} [dr] & C' \ar@{=} [dl]
          \\
             C' \did & C \dcell{f}
          \\
             C' \dcell{f'} & D \did
          \\
             D' & D.
         }
\end{equation}
Cells going up or down pass through symmetries by changing the column.  

\vspace{1ex}

\emph{Combining the basic moves \eqref{ascensor} and \eqref{swap} we form configurations of cells that fit valid equations in order to prove new equations.} 
The visual aspect of this calculus really helps to find how a given equation can (or cannot) be derived from another ones.

\section{Non-neutral Tannaka theory} \label{sec:Tannaka} 

In this section we make the constructions needed to develop a Non-neutral Tannaka theory (as in \cite{De}), over a general tensor category ($\Vat, \otimes, k)$. Let $B', B \in Alg_\Vat$, \linebreak $M \in B$-Mod, $N \in$ Mod-$B$, then we have $M \otimes N \in B$-Mod-$B$, $N \otimes_B M \in \Vat = k$-Mod. Consider the coequalizer $N \otimes M \stackrel{c}{\rightarrow} N \otimes_B M$.

\begin{\prop} \label{subec}
$M \mr{f} M'$ in $B$-Mod, $N \mr{g} N'$ in Mod-$B$, then 
$$\vcenter{\xymatrix{N \otimes M \ar[r]^{g \otimes f} \ar[d]^c & N' \otimes M' \ar[d]^c \\ N \stackrel[B]{\textcolor{white}{B}}{\otimes} M \ar[r]^{g \stackrel[B]{}{\otimes} f} & N' \stackrel[B]{\textcolor{white}{B}}{\otimes} M'}}
, \hbox{ i.e. (recall \ref{ascensoresconB}) } \quad
\vcenter{\xymatrix@C=-0.3pc@R=1pc{N \dcell{g} && M \dcell{f} \\ N' \dl & \dc{c} & \dr M' \\ N' \Brr & \,\,\, & M'}} 
=
\vcenter{\xymatrix@C=-0.3pc@R=1pc{N \dl & \dc{c} & \dr M \\ N \dcell{g} \Brr & \,\,\, & M \dcell{f} \\ N' \Brr & \,\,\,& M'}}$$

\cqd
\end{\prop}

\begin{\prop} \label{isomB}
The isomorphism $M \cong B \otimes_B M$ is given by $$\vcenter{\xymatrix@C=-0.3pc@R=1pc{& M \op{\cong} \\ B \Brr & \,\,\, & M}} = \vcenter{\xymatrix@C=-0.3pc@R=1pc{ \du{u} && M \did \\ B \dl & \dc{c} & \dr M \\ B \Brr & \,\,\, & M}}$$
\cqd
\end{\prop}

\subsection{Duality of modules}

\begin{\de}\label{moddual} Let $M \in B$-Mod. We say that $M$ has a right dual (as a $B$-module) if there exists $M^\wedge \in$ Mod-$B$, $M \otimes M^\wedge \stackrel{\eps}{\rightarrow} B$ morphism of $B$-Mod-$B$ and $k \stackrel{\eta}{\rightarrow} M^\wedge \otimes_B M$ morphism of $\Vat$ such that the triangular equations

\begin{equation} \label{triangular}
\xymatrix@C=-0.3pc@R=0.1pc
         { 			
          &  \ar@{-}[ldd] \ar@{-}[rdd] \ar@{}[dd]|{\eta} 
          & & &  M^{\!\wedge} \ar@2{-}[dd] 					
          & & & & & & & & &  M \ar@2{-}[dd] 
          & & & \ar@{-}[ldd] \ar@{-}[rdd] \ar@{}[dd]|{\eta} 
          & & & &  
          \\
		  & & & & & & & M^{\!\wedge} \ar@2{-}[dd] 
		  & & & & & & & & & & & & &  M \ar@2{-}[dd] 
		  \\
		    M^{\!\wedge} \ar@2{-}[dd] \ar@{}[rr]|*+<.6ex>[o][F]{\scriptscriptstyle{B}}
		  & \,\,\, & M \ar@{-}[rdd] 
		  & \ar@{}[dd]|{\varepsilon} 
		  & M^{\!\wedge} \ar@{-}[ldd] 
		  & \quad = \quad 
		  & & & & & \quad \quad \hbox{and} \quad\quad 
		  & & & M \ar@{-}[ddr] 
		  & \ar@{}[dd]|{\varepsilon} 
		  & M^{\!\wedge} \ar@{-}[ldd] \ar@{}[rr]|*+<.6ex>[o][F]{\scriptscriptstyle{B}} 
		  & \,\,\, & M \ar@2{-}[dd] 
		  & \quad = \quad 
		  & &  
		  \\
		  & & & & & & & M^{\!\wedge}  & & & & & & & & & & & & &  M. 
		  \\
			M^{\!\wedge} \ar@{}[rrr]|*+<.6ex>[o][F]{\scriptscriptstyle{B}}
		  & & & B & & & & & & & & & & & B \ar@{}[rrr]|*+<.6ex>[o][F]{\scriptscriptstyle{B}} & & & M	
		 }
\end{equation}
hold. In this case, we say that $M^\wedge$ is \emph{the} right dual of $M$ and we denote $M \dashv M^\wedge$.
\end{\de}

\begin{\prop}
A duality $M \dashv M^\wedge$ yields an adjunction $$\xymatrix { B'\hbox{-Mod} \ar@/^2ex/[rr]^{(-)\otimes M^\wedge} \ar@{}[rr]|{\bot} && B'\hbox{-Mod-}B \ar@/^2ex/[ll]^{(-) \stackrel[B]{}{\otimes}  M} }$$
given by the binatural bijection between morphisms

\begin{equation}\label{adjunciondeldual}
\begin{tabular}{c}
 $N \otimes M^\wedge \stackrel{\lambda}{\rightarrow} L$ of $B'$-$Mod$-$B$ \\ \hline \noalign{\smallskip}
 $N \stackrel{\rho}{\rightarrow} L \stackrel[B]{}{\otimes} M$ of $B'$-$Mod$
\end{tabular}
\end{equation}

for each $N \in B'$-Mod, $L \in B'$-Mod-$B$.
\end{\prop}

\begin{proof}
The bijection is given by
\begin{equation} \label{lambdarhoconascensores} \xymatrix@C=-0.3pc@R=1pc{&& N \op{\rho} &&& M^{\!\wedge} \did \\
						\lambda: \quad & L \did \ar@{}[rr]|*+<.6ex>[o][F]{\scriptscriptstyle{B}} && M && M^{\!\wedge} & \quad , \quad \quad \\
						& L \ar@{}[rrr]|*+<.6ex>[o][F]{\scriptscriptstyle{B}} &&& B \cl{\eps} }
\xymatrix@C=-0.3pc@R=1pc{& N \did &&& \op{\eta} \\
\rho: \quad & N && M^{\!\wedge} \ar@{}[rr]|*+<.6ex>[o][F]{\scriptscriptstyle{B}} & \,\,\, & M \did \\
& & L \cl{\lambda} \ar@{}[rrr]|*+<.6ex>[o][F]{\scriptscriptstyle{B}} &&& M.} \end{equation}
All the verifications are straightforward.
\end{proof}

\begin{remark} \label{homesdual}
Considering $B'=k$, we see that $M^\wedge$, if it exists, is unique (except for a canonical isomorphism), since it can be retrieved from the functor $(-) \otimes M^\wedge$. $\eta$ and $\eps$ can also be retrieved from the unit and counit of the adjunction, therefore are also unique. This can as well be proved explicitly by computing the isomorphism between two right duals of $M$.

Also, from the adjunction we see that if $M^\wedge$ exists, then $(-) \otimes_B M$ and $Hom_B(M^\wedge,-)$ are right adjoints of $(-) \otimes M^\wedge$, so we have $Hom_B(M^\wedge,L) = L \otimes_B M^\wedge$ for each \linebreak $L \in B'$-Mod-$B$.
\end{remark}

\begin{\de} \label{catconduales}
We will denote by $(B$-Mod$)_r$ the full subcategory of $B$-Mod consisting of those modules that have a right dual.
\end{\de}

\begin{\prop} \label{dualesfuntor}
There is a contravariant functor $(-)^\wedge: (B$-Mod$)_r \rightarrow $Mod-$B$ defined on the arrows $M \mr{f} N$ as
$$\xymatrix@C=-0.3pc@R=1pc{&&\op{\eta} &&  & N^{\!\wedge} \did \\
f^\wedge: \quad & M^{\!\wedge} \did \Brr & \,\,\, & M \dcell{f} && N^{\!\wedge} \did \\
&M^{\!\wedge} \did \Brr && N && N^{\!\wedge} \\
&M^{\!\wedge} \Brrr &&& B \cl{\eps}}$$
\end{\prop}
\begin{proof}
It is straightforward.
\end{proof}

\begin{sinnadastandard} \label{tensorprodofmodules} The case where $B$ is commutative.

Assume that $\Vat$ is symmetric and $B$ is a \emph{commutative} algebra in $\Vat$. Then there are obvious isomorphisms of categories \mbox{$B$-Mod $\cong$ Mod-$B \cong B \! = \! \! Mod$}, where the last category is defined as the full subcategory of $B$-Mod-$B$ consisting of those $B$-bimodules such that left and right multiplication coincide. 
The tensor product $\otimes_B$ of $B$-bimodules restricts to this category and in this way $B$-Mod is a tensor category with tensor product $\otimes_B$ and neutral element $B$. The known concept of dual in a tensor category yields in this case the following definition that we will compare with definition \ref{moddual}.

\begin{\de} \label{dualcomobimod}
Let $B$ be a \emph{commutative} algebra in $\Vat$, let $M \in B$-Mod. We say that $M$ has a right dual (as a $B \! =$bimodule) if there exists $M^\wedge \in$ Mod-$B$, $M \otimes_B M^\wedge \stackrel{\eps'}{\rightarrow} B$ and $B \stackrel{\eta'}{\rightarrow} M^\wedge \otimes_B M$ morphisms of $B \! =$bimodules such that the triangular equations

\begin{equation} \label{triangular2}
\xymatrix@C=-0.3pc@R=0.1pc
         { 			
          & B \ar@{-}[ldd] \ar@{-}[rdd] \ar@{}[dd]|{\eta'} \ar@{}[rrr]|*+<.6ex>[o][F]{\scriptscriptstyle{B}}
          & & &  M^{\!\wedge} \ar@2{-}[dd] 					
          & & & & & & & & &  M \ar@2{-}[dd] \ar@{}[rrr]|*+<.6ex>[o][F]{\scriptscriptstyle{B}}
          & & & B \ar@{-}[ldd] \ar@{-}[rdd] \ar@{}[dd]|{\eta'}
          & & & &  
          \\
		  & & & & & & & M^{\!\wedge} \ar@2{-}[dd] 
		  & & & & & & & & & & & & &  M \ar@2{-}[dd] 
		  \\
		    M^{\!\wedge} \ar@2{-}[dd] \ar@{}[rr]|*+<.6ex>[o][F]{\scriptscriptstyle{B}}
		  & \,\,\, & M \ar@{}[rr]|*+<.6ex>[o][F]{\scriptscriptstyle{B}} \ar@{-}[rdd] 
		  & \ar@{}[dd]|{\varepsilon'} 
		  & M^{\!\wedge} \ar@{-}[ldd] 
		  & \quad = \quad 
		  & & & & & \quad \quad \hbox{and} \quad\quad 
		  & & & M \ar@{}[rr]|*+<.6ex>[o][F]{\scriptscriptstyle{B}} \ar@{-}[ddr] 
		  & \ar@{}[dd]|{\varepsilon'} 
		  & M^{\!\wedge} \ar@{-}[ldd] \ar@{}[rr]|*+<.6ex>[o][F]{\scriptscriptstyle{B}} 
		  & \,\,\, & M \ar@2{-}[dd] 
		  & \quad = \quad 
		  & &  
		  \\
		  & & & & & & & M^{\!\wedge}  & & & & & & & & & & & & &  M. 
		  \\
			M^{\!\wedge} \ar@{}[rrr]|*+<.6ex>[o][F]{\scriptscriptstyle{B}}
		  & & & B & & & & & & & & & & & B \ar@{}[rrr]|*+<.6ex>[o][F]{\scriptscriptstyle{B}} & & & M	
		 }
\end{equation}
hold. In this case, we say that $M^\wedge$ is \emph{the} right dual of $M$ and we denote $M \dashv M^\wedge$.
\end{\de}

The last sentence of the definition does not introduce any ambiguity because of the following proposition.

\begin{\prop} \label{las2defdedualcoinciden}
Let $B$ be a \emph{commutative} algebra in $\Vat$, let $M \in B$-Mod. Then any right $B$-module, that we'll denote $M^\wedge$, is the right dual of $M$ as a $B$-module if and only if it is the right dual of $M$ as a $B \! =$bimodule. The arrows $\eta'$ and $\eps'$ are also in bijective correspondence with the arrows $\eta$ and $\eps$.
\end{\prop}

\begin{proof}
The correspondence is given by
$$\xymatrix@C=-0.3pc@R=1pc{& & \du{u} \\ \eta: \quad & & B \op{\eta'} && \quad , \quad \quad \\ & M^{\!\wedge} \Brr & \,\,\, & M} \xymatrix@C=-0.3pc@R=1pc{& M \dl & \dc{c} & \dr M^{\!\wedge} \\ \eps: \quad &   M \Brr & \,\,\, & M^{\!\wedge} \\ & & B \cl{\eps'}   }$$
$$\xymatrix@C=-0.3pc@R=1pc{& B \did &&& \op{\eta} \\
\eta': \quad & B \ar@{=}[drr] && M^{\!\wedge} \ar@{=}[dll] \Brr & \,\,\, & M \did & \quad , \quad \quad \\ 
& M^{\!\wedge} && B \Brr & \,\,\, & M \did \\
& &  M^{\!\wedge} \cl{m} \Brrr &&& M} \xymatrix@R=1pc{M \otimes M^\wedge \ar[dr]^\eps \ar[dd]_c \\ & B \\ M \otimes_B M^\wedge \ar@{.>}[ur]_{\exists ! \eps'}} $$

Let's verify the correspondence between the triangular identities, we only do the first one since the other one is symmetric (recall propositions \ref{subec} and \ref{isomB}).

$$\vcenter{\xymatrix@C=-0.3pc@R=1pc{ & \du{u} &&& M^{\!\wedge} \did \\
& B \op{\eta'} &&& M^{\!\wedge} \did \\
M^{\!\wedge} \did \Brr && M \dl & \dc{c} & \dr M^{\!\wedge} \\
M^{\!\wedge} \did \Brr && M \Brr && M^{\!\wedge} \\
M^{\!\wedge} \Brrr &&& B \cl{\eps'}}}
\vcenter{\xymatrix{=}}
\vcenter{\xymatrix@C=-0.3pc@R=1pc{ & \du{u} &&& M^{\!\wedge} \did \\
& B \op{\eta'} &&& M^{\!\wedge} \did \\
(M^{\!\wedge} \dl \Brr && M) \dc{c} && \dr M^{\!\wedge} \\
(M^{\!\wedge} \did \Brr && M) \Brr && M^{\!\wedge} \\
M^{\!\wedge} \Brrr &&& B \cl{\eps'}}}
\vcenter{\xymatrix{=}}
\vcenter{\xymatrix@C=-0.3pc@R=1pc{ & \du{u} &&& M^{\!\wedge} \did \\
& B \dl & \dcr{c} && M^{\!\wedge} \dr \\
& B \op{\eta'} \Brrr & && M^{\!\wedge} \did \\
(M^{\!\wedge} \did \Brr && M) \Brr && M^{\!\wedge} \\
M^{\!\wedge} \Brrr &&& B \cl{\eps'}}}
\vcenter{\xymatrix{=}}
\vcenter{\xymatrix@C=-0.3pc@R=1pc{ M^{\!\wedge} \did \\ M^{\!\wedge} }}$$

$$\vcenter{\xymatrix@C=-0.3pc@R=1pc{\, \\ &&&  M^{\!\wedge} \ar@{-}[dlll] \ar@{}[d]|\cong \ar@{-}[drrr] \\
B \did &&& \op{\eta} & \Brr &&  M^{\!\wedge} \did \\
B \ar@{=}[drr] &&  M^{\!\wedge} \ar@{=}[dll] \Brr & \,\,\, & M \did \Brr &&  M^{\!\wedge} \did \\
M^{\!\wedge} && B \Brr && M \did \Brr &&  M^{\!\wedge} \did \\
&  M^{\!\wedge} \cl{m} \did & \Brr && M \Brr &&  M^{\!\wedge} \\
&  M^{\!\wedge} & \Brr &&& B \cl{\eps'}}}
\vcenter{\xymatrix{=}}
\vcenter{\xymatrix@C=-0.3pc@R=1pc{ \du{u} &&&&&& M^{\!\wedge} \did \\
B \dl &&& \dc{c} &&& M^{\!\wedge} \dr \\
B \did &&& \op{\eta} & \Brr &&  M^{\!\wedge} \did \\
B \ar@{=}[drr] &&  M^{\!\wedge} \ar@{=}[dll] \Brr & \,\,\, & M \did \Brr &&  M^{\!\wedge} \did \\
M^{\!\wedge} && B \Brr && M \did \Brr &&  M^{\!\wedge} \did \\
&  M^{\!\wedge} \cl{m} \did & \Brr && M \Brr &&  M^{\!\wedge} \\
&  M^{\!\wedge} & \Brr &&& B \cl{\eps'}}}
\vcenter{\xymatrix{=}}$$

$$\vcenter{\xymatrix@C=-0.3pc@R=1pc{\du{u} &&&&&& M^{\!\wedge} \did \\
B \did &&& \op{\eta} &&& M^{\!\wedge} \did \\
B \ar@{=}[drr] && M^{\!\wedge} \ar@{=}[dll] \Brr & \,\,\, & M \did && M^{\!\wedge} \did \\
M^{\!\wedge} && B \Brr && M \did && M^{\!\wedge} \did \\
& (M^{\!\wedge} \cl{m} \dl & \Brr & \dcr{c} & M) && M^{\!\wedge} \dr \\
& (M^{\!\wedge} \did & \Brr && M) \Brr && M^{\!\wedge} \\
& M^{\!\wedge} & \Brr &&& B \cl{\eps'}}}
\vcenter{\xymatrix{=}}
\vcenter{\xymatrix@C=-0.3pc@R=1pc{ && \ar@{-}[drr] \ar@{-}[dll] \ar@{}[d]|{\eta} &&&& M^{\!\wedge} \did \\
M^{\!\wedge} \did && \du{u} \Brr & \,\,\, & M \did && M^{\!\wedge} \did \\
M^{\!\wedge}  && B \Brr && M \did && M^{\!\wedge} \did \\
& M^{\!\wedge} \cl{m} \did & \Brr && M \dl & \dc{c} & M^{\!\wedge} \dr \\
& M^{\!\wedge} \did & \Brr && M \Brr && M^{\!\wedge} \\
& M^{\!\wedge} & \Brr &&& B \cl{\eps'}}}
\vcenter{\xymatrix{=}}
\vcenter{\xymatrix@C=-0.3pc@R=1pc{ & \op{\eta} &&& M^{\!\wedge} \did \\
M^{\!\wedge} \did \Brr & \,\,\, & M && M^{\!\wedge} \\
M^{\!\wedge} \Brrr &&& B \cl{\eps}}}
\vcenter{\xymatrix{=}}
\vcenter{\xymatrix@C=-0.3pc@R=1pc{ M^{\!\wedge} \did \\ M^{\!\wedge} }}
$$
\end{proof}

\end{sinnadastandard}

\subsection{The $Nat^\lor$ adjunction} \label{sub:natpredual}

Consider now a category $\Cat$ and a functor $H: \Cat \rightarrow $Mod-$B$. We have an adjunction 
\begin{equation} \label{homadjunction}
\xymatrix { (B'\hbox{-Mod})^\Cat \ar@/^2ex/[rr]^{(-) \stackrel[\Cat]{}{\otimes} H} \ar@{}[rr]|{\bot} && B'\hbox{-Mod-}B \ar@/^2ex/[ll]^{Hom_B (H,-)} }
\end{equation}
where the functors are given by the formulae $$F \otimes_\Cat H = \int^{X \in \Cat} FX \otimes HX, \quad Hom_B(H,M)(C)=Hom_B(HC,M).$$

Assume now we have a full subcategory $(B$-Mod$)_0$ of $(B$-Mod$)_r$ (recall definition \ref{catconduales}), i.e. a full subcategory $(B$-Mod$)_0$ of $B$-$Mod$ such that every object has a right dual. Given $G: \Cat \rightarrow (B$-Mod$)_0$, using proposition \ref{dualesfuntor} we construct $G^\wedge: \Cat \rightarrow $Mod-$B$.

\begin{\de} \label{defnatpredual}
Given $G: \Cat \rightarrow (B$-Mod$)_0$, $F: \Cat \rightarrow B'$-Mod, we define $$Nat^\lor(F,G) = F \otimes_\Cat G^\wedge = \int^{X \in \Cat} FX \otimes GX^\wedge.$$
\end{\de}

\begin{\prop} \label{natpredualadjunction}
Given $G: \Cat \rightarrow (B$-Mod$)_0$, we have an adjunction
\begin{equation} 
\xymatrix { (B'\hbox{-Mod})^\Cat \ar@/^2ex/[rr]^{Nat^\lor((-),G)} \ar@{}[rr]|{\bot} && B'\hbox{-Mod-}B \ar@/^2ex/[ll]^{(-) \stackrel[B]{}{\otimes} G} }
\end{equation}
where the functor $(-) \otimes_B G$ is given by the formula $(M \otimes_B G) (C) = M \otimes_B (GC)$.
\end{\prop}
\begin{proof}
The value of the functor $Nat^\lor((-),G)$ in an arrow $F \stackrel{\theta}{\Rightarrow} H$ of $(B'\hbox{-Mod})^\Cat$ is the $B'$-$B$-bimodule morphism induced by
$$FX \otimes GX^\wedge \mr{\theta_X \otimes (GX)^\wedge} HX \otimes GX^\wedge \mr{\lambda_X} Nat^\lor(H,G).$$
The adjunction is given by the binatural bijections

\begin{center}
 \begin{tabular}{c}
  $Nat^\lor (F,G) \rightarrow C$ \\ \hline \noalign{\smallskip}
  $F \stackrel[\Cat]{}{\otimes} G^\wedge \rightarrow C$ \\ \hline \noalign{\smallskip}
  $F \Rightarrow Hom_B(G^\wedge, C)$ \\ \hline \noalign{\smallskip}
  $F \Rightarrow C \stackrel[B]{}{\otimes} G$ 
 \end{tabular}

\end{center}

justified by the adjunction \eqref{homadjunction} and the last part of the remark \ref{homesdual}. We leave the verifications to the reader.
\end{proof}

The unit of the adjunction is called the \emph{coevaluation} $F \Mr{\rho = \rho_F} Nat^\lor(F,G) \otimes_B G$. It can be checked that it is given by $$\rho_C: FC \mr{FC \otimes \eta} FC \otimes GC^\wedge \otimes_B GC \mr{\lambda_C \otimes GC} Nat^\lor(F,G) \otimes_B GC,$$
i.e. that it corresponds to $\lambda_C$ via the correspondence \eqref{lambdarhoconascensores}.

We also have the counit $Nat^\lor(L \otimes_B G, G) \mr{e = e_L} L$. It is induced by the arrows \linebreak $L \otimes_B GC \otimes GC^\wedge \mr{L \otimes_B \eps} L$.

\begin{center}
We now restrict to the case $B'=B$. 
\end{center}

\begin{\de} \label{defdeL}
Given $F: \Cat \rightarrow (B$-$Mod)_0$ , we define $$L = L(F) = End^\lor(F) = Nat^\lor(F,F).$$
\end{\de}

\begin{remark} \label{ecuaciondeL}
$L$ is universal among those $B$-bimodules satisfying the equation
$$\vcenter{\xymatrix@C=-0.3pc@R=1pc{ FC \dcellbb{F(f)} & & FD^\wedge \did 	\\
											FD  &  & FD^\wedge 															\\
											& L \cl{\lambda_{D}} &}}
\vcenter{\xymatrix{=}}
\vcenter{\xymatrix@C=-0.3pc@R=1pc{
	FC \did & & & \op{\eta} & & & FD^\wedge \did \\
	FC \did & & FC^\wedge \did \Brr & \,\,\, & FC \dcellbb{F(f)} & & FD^\wedge \did \\
	FC \did & & FC^\wedge \did \Brr & \,\,\, & FD & & FD^\wedge \\
	FC && FC^\wedge \Brrr & & & B \cl{\eps}   \\		
	& L \cl{\lambda_C} }	}$$
	for each $C \mr{f} D$ in $\Cat$.
\end{remark}

As usual, given $F,G,H: \Cat \rightarrow (B$-Mod$)_0$ we construct from the coevaluation a \emph{cocomposition} $$Nat^\lor(F,H) \stackrel{c}{\rightarrow} Nat^\lor(F,G) \stackrel[B]{\textcolor{white}{B}}{\otimes} Nat^\lor(G,H)$$
This is a $B$-bimodule morphism induced by the arrows $$FC \otimes HC^\wedge \mr{FC \otimes \eta \otimes HC^\wedge} FC \otimes GC^\wedge \stackrel[B]{\textcolor{white}{B}}{\otimes} GC \otimes HC^\wedge \mr{\lambda_C \otimes \lambda_C} Nat^\lor(F,G) \stackrel[B]{\textcolor{white}{B}}{\otimes} Nat^\lor(G,H)$$

The structure given by $c$ and $e$ is that of a \emph{cocategory enriched over} $B$-Bimod. Therefore, $L = L(F)$ is a coalgebra in the monoidal category $B$-Bimod, i.e. a $B$-bimodule with a coassociative comultiplication $L \mr{c} L \otimes_B L$ and a counit $L \mr{e} B$. This is called a \emph{cog\'ebro\"ide agissant sur B} in \cite{De}. Cog\'ebro\"ides act on $B$-modules as follows

\begin{\de} \label{defcmd}
Let $L$ be a \emph{cog\'ebro\"ide agissant sur B}, i.e. a coalgebra in $B$-Bimod. A (left) representation of $L$, which we will also call a (left) $L$-comodule, is a $B$-module $M$ together with a coaction, or comodule structure $M \mr{\rho} L \otimes_B M$, which is a morphism of $B$-modules such that
$$C1: \vcenter{\xymatrix@C=-0.3pc@R=1pc{&&& & M \ar@{-}[d] \ar@{-}[dlll] \ar@{}[dll]|{\ \ \rho} \\
																		& L \op{c} & \Brr & \,\,\, & M \did \\
																		L \Brr && L \Brr && M}}
\vcenter{\xymatrix{=}}
\vcenter{\xymatrix@C=-0.3pc@R=1pc{&&&& M \ar@{-}[d] \ar@{-}[dllll] \ar@{}[dll]|{\rho} \\
																	L \did & \Brr &&& M \drho{\rho} & \quad \quad \quad \quad \\
																		L \Brr & \,\,\, & L \Brr & \,\,\, & M}}
C2: \vcenter{\xymatrix@C=-0.3pc@R=1pc{&& M \drho{\rho} \\ L \dcell{e} \Brr & \,\,\, & M \did \\ B \Brr && M}}
\vcenter{\xymatrix{=}}
\vcenter{\xymatrix@C=-0.3pc@R=1pc{M \did \\ M}}$$
We define in an obvious way the comodule morphisms, and we have that way a category $Cmd(L)$. We denote by $Cmd_0(L)$ the full subcategory of those comodules whose subjacent $B$-module is in $(B$-Mod$)_0$.
\end{\de}

\begin{\prop} \label{lifting}
Given $F: \Cat \rightarrow (B$-Mod$)_0$, the unit $FC \mr{\rho_C} L \otimes_B FC$ yields a comodule structure for each $FC$. Then we obtain a lifting of the functor $F$ as follows
$$\xymatrix{\Cat \ar@{.>}[r]^{\tilde{F}} \ar[d]_F & Cmd_0(L) \ar[dl]^U \\ (B\hbox{-Mod})_0}$$
\end{\prop}

\begin{proof}
Since we know explicitly what $\rho$, $e$ and $c$ are, it is easy to check both equations on definition \ref{defcmd}. Both sides of the first equation are equal to the composition
$$\vcenter{\xymatrix@C=-0.3pc@R=1pc{FC \did &&& \op{\eta} &&&& \op{\eta} \\
																		FC && FC^{\!\wedge} \Brr & \,\,\, & FC && FC^{\!\wedge} \Brr & \,\,\, & FC \did \\
																		& L \cl{\lambda_C} & \Brr &&& L \cl{\lambda_C} \Brrr &&& FC,  }}$$
and the second equation is just a triangular equation for $FC \dashv FC^\wedge$.

We now verify that given an arrow $C \mr{f} D$ in $\Cat$, $F(f)$ is a comodule morphism (recall remark \ref{ecuaciondeL})

$$\vcenter{\xymatrix@C=-0.3pc@R=1pc{FC \dcellbb{F(f)} &&& \op{\eta} \\
																		FC && FD^\wedge \Brr & \,\,\, & FD \did \\
																		& L \cl{\lambda_D} &&& FD}}
\vcenter{\xymatrix{=}}
\vcenter{\xymatrix@C=-0.3pc@R=1pc{FC \did &&& \op{\eta} &&&& \op{\eta} \\
																	FC && FC^{\!\wedge} \Brr & \,\,\, & FC \dcellbb{F(f)} && FD^{\!\wedge} \did \Brr & \,\,\, & FD \did \\
																	& L \cl{\lambda_C} \did \Brrr &&& FD && FD^{\!\wedge} \Brr & \,\,\, & FD \did \\
																	& L & \Brr &&& B \cl{\eps} \Brrr &&& FD}}
\vcenter{\xymatrix{=}}
\vcenter{\xymatrix@C=-0.3pc@R=1pc{FC \did &&& \op{\eta} \\
																	FC && FC^{\!\wedge} \Brr & \,\,\, & FC \dcellbb{F(f)} \\
																	& L \cl{\lambda_C} \Brrr &&& FD}}$$
\end{proof}

\begin{lemma} \label{lemadeunicidad}
Let $M \in (B$-$Mod)_r$, $L \in B$-$Bimod$, $M \otimes M^\wedge \mr{\lambda} L$ in $B$-$Bimod$, and $\rho$ the corresponding $B$-module morphism via \eqref{lambdarhoconascensores}. Let $L \mr{e} B$, $L \mr{c} L \otimes_B L$ be a structure of cog\'ebro\"ide sur $B$. Then $\rho$ is a comodule structure for $M$ if and only if the following diagrams commute:
$$B1: \vcenter{\xymatrix{ M \otimes M^\wedge \ar[r]^{\lambda} \ar[d]_{M \otimes \eta \otimes M^\wedge } & L \ar[d]^{c} \\
M \otimes M^\wedge \stackrel[B]{\textcolor{white}{B}}{\otimes} M \otimes M^\wedge  \ar[r]^>>>>>{\lambda \stackrel[B]{}{\otimes} \lambda} & L \stackrel[B]{\textcolor{white}{B}}{\otimes} L}}
\quad \quad \quad B2: \vcenter{\xymatrix{M \otimes M^\wedge \ar[r]^{\lambda} \ar[d]_{\eps} & L \ar[dl]^{e} & \quad \quad \quad \quad \\ B}}$$
\end{lemma}

\begin{proof}
We can prove $B1 \iff C1$, $B2 \iff C2$. All the implications can be proved in a similar manner when using a graphical calculus, we show $C1 \implies B1$:

$$\vcenter{\xymatrix@C=-0.3pc@R=1pc{&& M \did &&&& \ar@{-}[dll] \ar@{}[d]|{\eta} \ar@{-}[drr] &&&& M^\wedge \did \\
				   && M \ar@{-}[d] \ar@{-}[dll] \ar@{}[dl]|{\rho} && M^\wedge \did \Brr &&&& M \ar@{-}[d] \ar@{-}[dll] \ar@{}[dl]|{\rho} && M^\wedge \did \\
				   L \did \Brr & \quad & M && M^\wedge \Brr & \quad & L \did \Brr & \quad & M && M^\wedge \\
				   L \Brr &&& B \cl{\eps} & \Brr && L \Brr &&& B \cl{\eps}}}
\vcenter{\xymatrix{\stackrel{\triangle}{=}}}
\vcenter{\xymatrix@C=-0.3pc@R=1pc{&&&& M \ar@{-}[d] \ar@{-}[dllll] \ar@{}[dll]|{\ \ \rho} && M^\wedge \did \\
	    L \did & \Brr &&& M \ar@{-}[d] \ar@{-}[dll] \ar@{}[dl]|{\rho} && M^\wedge \did \\
	    L \did \Brr & \quad & L \did \Brr & \quad & M && M^\wedge \\
	    L \Brr && L \Brrr &&& B \cl{\eps}}}
    \vcenter{\xymatrix{\stackrel{C1}{=}}}
	    \vcenter{\xymatrix@C=-0.3pc@R=1pc{&&& & M \ar@{-}[d] \ar@{-}[dlll] \ar@{}[dll]|{\ \ \rho} && M^\wedge \did \\
	    & L \op{c} & \Brr & \,\,\, & M && M^\wedge \\
	    L \Brr && L \Brrr &&& B \cl{\eps}}}$$

\end{proof}

\begin{remark} \label{remarkdeunicidad}
The previous lemma implies that $L \mr{e} B$, $L \mr{c} L \otimes_B L$ as defined before is the only possible cog\`ebro\"ide structure for $L$ that make each $\rho_X$ a comodule structure.
\end{remark}

We now give $L$ additional structure under some extra hypothesis (cf. propositions \ref{neutralbialg}, \ref{neutralhopf}) 

\begin{proposition}\label{bialg}
 If $\Cat$ and $F$ are monoidal, and $\Vat$ has a symmetry, then $L$ is a \linebreak $B \otimes B$-algebra. If in addition $\Cat$ has a symmetry and $F$ respects it, $L$ is commutative (as an algebra). \cqd 
\end{proposition}

We will not prove this proposition here, but show how the multiplication and the unit are constructed, since they are used explicitly in \ref{tannakacontext}.
The multiplication \mbox{$L\stackrel[B \otimes B]{\textcolor{white}{B}}{\otimes}L \mr{m} L$} is induced by the composites
$$
m_{X,Y}: (FX \otimes FX^\wedge)\stackrel[B \otimes B]{\textcolor{white}{B}}{\otimes}(FY \otimes FY^\wedge) \mr{\cong} (FX \stackrel[B]{}{\otimes} FY) \otimes (FY^\wedge \stackrel[B]{}{\otimes} FX^\wedge)$$

\hfill $\mr{\cong} F(X \otimes Y) \otimes F(X \otimes Y)^\wedge \mr{\lambda_{X \otimes Y}} L.$

The unit is given by the composition
$$
u: B \otimes B \mr{\cong} F(I) \otimes F(I)^\wedge \mr{\lambda_{I}} L.
$$
  
\begin{proposition}\label{hopf}
 If in addition $\Cat$ has a duality, then $L$ has an antipode. \cqd 
\end{proposition}
  
The antipode $L \mr{a} L$ is induced by the composites
$$
a_X: FX \otimes FX^\wedge  \mr{\cong} F(X^\wedge) \otimes FX \mr{\lambda_{X^\wedge}} L.
$$

\end{appendices}

\pagebreak

\bibliographystyle{model1-num-names}
\bibliography{<your-bib-database>}

\end{document}